%% file: main.tex
\documentclass[11pt,reqno]{amsart}
\usepackage[T1]{fontenc}
\usepackage{lmodern}
\usepackage{amsmath,amssymb,amsthm,mathtools,mathrsfs}
\usepackage[margin=1.08in]{geometry}
\usepackage{microtype}
\usepackage{enumitem,booktabs,array,longtable}
\usepackage[hidelinks]{hyperref}
\usepackage{url}
\usepackage{tikz}
\usetikzlibrary{calc}

\allowdisplaybreaks[2]
\numberwithin{equation}{section}
\setlist[enumerate]{leftmargin=2.3em}
\newtheorem{theorem}{Theorem}[section]
\newtheorem{proposition}[theorem]{Proposition}
\newtheorem{lemma}[theorem]{Lemma}
\newtheorem{corollary}[theorem]{Corollary}
\theoremstyle{definition}
\newtheorem{definition}[theorem]{Definition}
\theoremstyle{remark}
\newtheorem{remark}[theorem]{Remark}
\newcommand{\C}{\mathbb C}
\newcommand{\R}{\mathbb R}
\newcommand{\cS}{\mathcal S}
\newcommand{\cA}{\mathcal A}
\newcommand{\fg}{\mathfrak g}
\newcommand{\fh}{\mathfrak h}
\newcommand{\fp}{\mathfrak p}
\newcommand{\GL}{\operatorname{GL}}
\newcommand{\SO}{\operatorname{SO}}
\newcommand{\Or}{\operatorname{O}}
\newcommand{\Sp}{\operatorname{Sp}}
\newcommand{\Spin}{\operatorname{Spin}}
\newcommand{\Hom}{\operatorname{Hom}}
\newcommand{\End}{\operatorname{End}}
\newcommand{\Lie}{\operatorname{Lie}}
\newcommand{\Ad}{\operatorname{Ad}}
\newcommand{\tr}{\operatorname{tr}}
\newcommand{\str}{\operatorname{str}}
\newcommand{\diag}{\operatorname{diag}}
\newcommand{\Sym}{\operatorname{Sym}}
\newcommand{\supp}{\operatorname{supp}}
\newcommand{\rk}{\operatorname{rank}}
\newcommand{\one}{\mathbf 1}
\newcommand{\sslash}{/\!/}
\title[Multiplicity One for Connected Complex Symmetric Spaces]{Multiplicity One and van Dijk's Conjecture for Connected Complex Symmetric Spaces}
\author{Yuxuan Ge}

\address{School of Mathematics, University of Minnesota, Minneapolis, MN 55455, USA}
\email{ge005678@umn.edu}
\hypersetup{
  pdftitle={Multiplicity One and van Dijk's Conjecture for Connected Complex Symmetric Spaces},
  pdfauthor={Yuxuan Ge},
  pdfsubject={Regularity and multiplicity one for complex symmetric pairs},
  pdfkeywords={complex symmetric pairs, regularity, invariant theory, moment ideals, Schwartz distributions, multiplicity one}
}

\date{}
\subjclass[2020]{Primary 22E45;
Secondary 22E46, 14L24, 43A85}
\keywords{complex symmetric pair, regularity, invariant theory,
Schwartz distribution, nilpotent orbit, multiplicity one}
\begin{document}
\begin{abstract}
We prove regularity for symmetric pairs of connected complex reductive groups by combining invariant theory with the analysis of Schwartz distributions. For the classical residual families, a Gaussian lifting argument reduces the required equivariant vanishing to invariant theory in quadratic moment quotients: contraction and supertrace identities settle DIII and balanced CII, while Gram-matrix invariants and volume tensors handle BDI/Spin. In the exceptional EVII case, transverse slices and explicit radial differential operators establish the vanishing of the full sign-equivariant Schwartz-distribution space. Together with the reduction to irreducible factors and finite-component assembly, these results establish regularity for all symmetric pairs of connected complex reductive groups.
We further establish multiplicity one for irreducible Casselman–Wallach representations whenever the complex symmetric space is connected, allowing disconnected ambient groups. In particular, van Dijk's unitary Gelfand conjecture follows in this setting. Finally, we construct an invariant tempered distribution supported on an EVII nilpotent orbit closure whose Fourier transform also has nilpotent support, showing that the EVII isotropy representation is not special.
\end{abstract}
\maketitle
\setcounter{tocdepth}{2}
\tableofcontents

\section{Introduction}

Let $G$ be a complex reductive algebraic group and let $\theta$ be an
algebraic involution. Write $G^\theta=\{g\in G:\theta(g)=g\}$ for the fixed subgroup. The associated complex symmetric space is
$G/G^\theta$.
A basic multiplicity-one problem asks whether
\[
\dim\operatorname{Hom}_{G^\theta(\mathbb C)}
(E,\mathbb C)\leq 1
\]
for every irreducible Casselman--Wallach representation
$E$ of $G(\mathbb C)$. A pair satisfying this condition
is called a Gelfand pair (GP1).
The study of multiplicity-one questions for symmetric spaces goes back, in particular, to van Dijk's work on rank-one real symmetric pairs \cite{vD86}.
Van Dijk \cite[Conjecture~2]{vD08} later conjectured that complex symmetric pairs
satisfy the unitary Gelfand property.
In the formulation of Aizenbud and Gourevitch
\cite[\S7.5, Conjecture~1]{AG}, every connected complex
symmetric pair $(G,G^\theta)$, meaning that $G/G^\theta$
is connected, satisfies the unitary Gelfand property (GP3),
where multiplicity one is required only for the smooth
vectors of irreducible unitary representations.
Aizenbud and Gourevitch
introduced the regularity condition and showed that their regularity
conjecture for symmetric pairs would imply van Dijk's conjecture.
They also established regularity for a number of classical families
\cite{AG10}. \footnote{References to \cite{AG} use arXiv:0812.5063v5 and its
numbering. Appendix~\ref{app:homogeneity} concerns that version's
complex homogeneity assertion; \cite{AGcorr} revises the adapted-character
definition and proof. On the semisimple odd summands used here,
$Q(\fp)=\fp$.}

Concrete examples of Gelfand pairs among complex symmetric
pairs include
$(\mathrm{GL}_{2n}(\mathbb C),\mathrm{Sp}_{2n}(\mathbb C))$,
established by Sayag \cite{Say08},
and $(\mathrm{GL}_{n+m}(\mathbb C),
\mathrm{GL}_n(\mathbb C)\times\mathrm{GL}_m(\mathbb C))$,
proved by Aizenbud and Gourevitch \cite{AG}.
Their subsequent work \cite{AG10} also established the Gelfand property
for $(\mathrm{GL}_n(\mathbb C),\mathrm{O}_n(\mathbb C))$
and $(\mathrm{O}_{n+m}(\mathbb C),
\mathrm{O}_n(\mathbb C)\times\mathrm{O}_m(\mathbb C))$. In a related direction, Sun and Zhu \cite{SZ12} proved Archimedean
branching multiplicity-one theorems for standard inclusions of classical
groups. Their orthogonal theorem gives the bound for
$(\SO_{n+1}(\C),\SO_n(\C))$, and hence for the complex symmetric
pair $(\SO_{n+1}(\C),S(\Or_n(\C)\times\Or_1(\C)))$
by restriction to the identity component of the fixed subgroup.

Over non-Archimedean local fields, multiplicity-one
theorems for symplectic and linear pairs were established
by Heumos--Rallis \cite{HR90} and Jacquet--Rallis
\cite{JR96}. For Galois pairs
$(\mathrm{GL}_n(E),\mathrm{GL}_n(F))$, with $E/F$
quadratic, see Flicker \cite[Proposition~11]{Fli91}
and the uniform treatment of Aizenbud--Gourevitch
\cite{AG}. In contrast, general $p$-adic symmetric
pairs need not satisfy multiplicity one; explicit
higher-multiplicity examples appear in
\cite{AP03,FLO12,CG22}.

Returning to complex symmetric pairs, Rubio \cite{Rubio} established the Gelfand property for several exceptional pairs, including $(F_4,\mathrm{Spin}_9)$, as well as the infinite family $(\mathrm{Spin}_{4q+2},\mathrm{Spin}_{4q+1})$. He also carried out a systematic analysis of irreducible factors and their descendants, reducing the remaining regularity problem to a short list of residual families.
Li--Xiao--Yu stated regularity and multiplicity one for connected
complex reductive groups in \cite[Theorems~1.1--1.2]{LXYv1}.
For the passage from irreducible factors to general connected
reductive symmetric pairs, we use the canonical central cover
and finite-component assembly of \cite[\S3]{LXYv1}.
The revised version of the same preprint \cite[Theorem~2.2]{LXYv2} instead uses
partial Fourier transforms in its assembly and establishes
regularity also for disconnected ambient groups. \footnote{The first version of the present paper was
submitted to arXiv on 10 October 2026 (UTC), before
the revised version of Li--Xiao--Yu's paper became
publicly available. Their revised version had been
submitted on 9 October 2026 (UTC).}
The residual families are treated by the arguments outlined below. The local obstructions in Appendices~\ref{app:obstructions}
and~\ref{app:homogeneity} concern specific arguments in
\cite{LXYv1}; the revised proofs do not use the extension steps
or the two-character homogeneity assertion examined there.
Our main results are as follows.

\begin{theorem}[Main theorem]\label{thm:main}
Let $G$ be a connected complex reductive algebraic group and let $\theta$ be an algebraic involution.
\begin{enumerate}[label=\textup{(\roman*)}]
\item The symmetric pair $(G,G^\theta,\theta)$ is regular in the
sense of Aizenbud--Gourevitch \cite[Definition~7.4.2]{AG}. See also~\eqref{eq:regularity}.
\item Let
\[
 m:\widetilde G=Z(G)^\circ\times G_{\rm der}^{\rm sc}\longrightarrow G
\]
be the canonical central cover, with lifted involution $\widetilde\theta$ as in \cite[Lemma~3.1]{LXYv1}, and put $\widetilde H=\widetilde G^{\widetilde\theta}$. Then $(\widetilde G,\widetilde H,\widetilde\theta)$ is Gelfand--Kazhdan: every $\widetilde H$-bi-invariant Schwartz distribution on $\widetilde G$ is fixed by
\[
 \widetilde g\longmapsto\widetilde\theta(\widetilde g^{-1}).
\]
\end{enumerate}
\end{theorem}
The Gelfand--Kazhdan conclusion in part~(ii), together with
a compatible Chevalley involution and finite central descent following \cite[Proposition~8.1 and \S8.3]{LXYv1}, gives multiplicity one
(called \(\mathrm{GP}1\) in \cite{AG}) for connected \(G\). The finite-component
transfer of Section~6.5 then yields the following full
multiplicity-one statement.
\begin{corollary}[Multiplicity one]\label{cor:multiplicity}
Let $G$ be a possibly disconnected complex reductive algebraic group, let $\theta$ be an algebraic involution, and suppose that $G/G^\theta$ is connected. Then the following assertions hold.
\begin{enumerate}[label=\textup{(\roman*)}]
\item For every irreducible Casselman--Wallach representation $E$ of $G(\mathbb{C})$,
\[
 \dim\Hom_{G^\theta(\mathbb{C})}(E,\C)\leq1.
\]
\item (van Dijk's conjecture) In particular, for every irreducible unitary representation $\pi$ of $G(\mathbb{C})$,
\[
 \dim\Hom_{G^\theta(\mathbb{C})}(\pi^\infty,\C)\leq1.
\]
Here $\pi^\infty$ denotes the space of smooth vectors of $\pi$.
\end{enumerate}

\end{corollary}

The regularity and connected-group GP1 conclusions above were already
stated in \cite{LXYv1}; here they follow from alternative local proofs.
Proposition~\ref{prop:GP1-transfer} gives the finite-component transfer
used for disconnected ambient groups. By comparison, \cite{LXYv2}
proves regularity for arbitrary component groups and states GK, GP2,
and GP3 consequences for good pairs, including connected symmetric spaces.

\subsection{Reduction to irreducible factors}
The proof of Theorem~\ref{thm:main}(i) begins on the canonical central cover.
We decompose its simply connected derived group into
\(\widetilde\theta\)-orbits of simple factors and analyze the
regularity problem orbit by orbit. 
Following \cite[Definition~4.1]{Rubio}, put $\mathcal A_\theta
=\{g\in G:\theta(g)\in gZ(G)\}$. A symmetric pair is called pleasant if
\[
\operatorname{Ad}(\mathcal A_\theta)
=\operatorname{Ad}(G^\theta).
\]
 A two-factor orbit gives a
diagonal symmetric pair, which is pleasant
\cite[Example~4.3]{Rubio}. A one-factor orbit determines an
irreducible complex symmetric Lie algebra pair; in treating it we
retain the corresponding effective group action and finite multiplier
data, since regularity is a property of the symmetric group pair, not
of the Lie algebra pair alone.

Pleasant factors automatically satisfy the stronger componentwise
condition introduced in Section~\ref{sec:assembly}, since there is no nontrivial
effective component character to test. Nice factors are regular by \cite[Lemma~4.22]{Rubio}, and
Li--Xiao--Yu \cite[Lemma~3.7]{LXYv1} formulate the corresponding componentwise implication
via speciality. We retain the character-by-character argument of \cite[Lemma~3.7]{LXYv1},
but reprove its nice-factor Fourier-support input using real
homogeneity and the distinguished-defect calculations in Appendix~\ref{app:nice}.

The remaining cases requiring additional arguments are organized,
according to their underlying irreducible symmetric Lie algebra
types, into the following four Cartan families. Write $\mathfrak g=\mathfrak h\oplus\mathfrak p$
for the $(+1)$- and $(-1)$-eigenspace decomposition
of $d\theta$, respectively. The isotropy
representation of $H=G^\theta$ is the adjoint
action of $H$ on
$\mathfrak p\simeq T_{eH}(G/H)$, given by
\[
h\cdot X=\operatorname{Ad}(h)X,
\qquad h\in H,\ X\in\mathfrak p.
\]
The underlying isotropy modules for the four
residual families are as follows; the corresponding
connected effective group actions and finite
multiplier data are specified in
Sections~\ref{sec:DC}--\ref{sec:EVII}.

\[
\begin{aligned}
\mathrm{DIII}:&\quad
(\mathfrak g,\mathfrak h)=(\mathfrak{so}_{2r},\mathfrak{gl}_r),
\qquad
\mathfrak p\simeq\Lambda^2V\oplus\Lambda^2V^*,
\quad V=\mathbf \C^r;\\
\text{balanced CII}:&\quad
(\mathfrak g,\mathfrak h)
=(\mathfrak{sp}_{4r},\mathfrak{sp}_{2r}\oplus\mathfrak{sp}_{2r}),
\qquad
\mathfrak p\simeq\operatorname{Hom}(W,U),
\quad \dim U=\dim W=2r;\\
\mathrm{BDI/Spin}:&\quad
(\mathfrak g,\mathfrak h)
=(\mathfrak{so}_{r+s},\mathfrak{so}_r\oplus\mathfrak{so}_s),
\qquad
\mathfrak p\simeq\operatorname{Hom}(W,U),
\quad (\dim U,\dim W)=(r,s);\\
\mathrm{EVII}:&\quad
(\mathfrak g,\mathfrak h)
=(\mathfrak e_7,\mathfrak e_6\oplus\C),
\qquad
\mathfrak p\simeq J_{27}\oplus J_{27}^*.
\end{aligned}
\]
Here $U,W$ are symplectic in CII and quadratic in BDI, and $J_{27}$ is the exceptional simple Jordan algebra. Unequal
CII blocks are pleasant, so only the balanced CII family appears
above. Within the Spin family, unequal odd--odd blocks are pleasant
by Lemma~\ref{lem:odd-spin}, while the cases \(|r-s|\le 2\) belong to the nice
branch \cite[Lemma~4.21 and Table~3]{Rubio}; thus the residual Spin range has \(|r-s|>2\) and \(rs\)
even.

Sections~\ref{sec:DC}--\ref{sec:EVII} prove full-space vanishing
for the relevant finite-component characters in DIII, balanced CII,
the residual Spin range, and EVII. These statements imply
componentwise regularity and hence ordinary regularity. The nice
factors are treated in Section~\ref{sec:nice-factors}.
\subsection{Equivariant distribution vanishing}
The matrix groups below realize the connected effective isotropy actions through their images; ineffective central kernels leave the invariant rings, moment ideals, and distribution spaces considered here unchanged.

\begin{theorem}[Equivariant distribution vanishing]\label{thm:ambient}
The following character spaces of tempered distributions vanish.
\begin{enumerate}[label=\textup{(\roman*)}]
\item \textup{DIII.} For every $r\geq1$, fix a basis of $V=\C^r$ and put
\[
 P_D=\Lambda^2V\oplus\Lambda^2V^*,\qquad H_D=\GL(V),
\qquad
\tau_D:P_D\to P_D,\quad (B,C)\mapsto(C,B).
\]
The action is $g(B,C)=(gBg^t,g^{-t}Cg^{-1})$, and $\tau_D$ normalizes $H_D$ by $g\mapsto g^{-t}$. For the character $\chi_D$ trivial on $H_D$ and equal to $-1$ on $\tau_D$,
\begin{equation}\label{eq:intro-classical}
 \cS^*(P_D)^{H_D\rtimes\langle\tau_D\rangle,\chi_D}=0.
\end{equation}
\item \textup{Balanced CII.} For every $r\geq1$, identify symplectic spaces $U,W$ of dimension $2r$ and put
\[
 P_C=\Hom(W,U),\qquad H_C=\Sp(U)\times\Sp(W),
\qquad
\tau_C:P_C\to P_C,\quad X\mapsto -X^*.
\]
Here $X^*$ is the symplectic adjoint, $(g,h)X=gXh^{-1}$, and $\tau_C$ is the block swap, which exchanges the group factors. For the character $\chi_C$ trivial on $H_C$ and equal to $-1$ on $\tau_C$,
\begin{equation}\label{eq:intro-C}
 \cS^*(P_C)^{H_C\rtimes\langle\tau_C\rangle,\chi_C}=0.
\end{equation}
\item \textup{BDI/Spin.} For nondegenerate quadratic spaces $U,W$ of arbitrary positive dimensions, let $P=\Hom(W,U)$ with $(g,h)X=gXh^{-1}$. Then
\begin{equation}\label{eq:intro-orthogonal}
 \cS^*(P)^{\Or(U)\times\Or(W),\,\det_U\boxtimes\one}=0,\quad
 \cS^*(P)^{\Or(U)\times\Or(W),\,\one\boxtimes\det_W}=0,
\end{equation}
where $\det_U$ and $\det_W$ are the determinant characters of the two
orthogonal factors.
\end{enumerate}
\end{theorem}

The component calculations in Sections~\ref{sec:DC}--\ref{sec:orthogonal} show that these vanishing
statements exhaust the nontrivial effective component characters
occurring in the corresponding residual classical cases.
Consequently, the associated finite-component data are componentwise regular
in the sense of Section~\ref{sec:assembly}.

The exceptional residual factor admits the following full-space
vanishing statement, proved in Section~\ref{sec:EVII}.

\begin{theorem}[EVII sign vanishing]\label{thm:intro-EVII}
Let $\fp=J_{27}\oplus J_{27}^*$ be the EVII isotropy representation and let
\[
 L=\operatorname{im}(E_6^{\mathrm{sc}}\times\C^\times\to\GL(\fp))
 \simeq(E_6^{\mathrm{sc}}\times\C^\times)/\mu_3
\]
be its connected effective group, acting by
$(h,t)(x,\phi)=(t\rho(h)x,t^{-1}\rho^\vee(h)\phi)$.
Here $\rho$ is the minuscule representation of $E_6^{\mathrm{sc}}$ on
$J_{27}$.
Let $\tau$ be the restriction to $\fp$ of a Chevalley involution of
$E_7$ commuting with the EVII involution; it exchanges $J_{27}$ and
$J_{27}^*$. Put $J=L\rtimes\langle\tau\rangle$, and let $\chi$ be the
character trivial on $L$ with $\chi(\tau)=-1$. Then
\[
 \cS^*(\fp)^{J,\chi}=0.
\]
In particular, the EVII finite-component datum is componentwise regular.
\end{theorem}

Combining the above componentwise regularity statements for all
irreducible factors, the product and intermediate-component argument
of Section~\ref{sec:assembly}, following \cite[\S3]{LXYv1}, yields regularity of
\((G,G^\theta,\theta)\) for every connected complex reductive \(G\).

\subsection{Outline of the distributional argument}
Let a connected reductive group $L$ act on a complex vector space $P$ with an invariant nondegenerate bilinear form $\beta$. Differentiating the invariance condition gives first-order differential
operators annihilating invariant distributions. Using $\beta$ to identify
$P \simeq P^*$, write $x,y$ for the two copies of $P$ in
$T^*P \simeq P \oplus P$. On polynomial Gaussians, the corresponding
operators have quadratic terms
\[
\mu_a(x,y)=-\beta(ax,y), \qquad a\in\operatorname{Lie}L.
\]
The key algebraic problem is to show that the relevant finite-component
character space vanishes in the moment quotient $\cA/I_\mu$, where
\[
 \cA=\C[P\oplus P],\qquad
 I_\mu=(\mu_a:a\in\Lie L).
\]
For isotropy representations of reductive symmetric pairs, the
moment coordinates generate the same ideal as the coordinate functions
of $[x,y]$. Thus $I_\mu=I_{\rm comm}$ as ideals, without passing to
radicals, and the algebraic problem becomes an invariant-theoretic
calculation on the commuting scheme.

In DIII and balanced CII, polarization and the general-linear and symplectic first fundamental theorems express invariants as polynomials in trace contractions. The relevant component exchanges the two block traces, and the supertrace identity places their difference in $I_\mu$. In BDI, the orthogonal first fundamental theorem reduces invariants to Gram matrices. Determinant semi-invariants use one volume tensor, and the moment equations place the resulting Gram commutators in $I_\mu$.

Theorem~\ref{thm:gaussian} then gives the common implication
\[
 (\cA/I_\mu)^{K,\chi}=0
 \quad\Longrightarrow\quad
 \cS^*(P)^{K,\chi}=0,
\]
where $F$ is a finite group of linear symmetries of $P$ normalizing
$L$ and preserving $\beta$, $K=L\rtimes F$, and
$\chi:K\to\{1,-1\}$ is trivial on $L$.
On polynomial Gaussians, the exact operator is
\[
\mathcal D=m+\mathcal F,
\]
where $m(a\otimes Q)=\mu_aQ$
is multiplication by the moment coordinates and $\mathcal{F}$ is the first-order
differential part. The vanishing of the relevant character space in
$\cA/I_\mu$ implies that $m$ is surjective on each homogeneous character
space. We may therefore choose a graded right inverse $\sigma$ of $m$.
Consequently,
\[
\mathcal{D}\sigma=1+N,\qquad N=\mathcal{F}\sigma.
\]
Since $N$ lowers degree by two, it is locally nilpotent, and hence the finite geometric sum for $(1+N)^{-1}$ expresses each polynomial
Gaussian in the character space as lying in the image of the
infinitesimal action.
Finally, Hermite density in the Schwartz topology gives the desired
distributional vanishing.

For EVII, we use a separate argument that retains the sign character
throughout. Semisimple descent first forces both a sign-equivariant
distribution $T$ and its Fourier transform to be supported on the
nilpotent cone. At a nodal transverse slice, exchanging the two branches
makes their extension residues contribute with the same sign, so the
resulting scaling anomaly cannot be cancelled by a distribution supported
at the node. This excludes that orbit and the two larger orbits whose
closures contain it. The holomorphic and antiholomorphic quadratic
operators then give two commuting, locally finite
$\mathfrak{sl}_2$ actions. Comparing their highest weights with the
Kazhdan degrees reduces the remaining support to three nilpotent
orbits. A nonzero transverse quadratic symbol excludes one; exact
transverse operator calculations exclude the other two, including mixed
holomorphic--antiholomorphic jets. Thus no nonzero sign distribution
remains, proving Theorem~\ref{thm:intro-EVII}.

The classical arguments here and in \cite[Theorem~3.2 and
Section~4]{LXYv2} share the moment-map, Gaussian degree-lowering,
and classical invariant-theoretic methods. Our algebraic inputs are
the supertrace identity of Lemma~\ref{lem:supertrace} and the
determinant-ideal calculation of Lemma~\ref{lem:determinant-ideal};
Theorem~\ref{thm:gaussian} retains a prescribed sign character of a
finite symmetry group. The full-space DIII and balanced CII symmetries
also occur in \cite[Propositions~4.3 and~4.6]{LXYv2}.
For unequal orthogonal block dimensions,
\cite[Proposition~4.9]{LXYv2} proves that every
$\SO(U)\times\SO(W)$-invariant tempered distribution is
$\Or(U)\times\Or(W)$-invariant. Theorem~\ref{thm:ambient}(iii)
treats the two single-determinant characters in all positive dimensions,
including the balanced case. In that case, the double-determinant
class $\det x$ is nonzero in the moment quotient, as observed in
Section~\ref{sec:orthogonal}.

In EVII, both proofs use the node at orbit~20 and the two commuting
$\mathfrak{sl}_2$ actions. The treatment here of orbits~17 and~18
uses the radial coefficient maps
\eqref{eq:EVII-M17}--\eqref{eq:EVII-M18} on the complete candidate
jet spaces, whereas \cite[Sections~8.4--8.6]{LXYv2} uses a boundary
calculation. Appendix~\ref{app:homogeneity} gives the additional
construction of an invariant residue measure on
$\overline{\mathcal O_{20}}$ whose Fourier transform also has
nilpotent support. This proves failure of EVII speciality and is
compatible with sign vanishing, since the measure is invariant under
the effective multiplier group.

\subsection{Conventions}\label{sec:conventions}
All algebraic groups, Lie algebras, and finite-dimensional algebraic representations are over $\C$. Ambient groups are connected through the regularity and canonical-cover arguments; the final transfer permits disconnected groups. 
When discussing smooth representations and distributions,
we identify complex algebraic groups with their groups
of complex points, regarded as real Lie groups.
Throughout, $H=G^\theta$ is the fixed subgroup, $H^\circ$ is its identity component, and $L=\operatorname{im}(H^\circ\to\GL(\fp))$ is its connected effective image. Schwartz spaces are taken on the underlying real Nash manifolds. We write $\cS(X)$ for Schwartz functions and $\cS^*(X)$ for their continuous dual, whose elements are called tempered distributions or Schwartz distributions. For a linear action $\rho$, our conventions are
\begin{equation}\label{eq:actions}
 (g_*T)(f)=T(f\circ\rho(g)),\qquad
 U_gf=f\circ\rho(g^{-1}).
\end{equation}
Thus distributions act on functions. The notation $\cS^*(X)^{K,\chi}$ means $g_*T=\chi(g)T$ for all $g\in K$. All representation-theoretic $\Hom$ spaces consist of continuous complex-linear maps. A Casselman--Wallach representation is a smooth Fr\'echet representation of moderate growth with an admissible, finitely generated Harish--Chandra module. Here Harish--Chandra modules are defined using the complexification of the underlying real Lie algebra. The normalized complex absolute value is $|z|_{\C}=z\bar z$.

For a reductive group $L$ acting on $V$, we use
\[
 Q_L(V)=(V/V^L)(\C),\qquad
 \Gamma_L(V)=\pi^{-1}(\pi(0)),\qquad
 R_L(V)=Q_L(V)\setminus\Gamma_L(V),
\]
where $\pi:Q_L(V)\to Q_L(V)\sslash L$ is the categorical quotient. An invariant complement identifies $Q_L(V)$ with its canonical linear model. For an isotropy representation, $\Gamma_H(\fp)$ is the nilpotent cone.

\subsection{Organization of the paper}
Section~\ref{sec:gaussian} develops the moment-quotient and finite-component framework,
including central covers and effective components, characterwise
regularity and the product argument, and the Gaussian lifting theorem. Sections~\ref{sec:DC} and \ref{sec:orthogonal} treat
the classical residual families. Section~\ref{sec:EVII} proves direct
EVII sign vanishing. Section~\ref{sec:consequences} combines these results with the nice-factor calculation and
Rubio's irreducible reduction to prove regularity for symmetric pairs
of connected complex reductive groups and then derives the
multiplicity-one consequences. Appendix~\ref{app:nice}
contains the distinguished-nilpotent calculations;
Appendix~\ref{app:obstructions} treats the classical local obstructions;
Appendix~\ref{app:homogeneity} treats
the complex homogeneity obstruction and constructs an invariant
distribution showing that the EVII isotropy representation is not
special.

\subsection{Logical structure of the argument}
The following diagram summarizes the main logical dependencies.  Citations beside
arrows indicate the source of the corresponding implication, while section
references indicate where the required input is established here.

\clearpage

\begin{center}
\begin{tikzpicture}[
  scale=0.78,
  transform shape,
  >=stealth,
  box/.style={
    draw,
    rectangle,
    align=center,
    inner sep=4pt,
    font=\small,
    text width=3.9cm
  },
  main/.style={
    draw,
    rectangle,
    align=center,
    inner sep=4pt,
    font=\small,
    text width=7.6cm
  },
  mid/.style={
    draw,
    rectangle,
    align=center,
    inner sep=4pt,
    font=\small,
    text width=4.5cm
  },
  tiny/.style={
    draw,
    rectangle,
    align=center,
    inner sep=3pt,
    font=\small,
    text width=2.8cm
  },
  src/.style={
    font=\scriptsize,
    align=center,
    fill=white,
    inner sep=1.3pt
  }
]

% -------------------------------------------------
% Factorwise inputs
% -------------------------------------------------

\node[box] (rubio) at (-4.35,0)
{Rubio's pleasant/nice analysis:\\
\cite[Section~4 and Tables~2--4]{Rubio}};

\node[box] (ours) at (4.35,0)
{Four residual families and nice factors:\\ 
Sections~\ref{sec:DC}-\ref{sec:EVII} and \ref{sec:nice-factors},\\ Appendix~\ref{app:nice}};

\node at ($(rubio.east)!0.5!(ours.west)$)
{\Large $+$};

\coordinate (j1) at (0,-1.30);

\draw (rubio.south) |- (j1);
\draw (ours.south) |- (j1);

\node[main] (irr) at (0,-2.25)
{All irreducible-factor data are
componentwise regular};

\draw[->] (j1) -- (irr.north);

% -------------------------------------------------
% Global regularity
% -------------------------------------------------

\node[main] (reg) at (0,-4.15)
{Every symmetric pair of a connected complex reductive group is regular};

\draw[->] (irr.south) -- (reg.north);

\node[src] at (4.25,-3.20)
{Section~\ref{sec:assembly}; cf.~\cite[Section~3]{LXYv1}};

% -------------------------------------------------
% Canonical cover
% -------------------------------------------------

\node[mid] (cover) at (0,-6.00)
{Central cover $\widetilde G$:\\
Section~\ref{sec:gk-consequence}; cf.~\cite[Lemma~3.1]{LXYv1}};

\draw[->] (reg.south) -- (cover.north);

% -------------------------------------------------
% Descendants + good
% -------------------------------------------------

\coordinate (s2) at (0,-7.18);

\draw (cover.south) -- (s2);

\node[box] (desc) at (-4.35,-8.52)
{All descendants of\\
$(\widetilde G,\widetilde H,\widetilde\theta)$\\
are regular};

\node[box] (good) at (4.35,-8.52)
{$(\widetilde G,\widetilde H,\widetilde\theta)$\\
is good};

\draw (s2) -| (desc.north);
\draw (s2) -| (good.north);

\node at ($(desc.east)!0.5!(good.west)$)
{\Large $+$};

\node[src,anchor=east] at (-5.05,-7.42)
{\cite{Steinberg}};

\node[src,anchor=west] at (5.05,-7.42)
{\cite[Corollary~7.1.7]{AG}};

\coordinate (j2) at (0,-9.92);

\draw (desc.south) |- (j2);
\draw (good.south) |- (j2);

\node[src] at (0,-9.33)
{\cite[Theorem~7.4.5]{AG}};

\node[tiny] (gk) at (0,-10.88)
{Gelfand--Kazhdan\\
(GK)};

\draw[->] (j2) -- (gk.north);

% -------------------------------------------------
% GK -> GP2
% -------------------------------------------------

\node[tiny] (gp2) at (0,-12.65)
{$\mathbf{GP_2}$ on\\
the cover};

\draw[->] (gk.south) -- (gp2.north);

% -------------------------------------------------
% GP2 + compatible Chevalley involution
% -------------------------------------------------

\node[mid] (chev) at (5.10,-12.65)
{Chevalley involution + Central descent:\\
Section~\ref{sec:multiplicity-consequence}; cf.~\cite[Proposition~8.1]{LXYv1}};

\node at ($(gp2.east)!0.5!(chev.west)$)
{\Large $+$};

\coordinate (j3) at (0,-14.28);

\draw (gp2.south) -- (j3);
\draw (chev.south) |- (j3);

% -------------------------------------------------
% GP1 for connected G + connected quotient
% -------------------------------------------------

\node[mid] (gp1c) at (-3.75,-15.55)
{$\mathbf{GP_1}$ for connected $G$};

\draw[->] (j3) -| (gp1c.north);

\node[mid] (quot) at (4.35,-15.35)
{Finite-component transfer:\\
Section~\ref{sec:disconnected-GP1}};

\node at ($(gp1c.east)!0.5!(quot.west)$)
{\Large $+$};

\coordinate (j4) at (0,-16.95);

\draw (gp1c.south) |- (j4);
\draw (quot.south) |- (j4);

\node[main] (gp1d) at (0,-17.92)
{$\mathbf{GP_1}$ for possibly disconnected $G$\\
whenever $G/G^\theta$ is connected};

\draw[->] (j4) -- (gp1d.north);

% -------------------------------------------------
% van Dijk
% -------------------------------------------------

\node[main] (gp3) at (0,-19.82)
{$\mathbf{GP_3}$: van Dijk's conjecture};

\draw[->] (gp1d.south) -- (gp3.north);

\end{tikzpicture}
\end{center}

\clearpage

\subsection{Acknowledgements and disclosure of AI use}
The author thanks his advisor, Dihua Jiang, for his guidance and
encouragement. ChatGPT-6 Astra and Claude Opus 5.5 were used for proof
exploration and drafting. The author is responsible for the mathematical
content, references, and presentation.

\section{Effective data and distributional vanishing}\label{sec:gaussian}\label{sec:geometry}

\input{finite_components}

\subsection{The moment quotient and finite group actions}

Choose a nondegenerate symmetric invariant form $B$ on $\fg$ which is also $\theta$-invariant. Its restrictions to $\fh$ and $\fp$ are nondegenerate, and these summands are orthogonal. Using $B|_{\mathfrak p}$ to identify
$T^*\mathfrak p\simeq \mathfrak p\oplus\mathfrak p$, the cotangent moment map
is given by
\[
\langle \mu(x,y),a\rangle=-B([a,x],y)
=-B(a,[x,y]),\qquad a\in\mathfrak h,
\]
where the second equality follows from the invariance of $B$.
The choice of sign in the moment map is immaterial for the ideals below. Since $[\mathfrak p,\mathfrak p]\subseteq\mathfrak h$ and
$B|_{\mathfrak h}$ is nondegenerate, the moment coordinates generate the
same ideal as the coordinate functions of $[x,y]$. Writing $\cA=\C[\mathfrak p\oplus\mathfrak p]$ and letting $\mathfrak C^2(\mathfrak p)$ denote the closed subscheme defined
by $[x,y]=0$, we obtain
\[
I_\mu=I_{\mathrm{comm}},\qquad
\cA/I_\mu=\C[\mathfrak C^2(\mathfrak p)].
\]
For a matrix model $P\simeq\mathfrak p$, we also write
$\mathfrak C^2(P)$ for this scheme. Ineffective Lie directions contribute zero moment coordinates and do not change the ideal.

If a finite group $F$ of linear symmetries preserves $B$ and normalizes
the connected effective group $L$, its diagonal action preserves the
moment ideal and hence acts on
\begin{equation}\label{eq:connected-quotient}
 \C[\mathfrak C^2(\fp)]^{L}
 =\C[\mathfrak C^2(\fp)\sslash L].
\end{equation}
For a character $\chi:F\to\{1,-1\}$, inflated to $L\rtimes F$,
the following vanishing is the hypothesis of
Theorem~\ref{thm:gaussian}: 
\[
\bigl(\C[\mathfrak C^2(\mathfrak p)]^L\bigr)^{F,\chi}=0,
\]
where the character space is taken after $L$-invariants, retaining the action of $F$.

\subsection{Lifting to Schwartz distributions}

Let $L_0$ be a connected complex reductive group acting on a complex vector space $P$, and let a finite group $F$ act linearly on $P$ and by compatible automorphisms on $L_0$. Put $K=L_0\rtimes F$. Suppose that $P$ has a $K$-invariant nondegenerate complex bilinear form $\beta$. Let $\chi:K\to\{1,-1\}$ be a character trivial on $L_0$.
The action of $L_0$ need not be effective; in applications it may be a matrix group mapping onto $L$.
The proof uses the polynomial-degree filtration on polynomial Gaussians. Fix an $F$-stable maximal compact subgroup $L_c\subset L_0$.\footnote{Such a subgroup exists in general by the Cartan fixed-point theorem.} Put $K_c=L_c\rtimes F$; this compact subgroup meets every component of $K$, and $\Lie L_0=\mathfrak l_c\otimes_\R\C$, where $\mathfrak l_c=\Lie L_c$.

Write $\mathfrak l=\Lie L_0$ and introduce the polynomial ring and moment ideal
\begin{equation}\label{eq:moment}
 \cA=\C[P\oplus P],\qquad
 \mu_a(x,y)=-\beta(ax,y)\quad(a\in\mathfrak l),\qquad
 I_\mu=(\mu_a:a\in\mathfrak l).
\end{equation}
The action on the two formal copies is diagonal, and the action on polynomials is inverse pullback. The form $\beta$ identifies the second copy with the dual representation, so \eqref{eq:moment} is the quadratic cotangent moment ideal, up to a common sign.

\begin{theorem}[Gaussian lifting]\label{thm:gaussian}
In this setting, if
\begin{equation}\label{eq:moment-vanishing}
 (\cA/I_\mu)^{K,\chi}=0,
\end{equation}
then $\cS^*(P)^{K,\chi}=0$.
\end{theorem}

\begin{proof}
Start with any positive definite Hermitian form on $P$ and average
it over $K_c$ with Haar probability measure. The result $h$ is still
positive definite and is $K_c$-invariant; take it linear in its first
variable. Nondegeneracy of $\beta$ defines a unique map $c:P\to P$ by
\begin{equation}\label{eq:antilinear}
 \beta(x,cz)=h(x,z)\qquad(x,z\in P).
\end{equation}
Since the right side is conjugate-linear in $z$, this map is
antilinear. If $cz=0$, then $h(x,z)=0$ for every $x$, hence $z=0$;
finite dimensionality makes $c$ invertible. For $k\in K_c$,
\[
 \beta(kx,c(kz))=h(kx,kz)=h(x,z)=\beta(kx,kcz).
\]
Nondegeneracy in the first variable gives $c(kz)=kcz$.

The substitution $(x,y)=(z,cz)$ identifies $\cA$ with the complex
polynomial functions on the underlying real vector space of $P$.
To see this, choose complex coordinates and write
$cz=C\bar z$ with $C$ an invertible complex matrix. If $z=u+iv$,
then
\[
 x=u+iv,\qquad y=C(u-iv),\qquad
 u=\tfrac12(x+C^{-1}y),\quad
 v=\tfrac1{2i}(x-C^{-1}y).
\]
These are inverse complex-linear changes of $2\dim_\C P$ formal
coordinates. Thus $\C[x,y]\simeq\C[u,v]$, the algebra of
complex-valued real polynomial functions. Compact equivariance of
$c$ makes this an identification of $K_c$-modules. Set
\[
 q(x,y)=\beta(x,y),\qquad
 \gamma(z)=\exp\bigl(-h(z,z)/2\bigr).
\]
We write $P_0\gamma$ for $P_0(z,cz)\gamma(z)$ when $P_0\in\cA$.
The identity $q(z,cz)=h(z,z)$ will determine the multiplication
term in the infinitesimal action.

For $a\in\mathfrak l_c=\Lie L_c$, the real one-parameter subgroup $t\mapsto\exp(ita)$ acts on these coordinates with velocities
\[
 \dot x=iax,\qquad \dot y=-iay.
\]
Indeed, differentiating $c(\exp(ta)z)=\exp(ta)cz$ for real $t$
first gives $c(az)=acz$. Along $z(t)=\exp(ita)z$, antilinearity then
gives $c(iaz)=-ia\,cz$, which is the second displayed velocity.
Consequently the forward vector field, divided by $i$, is
\[
 F_a=(ax)\cdot\partial_x-(ay)\cdot\partial_y.
\]
Infinitesimal invariance of $\beta$ gives
$F_aq=\beta(ax,y)-\beta(x,ay)=2\beta(ax,y)$.
Thus the exact polynomial operator on Gaussian tests is
\begin{equation}\label{eq:gaussian-operator}
 \mathcal D_aP_0=F_aP_0+\mu_aP_0,\qquad
 dU(ia)(P_0\gamma)=-i(\mathcal D_aP_0)\gamma.
\end{equation}
The polynomial identity $\mathcal D_a=F_a+\mu_a$ extends complex-linearly in $a$ to $\mathfrak l$; the identity involving $dU(ia)$ in \eqref{eq:gaussian-operator} is asserted for $a\in\mathfrak l_c$. There is no divergence term, since the test objects in \eqref{eq:actions} are functions. The diagonal algebraic action on $(x,y)$ is the complexification of the compact action; the opposite signs above describe the noncompact directions of the original real group.

On $\mathfrak l\otimes\cA$ define
\[
 m(a\otimes P_0)=\mu_aP_0,\qquad
 \mathcal F(a\otimes P_0)=F_aP_0,\qquad
 \mathcal D=m+\mathcal F.
\]
Here $K_c$ acts on the Lie-generator factor $\mathfrak l$ by the adjoint
action, including the differentials of the automorphisms from $F$:
\[
 k\cdot(a\otimes Q)=\Ad(k)a\otimes U_kQ.
\]
For this action one has
\[
 U_k\mu_a=\mu_{\Ad(k)a},\qquad
 U_k(F_aQ)=F_{\Ad(k)a}(U_kQ).
\]
The first identity follows by moving $k$ through the invariant form
$\beta$; the second follows by conjugating the two linear vector
fields. Thus $m$, $\mathcal F$, and $\mathcal D$ are
$K_c$-equivariant on the whole tensor. If $\cA_d$ denotes homogeneous
degree $d$, then $m$ raises the coefficient degree by two and
$\mathcal F$ preserves it.

On every finite-dimensional polynomial or tensor space, invariance
under $L_c$ is equivalent to invariance under $L_0$. In fact,
differentiation gives invariance under $\mathfrak l_c$, complexification
gives invariance under $\mathfrak l$, and $L_0$ is connected. Retaining $F$
therefore identifies the $K_c$- and $K$-character spaces.
By \eqref{eq:moment-vanishing}, every
$P_0\in\cA_d^{K_c,\chi}$ belongs to $I_\mu$.
Choose a basis $a_1,\ldots,a_l$ of $\mathfrak l$. Ideal membership gives a
finite expression $P_0=\sum_j\mu_{a_j}\widetilde Q_j$.
Taking its degree-$d$ part, since every $\mu_{a_j}$ has degree two,
gives
\[
 P_0=\sum_j\mu_{a_j}Q_j,\qquad Q_j\in\cA_{d-2}.
\]
We take $\cA_e=0$ for $e<0$.

The character projector on the tensor space is
\begin{equation}\label{eq:tensor-projector}
 \Pi_\chi^{(1)}\left(\sum_j a_j\otimes Q_j\right)
 =\int_{K_c}\chi(k)^{-1}
       \sum_j\Ad(k)a_j\otimes U_kQ_j\,dk.
\end{equation}
Equivariance of $m$ shows that its image is
$\Pi_\chi P_0=P_0$. Thus
\begin{equation}\label{eq:homogeneous-surjection}
 m:(\mathfrak l\otimes\cA_{d-2})^{K_c,\chi}
       \longrightarrow\cA_d^{K_c,\chi}
       \quad\hbox{is surjective}.
\end{equation}
In particular the target is zero for $d<2$.
The averaging in \eqref{eq:tensor-projector} includes the Lie-generator factor: in DIII, for example, the component sends the central Lie element $I_r$ to $-I_r$.

Choose a linear right inverse $\sigma_d$ of
\eqref{eq:homogeneous-surjection} in each degree and set
$\sigma=\bigoplus_d\sigma_d$. Both character spaces carry the scalar action $\chi$, so each $\sigma_d$ is $K_c$-equivariant. The degree changes are
\[
 \cA_d^{K_c,\chi}\xrightarrow{\ \sigma_d\ }
 (\mathfrak l\otimes\cA_{d-2})^{K_c,\chi}
 \xrightarrow{\ \mathcal F\ }\cA_{d-2}^{K_c,\chi}.
\]
Consequently $N=\mathcal F\sigma$ lowers degree by two, preserves
the character space, and is locally nilpotent. Moreover
$\mathcal D\sigma=m\sigma+\mathcal F\sigma=1+N$.
For $P_0$ of degree at most $d$, take $M=\lfloor d/2\rfloor$;
then $N^{M+1}P_0=0$. The finite geometric sum gives the exact identity
\begin{equation}\label{eq:finite-correction}
\begin{split}
 \mathcal D\left(\sigma\sum_{j=0}^{M}(-N)^jP_0\right)
 &=(1+N)\sum_{j=0}^{M}(-N)^jP_0\\
 &=P_0+(-1)^M N^{M+1}P_0=P_0.
\end{split}
\end{equation}
This finite correction shows that every polynomial in the character
space lies in the image of $\mathcal D$.

Let $T\in\cS^*(P)^{K,\chi}$. Differentiating its $L_0$-invariance in \eqref{eq:gaussian-operator} gives $T((\mathcal D_aQ)\gamma)=0$ for $a\in\mathfrak l_c$. Continuity of linear vector fields on Schwartz space justifies the differentiation. Since $a\mapsto\mathcal D_a$ is complex-linear, this vanishing extends to every $a\in\mathfrak l$. By \eqref{eq:finite-correction}, $T(P_0\gamma)=0$ for $P_0\in\cA^{K_c,\chi}$.
For arbitrary $f\in\cS(P)$ put
\[
 \Pi_\chi f=\int_{K_c}\chi(k)^{-1}U_kf\,dk.
\]
Since $\chi^2=1$, equivariance gives $T(\Pi_\chi f)=T(f)$. The Gaussian is $K_c$-invariant, hence
$\Pi_\chi(P_0\gamma)=(\Pi_\chi P_0)\gamma$.
It follows that $T$ annihilates every polynomial Gaussian.

After a real-linear change of coordinates, $h(z,z)=|u|^2$ on $\R^{2\dim_\C P}$, so these tests are the polynomial Gaussians $p(u)e^{-|u|^2/2}$. Their span is dense in the Schwartz topology by the usual Hermite expansion \cite{Toft}, and continuity gives $T=0$; in dimension zero the test space already consists of constants.
\end{proof}

\begin{remark}\label{rem:actual-ideal}
Condition~\eqref{eq:moment-vanishing} is equivalent to $\cA^{K,\chi}\subset I_\mu$, since taking a character space is exact in each finite-dimensional degree for the reductive group $K$.
\end{remark}

\section{DIII and balanced CII}\label{sec:DC}

In this section we treat the DIII and balanced CII families. In both families, the first fundamental theorem gives trace generators, and a supertrace identity proves that the component acts trivially on the invariant moment quotient.

\subsection{Invariant tensors and the supertrace identity}\label{sec:tensors}

We use the first fundamental theorems in their tensor form. For $\GL(V)$, invariant tensors in mixed powers of $V$ and $V^*$ are spanned by complete contractions pairing a $V$-index with a $V^*$-index; the numbers of the two kinds of index must agree. For symplectic groups, invariant tensors in powers of the standard representation are spanned by pairings using the defining form. These generation statements hold in every dimension. See \cite[\S5.3.1, Theorem 5.3.1 and Corollary 5.3.2]{GoodmanWallach} for $\GL(V)$ and \cite[\S5.3.2, Theorem 5.3.3]{GoodmanWallach} for symplectic groups.

Here is the trace identity used in both nonorthogonal cases. If $E=E_+\oplus E_-$ and $T$ is an endomorphism preserving the grading, write
$\str T=\tr(T|_{E_+})-\tr(T|_{E_-})$. With respect to the grading $E=E_+\oplus E_-$, we call a block matrix
odd if it interchanges the two summands, equivalently if it has the form
\[
\begin{pmatrix}
0 & A\\
C & 0
\end{pmatrix}.
\]

\begin{lemma}[Supertrace identity]\label{lem:supertrace}
Let $Z_1,\ldots,Z_{2k}$ be odd block matrices over a commutative $\C$-algebra. With ordinary matrix commutators,
\begin{equation}\label{eq:supertrace}
\begin{split}
 2\str(Z_1\cdots Z_{2k})
 =\sum_{j=2}^{2k}\str\bigl(
 Z_2\cdots Z_{j-1}[Z_1,Z_j]
 Z_{j+1}\cdots Z_{2k}\bigr).
\end{split}
\end{equation}
Consequently, if $W=Z_{i_1}\cdots Z_{i_{2k}}$
is an even word, that is, a product of an even number of odd block matrices,
then the difference between the upper and lower block traces of $W$ belongs
to the ideal generated by the entries of the pairwise commutators of its
letters.
\end{lemma}

\begin{proof}
Put $B=Z_2\cdots Z_{2k}$, which is odd. To see the cyclic sign
directly, write two arbitrary odd matrices as
\[
 Z=\begin{pmatrix}0&A\\C&0\end{pmatrix},\qquad
 B=\begin{pmatrix}0&D\\E&0\end{pmatrix}.
\]
Then
\[
 \str(ZB)=\tr(AE)-\tr(CD),\qquad
 \str(BZ)=\tr(DC)-\tr(EA)=-\str(ZB).
\]
The rectangular trace identities used here follow by summing
matrix entries, so remain valid over any commutative coefficient
algebra. In particular $\str([Z_1,B])=2\str(Z_1B)$ for the
commutator $[Z_1,B]=Z_1B-BZ_1$.
The ordinary Leibniz identity is
\[
 [Z_1,Z_2\cdots Z_{2k}]
 =\sum_{j=2}^{2k}Z_2\cdots Z_{j-1}[Z_1,Z_j]
                         Z_{j+1}\cdots Z_{2k}.
\]
Applying supertrace proves \eqref{eq:supertrace}. The grading supplies the cyclic trace sign, while the commutator supplies the commuting equations.

More explicitly, each entry of a product $A[Z_i,Z_j]C$ is
$\sum_{u,v}A_{pu}[Z_i,Z_j]_{uv}C_{vq}$. Hence every summand of
\eqref{eq:supertrace} lies in the ideal of commutator
entries. This identity, including division by two, holds over every commutative $\C$-algebra, including those with nilpotents.
\end{proof}

\begin{corollary}\label{cor:trace-ideal}
Let $H$ be a connected reductive group acting by block-diagonal
conjugation on a linear space $P$ of odd block matrices. Suppose that
$\C[P\oplus P]^H$ is generated by upper block traces of even words.
Let $\tau$ be an involution of $P$ induced by a conjugation interchanging
the two blocks, and assume that $\tau$ normalizes the image of $H$ in
$\GL(P)$. Let $I\subset\C[P\oplus P]$ be the ideal generated by the entries
of the commutator of the two copies. Then
\[
 \bigl(\C[P\oplus P]/I\bigr)^{H,\tau=-1}=0.
\]
\end{corollary}

\begin{proof}
Let $t_i$ denote the upper trace generators. The involution takes
$t_i$ to the lower trace of the same word. Every letter of that
word is one of the two formal copies, so its pairwise commutators
are zero or the defining commutator with either sign.
Lemma~\ref{lem:supertrace} therefore gives $t_i-\tau(t_i)\in I$.

An invariant $F$ has an expression $f(t_1,\ldots,t_N)$ involving
finitely many generators. For indeterminates $u_i,v_i$, ordinary
polynomial division, or successive replacement of one variable,
gives polynomials $g_i$ such that
\begin{equation}\label{eq:polynomial-telescope}
 f(u_1,\ldots,u_N)-f(v_1,\ldots,v_N)
 =\sum_{i=1}^N(u_i-v_i)g_i(u,v).
\end{equation}
Substitute $u_i=t_i$, $v_i=\tau(t_i)$. It follows that
$F-\tau(F)\in I$. The identity in the free polynomial ring remains valid after evaluation in the invariant algebra. If $\tau(F)=-F$, we obtain
$2F\in I$, and hence $F\in I$.

Finally $I$ is stable under $H$ and $\tau$. In every homogeneous
degree, compact averaging lifts an invariant class in the quotient
to an $H$-invariant polynomial. If the class is odd, replacing this
lift by $(F-\tau F)/2$ gives an odd invariant lift. Such a lift
belongs to $I$ by the preceding argument, so the class is zero.
\end{proof}

\subsection{DIII: the moment equations}
Fix a basis of $V=\C^r$. Write the elements of $P_D$ as pairs of skew-symmetric matrices. The action and normalizing involution are
\begin{equation}\label{eq:D-action}
 g(B,C)=(gBg^t,g^{-t}Cg^{-1}),\qquad
 \tau_D(B,C)=(C,B).
\end{equation}
Thus $\tau_D$ acts on $\GL(V)$ by $g\mapsto g^{-t}$. For
$x=(B_1,C_1)$ and $y=(B_2,C_2)$, put
\[
 Z_x=\begin{pmatrix}0&B_1\\ C_1&0\end{pmatrix},\qquad
 Z_y=\begin{pmatrix}0&B_2\\ C_2&0\end{pmatrix},\qquad
 M=B_1C_2-B_2C_1.
\]
The matrices $Z_x$ and $Z_y$ lie in the orthogonal Lie algebra
for the hyperbolic form on $V\oplus V^*$. Direct multiplication gives
\begin{equation}\label{eq:D-bracket}
 [Z_x,Z_y]=\diag(M,-M^t).
\end{equation}
Use the nondegenerate symmetric invariant form
\[
 \beta_D(x,y)=-\frac12\tr(B_1C_2+B_2C_1).
\]
It is also $\tau_D$-invariant. For $a\in\mathfrak{gl}(V)$ the infinitesimal action is
$a(B,C)=(aB+Ba^t,-a^tC-Ca)$, so
\begin{align*}
 \beta_D(ax,y)
 &=-\frac12\tr\bigl(aB_1C_2+B_1a^tC_2
                         -B_2a^tC_1-B_2C_1a\bigr)\\
 &=-\tr(aM).
\end{align*}
For instance, transposition and skew-symmetry give
$\tr(a^tC_2B_1)=\tr(aB_1C_2)$, accounting for the first two equal contributions. Therefore
\begin{equation}\label{eq:D-moment}
 \mu_a(x,y)=\tr(aM),\qquad I_D=(M_{ij}:1\leq i,j\leq r).
\end{equation}
Indeed, $a=E_{ij}$ gives $\mu_a=M_{ji}$, so each entry is a moment
coordinate and every moment coordinate is a linear combination of
entries. The central element $a=I_r$ gives
\[
 a(B,C)=(2B,-2C),\qquad \mu_{I_r}=\tr M.
\]
For $r\geq2$, the nontrivial central action contributes the trace
equation. In \eqref{eq:D-bracket}, the lower commutator block is
the negative transpose of the upper block.

\subsection{DIII: invariant contractions}
\begin{lemma}\label{lem:D-FFT}
On any finite number of copies of
$\Lambda^2V\oplus\Lambda^2V^*$, the $\GL(V)$-invariant algebra is
generated by
\begin{equation}\label{eq:D-traces}
 \tr(B_{i_1}C_{i_2}\cdots B_{i_{2k-1}}C_{i_{2k}}),\qquad k\geq1,
\end{equation}
where labels may repeat. This generation statement holds in every
dimension of $V$.
\end{lemma}

\begin{proof}
First take a multihomogeneous invariant and polarize separately
all its matrix occurrences. In characteristic zero the resulting
multilinear function recovers the original polynomial by setting
the new copies equal; polarization is equivariant. If there are
$p$ occurrences of $B$ and $q$ of $C$, a scalar matrix $tI$ sends
its arguments to $t^2B$ and $t^{-2}C$. Invariance forces
$t^{2p-2q}=1$ for all $t\in\C^\times$, hence $p=q$.

In index notation each $B$ has two upper indices and each $C$ two
lower indices. The equivariant antisymmetrizers
$V^{\otimes2}\to\Lambda^2V$ and
$(V^*)^{\otimes2}\to\Lambda^2V^*$ allow us to apply the
general-linear tensor FFT before restricting to the skew tensors.
Every multilinear invariant is therefore a linear combination of
complete contractions pairing each upper index with a lower one.
Make a vertex for each polarized occurrence and an edge for each
paired pair of indices. Every vertex has valence two, and every
edge joins a $B$-vertex to a $C$-vertex. Each connected component is
therefore an alternating even cycle. Two parallel edges are
allowed and give a cycle of length two.

Orient such a cycle and number its consecutive indices
$a_1,\ldots,a_{2k}$. Up to the sign from reversing the two indices
at individual skew vertices, its contraction is
\[
 \sum_{a_1,\ldots,a_{2k}}
 (B_{i_1})^{a_1a_2}(C_{i_2})_{a_2a_3}
 \cdots
 (B_{i_{2k-1}})^{a_{2k-1}a_{2k}}
 (C_{i_{2k}})_{a_{2k}a_1}.
\]
This is exactly \eqref{eq:D-traces}. Different connected components
have disjoint summation indices, so their contractions multiply.
Depolarization proves the assertion, with repeated labels obtained by setting polarized arguments equal.

\end{proof}

The trace \eqref{eq:D-traces} is the upper block trace of the corresponding word in $Z_x,Z_y$. Conjugation by
$S=\left(\begin{smallmatrix}0&I\\I&0\end{smallmatrix}\right)$
realizes $\tau_D$. Equations \eqref{eq:D-bracket} and \eqref{eq:D-moment} identify the commutator ideal with the moment ideal. Corollary~\ref{cor:trace-ideal} yields
\begin{equation}\label{eq:D-quotient}
 \bigl(\C[P_D\oplus P_D]/I_D\bigr)^{H_D,\tau_D=-1}=0.
\end{equation}
Since $F-\tau_DF\in I_D$ for every $H_D$-invariant polynomial, exactness of connected invariants shows that $\tau_D$ acts identically on $\C[\mathfrak C^2(P_D)]^{\GL(V)}$.

\subsection{Balanced CII: symplectic contractions}
Choose symplectic bases of the identified spaces $U,W$ with
$J_U=J_W=\diag(J_2,\ldots,J_2)$ and
$J_2=\left(\begin{smallmatrix}0&1\\-1&0\end{smallmatrix}\right)$.
For $X:W\to U$, put
\[
 X^*=J_W^{-1}X^tJ_U,\qquad
 Z_X=\begin{pmatrix}0&X\\-X^*&0\end{pmatrix}
 \in\mathfrak{sp}(U\oplus W).
\]
The action is $(g,h)X=gXh^{-1}$. Conjugation by
$S=\left(\begin{smallmatrix}0&I\\I&0\end{smallmatrix}\right)$
gives $SZ_XS^{-1}=Z_{-X^*}$, so the component is exactly
$\tau_CX=-X^*$ and exchanges the two group factors.
The form $\beta_C(X,Y)=\tr(X^*Y)$ is symmetric, nondegenerate,
and invariant under both the group and this component.
Put
\begin{equation}\label{eq:C-LR}
 \mathsf{L}=XY^*-YX^*,\qquad \mathsf{R}=X^*Y-Y^*X.
\end{equation}
Then $\mathsf{L}^*=-\mathsf{L}$, $\mathsf{R}^*=-\mathsf{R}$, and
\begin{equation}\label{eq:C-bracket}
 [Z_X,Z_Y]=\diag(-\mathsf{L},-\mathsf{R}).
\end{equation}
For $(a,b)\in\mathfrak{sp}(U)\oplus\mathfrak{sp}(W)$,
\begin{equation}\label{eq:C-moment}
 \mu_{a,b}(X,Y)=-\tr((aX-Xb)^*Y)
 =-\tfrac12(\tr_U(a\mathsf{L})+\tr_W(b\mathsf{R})).
\end{equation}
Indeed, $a^*=-a$ gives
$\tr(aXY^*)=-\tr(aYX^*)$, and similarly for $b$.
The trace pairing is nondegenerate on each star-skew-adjoint Lie algebra, since its orthogonal complement in the full matrix algebra is the star-self-adjoint subspace. Thus every entry
of $\mathsf{L},\mathsf{R}$ is a linear combination of their trace-dual coordinates,
and
\begin{equation}\label{eq:C-ideal}
 I_C=(\text{entries of }\mathsf{L},\text{ entries of }\mathsf{R})=I_{\rm comm}.
\end{equation}
Both Lie factors occur, including in symplectic dimension two.

\begin{proposition}\label{prop:C-quotient}
Every $H_C$-invariant polynomial $F$ on $P_C\oplus P_C$ satisfies $F-\tau_CF\in I_C$. In particular,
\begin{equation}\label{eq:C-quotient}
 \bigl(\C[P_C\oplus P_C]/I_C\bigr)^{H_C,\tau_C=-1}=0.
\end{equation}
\end{proposition}

\begin{proof}
Polarize a multihomogeneous invariant, giving a separate argument to each matrix occurrence. Under $P_C\simeq U\otimes W^*$, each occurrence has one $U$-index and one $W^*$-index. The symplectic FFT for $\Sp(U)\times\Sp(W)$ expresses every multilinear invariant as a linear combination of complete contractions, pairing the $U$-indices and the $W$-indices independently by their symplectic forms.

Give each matrix occurrence a vertex and each pairing an edge, marked by its factor. Each vertex is incident to one edge of each kind, so every connected component is an even cycle with alternating edge types. Traversing successive $U$- and $W$-pairings alternates a map $X$ with its symplectic adjoint $X^*$; the cycle therefore gives, up to the sign determined by the orientations,
\begin{equation}\label{eq:C-traces}
 t_{i_1,\ldots,i_{2k}}
 =\tr(X_{i_1}X_{i_2}^*\cdots X_{i_{2k-1}}X_{i_{2k}}^*).
\end{equation}
Disjoint cycles give products of these traces. Depolarization therefore proves that \eqref{eq:C-traces} generates the invariant algebra on any number of copies. Labels may repeat, and the FFT gives this generation in every positive symplectic dimension.

The upper block trace of $Z_{X_{i_1}}\cdots Z_{X_{i_{2k}}}$ is $(-1)^kt_{i_1,\ldots,i_{2k}}$. Since $SZ_XS^{-1}=Z_{\tau_CX}$, the component exchanges the upper and lower block traces of the same word. Lemma~\ref{lem:supertrace} expresses their difference in the ideal of its pairwise commutator entries. On two copies every letter is $Z_X$ or $Z_Y$, so these entries are zero or entries of $[Z_X,Z_Y]$ with either sign. Equations \eqref{eq:C-bracket} and \eqref{eq:C-ideal} identify their ideal with $I_C$. Hence each generator $t$ satisfies
\[
 t-\tau_Ct\in I_C.
\]

Write an invariant polynomial as $F=f(t_1,\ldots,t_N)$. The telescoping identity \eqref{eq:polynomial-telescope} gives $F-\tau_CF\in I_C$. Finally, compact averaging in each degree lifts an $H_C$-invariant quotient class to an invariant polynomial. For an odd class, replacing the lift by $(F-\tau_CF)/2$ gives an odd invariant lift in $I_C$, so the class is zero.
\end{proof}

\subsection{Distributional vanishing and regularity}

\begin{proof}[Proof of Theorem~\ref{thm:ambient}\textup{(i), (ii)}]
For DIII take $H_{D,c}=U(r)$ and
\[
 c_D(B,C)=(\bar C,\bar B).
\]
This commutes with $U(r)$ and $\tau_D$, and
\[
 \beta_D(x,c_Dx)
 =-\tfrac12\tr(B\bar B+\bar C C)
 =\sum_{i<j}(|B_{ij}|^2+|C_{ij}|^2).
\]
The form is positive definite on the independent skew entries. The moment operator in Theorem~\ref{thm:gaussian} is therefore
$\mathcal D_a=F_a+\tr(aM)$, with $F_a$ obtained from the action \eqref{eq:D-action}. Equations \eqref{eq:D-moment} and \eqref{eq:D-quotient}
verify the hypothesis of Theorem~\ref{thm:gaussian}. This proves the first assertion. For $r=1$ the representation is zero and the nontrivial component character itself forces vanishing.

For CII take $H_{C,c}=\operatorname{USp}(U)\times\operatorname{USp}(W)$ and
\[
 c_CX=-J_U\bar XJ_W.
\]
It commutes with the compact group and $\tau_C$. Moreover
$\beta_C(X,c_CX)=\tr(X^t\bar X)=\sum_{i,j}|X_{ij}|^2$.
On the two formal copies the operator is explicitly
\[
 \mathcal D_{a,b}
 =(aX-Xb)\cdot\partial_X+(-aY+Yb)\cdot\partial_Y
       -\tfrac12(\tr_U(a\mathsf{L})+\tr_W(b\mathsf{R})).
\]
Here the last term denotes multiplication. Apply Theorem~\ref{thm:gaussian} using \eqref{eq:C-moment} and \eqref{eq:C-quotient}. In each case the chosen compact subgroup is stable under the component involution.
\end{proof}

\begin{lemma}\label{lem:DC-components}
For DIII, the connected effective group is the image of $\GL_r$ on
$\Lambda^2V\oplus\Lambda^2V^*$. Its maximal multiplier quotient is
trivial for odd $r$ and has order two for even $r$, represented by
$\tau_D(B,C)=(C,B)$. For balanced CII, it is the image of
$\Sp(U)\times\Sp(W)$ with multiplier quotient of order two,
represented by $\tau_CX=-X^*$.
\end{lemma}
\begin{proof}
We reproduce the calculations of \cite[Lemma 4.2, Corollary 4.3,
Lemma 4.4, and the proof
of Lemma 5.1]{LXYv1}. In the orthogonal model on $V\oplus V^*$,
the DIII involution is conjugation by $s=\diag(iI,-iI)$.
A multiplier $a\in\SO_{2r}$ satisfies $asa^{-1}=\pm s$.
For the plus sign it is $\diag(g,g^{-t})$. For the minus sign it
is $\left(\begin{smallmatrix}0&g\\g^{-t}&0\end{smallmatrix}\right)$,
whose determinant is $(-1)^r$. Such elements exist exactly when
$r$ is even and form the coset of
$S=\left(\begin{smallmatrix}0&I\\I&0\end{smallmatrix}\right)$.
Conjugation by $S$ gives $\tau_D$. The two cosets remain distinct
in the effective action for the simple DIII ranks. Central lifts
and quotients have the same maximal effective multiplier group by
Lemma~\ref{lem:central-interface}, and their connected fixed images
have the displayed Lie algebra and hence the same connected image.

For CII use $s=\diag(I_U,-I_W)$ in $\Sp(U\oplus W)$.
The center is $\{\pm I\}$, so a multiplier either preserves both
blocks or exchanges them. In the balanced case every exchange is
a block-diagonal element times $S$. Direct multiplication gives
$SZ_XS^{-1}=Z_{-X^*}$, and $S$ exchanges the two symplectic factors.
The block-diagonal subgroup is connected, and the exchange gives
the other effective component. These two cosets give the full
multiplier group.
\end{proof}

\begin{corollary}\label{cor:DC-regular}
Every DIII pair $(D_r,A_{r-1}+\C)$, $r\geq4$, in every connected
isogeny form, and every balanced CII pair $(C_{2r},C_r+C_r)$,
$r\geq1$, is regular. Their finite-component data are componentwise regular.
\end{corollary}
\begin{proof}
Lemma~\ref{lem:DC-components} leaves at most one nontrivial character.
Theorem~\ref{thm:ambient}\textup{(i), (ii)} proves its full-space
vanishing, which is the consequent of \eqref{eq:componentwise}.
The fixed-vector quotient is the displayed isotropy space by
Lemma~\ref{lem:cover-center}. Proposition~\ref{prop:intermediate}
then handles every full fixed subgroup after a central quotient.
\end{proof}

\section{BDI/Spin: Gram matrices and determinant tensors}\label{sec:orthogonal}

We now turn to the BDI/Spin family, whose proof uses orthogonal invariant theory. The orthogonal first fundamental theorem reduces the two-copy invariants to Gram matrices. We first prove the determinant-tensor statement that will be applied to their symmetric parts.

\subsection{A determinant ideal}
The orthogonal tensor FFT spans invariant tensors by pairings with the quadratic form; see \cite[\S5.3.2, Theorem 5.3.3]{GoodmanWallach}. If $E$ is an $n$-dimensional quadratic space, with inverse quadratic tensor $g_E$ and a nonzero volume tensor $\varepsilon_E\in\Lambda^nE$, then
\begin{equation}\label{eq:volume-tensors}
 (E^{\otimes d})^{\Or(E),\det}
 =\operatorname{span}\left\{
 \text{permutations of }\varepsilon_E\otimes
 g_E^{\otimes(d-n)/2}\right\}.
\end{equation}
The space is zero unless $d\geq n$ and $d-n$ is even
\cite[Theorem 2.1 and Remark 2.2]{LZvolume}.

For symmetric $n\times n$ matrices $A_1,\ldots,A_m$, let
$\mathcal B=\C[(\Sym^2\C^n)^{\oplus m}]$, with simultaneous orthogonal conjugation, and let $\mathfrak c$ be the ideal generated by the entries of $[A_i,A_j]$.

\begin{lemma}[Determinant character in the commutator ideal]\label{lem:determinant-ideal}
For $n\geq1$,
\[
 \mathcal B^{\Or_n,\det}\subset\mathfrak c,
 \qquad
 (\mathcal B/\mathfrak c)^{\Or_n,\det}=0.
\]
\end{lemma}

\begin{proof}
The central element $-I_n$ acts trivially on symmetric matrices
by conjugation. If $n$ is odd, its determinant is $-1$, so the
determinant-semi-invariant space is zero. Suppose henceforth that
$n>0$ is even, and polarize a homogeneous determinant-semi-invariant
polynomial, giving a separate argument to each matrix occurrence.
The equivariant symmetrization from two standard indices to a
symmetric matrix allows us to apply \eqref{eq:volume-tensors}.
The result is a linear combination of contractions with exactly
one volume tensor and all other indices paired by the quadratic
form.

In the contraction graph, each symmetric-matrix occurrence supplies two tensor indices and hence a bivalent vertex, each orthogonal pairing gives an edge, and the unique volume tensor gives one $n$-valent vertex. Deleting the volume vertex leaves paths and cycles: a cycle contributes the trace of a matrix word, while a path contributes a word whose two endpoint indices occupy two slots of the volume tensor. Since the matrices are symmetric, each path word has the form $M=A_{i_1}\cdots A_{i_l}$; a direct pairing of two slots of the volume tensor is an empty path and contributes $I_n$. Thus every contraction, up to the sign from ordering the slots of the volume tensor, has the form
\begin{equation}\label{eq:volume-graph}
 \left(\prod_\ell\tr(W_\ell)\right)
 \sum_{j_1,\ldots,j_n}\varepsilon_{j_1\cdots j_n}
       \prod_{a=1}^{n/2}(M_a)_{j_{2a-1},j_{2a}}.
\end{equation}

Modulo $\mathfrak c$, the matrices $A_i$ commute as matrices over
the quotient algebra. Since each is symmetric, a path word then
satisfies
\[
 M^t=A_{i_l}\cdots A_{i_1}
       =A_{i_1}\cdots A_{i_l}=M\pmod{\mathfrak c}.
\]
For explicit ideal membership, choose a sequence of adjacent
interchanges taking the word $W_0=M$ to $W_N=M^t$. At the
$\nu$th step write the two successive words as
$W_{\nu-1}=B_\nu A_{i_\nu}A_{j_\nu}C_\nu$ and
$W_\nu=B_\nu A_{j_\nu}A_{i_\nu}C_\nu$, where $B_\nu,C_\nu$
are the surrounding words (possibly empty). Telescoping gives
\begin{equation}\label{eq:word-reversal}
 M-M^t=\sum_{\nu=1}^N(W_{\nu-1}-W_\nu)
      =\sum_{\nu=1}^N
          B_\nu[A_{i_\nu},A_{j_\nu}]C_\nu.
\end{equation}
Thus every entry of $M-M^t$ is a polynomial combination of
commutator entries.

There is at least one path in \eqref{eq:volume-graph}, since $n>0$.
Fix one of its factors $M$, and interchange its two volume indices
in the summation. All other path factors use different slots of the volume tensor, so they are unchanged, whereas the volume tensor changes
sign. The contraction with $M^t$ is consequently the negative of
the contraction with $M$. Replacing $M$ by $(M-M^t)/2$ leaves the
contraction unchanged. Equation~\eqref{eq:word-reversal} now
puts that contraction in $\mathfrak c$. For an empty path it is
zero already, since $I_n$ is symmetric. Summing contractions and
depolarizing proves $\mathcal B^{\Or_n,\det}\subset\mathfrak c$.

The ideal is homogeneous and orthogonally stable. In each finite
degree, average a polynomial lift of a determinant-equivariant
quotient class over $\Or_n(\R)$ with the determinant character.
The resulting lift is determinant-equivariant for $\Or_n(\C)$,
by complexifying the compact connected Lie algebra and retaining
a reflection. It lies in $\mathfrak c$ by the inclusion just
proved, so the quotient class is zero.
\end{proof}

\subsection{The BDI/Spin moment quotient}
Let $U,W$ be quadratic spaces of positive dimensions $r,s$, and choose orthonormal complex bases. Put $P=\Hom(W,U)$, with action
$(g,h)X=gXh^{-1}$.
For two formal $r\times s$ matrices $x,y$, write
\[
 \mathsf{L}=xy^t-yx^t,\qquad \mathsf{R}=x^ty-y^tx,\qquad
 I=(\text{entries of }\mathsf{L},\text{ entries of }\mathsf{R}).
\]
The invariant bilinear form is $\beta(x,y)=\tr(x^ty)$. If
$(a,b)\in\mathfrak{so}(U)\oplus\mathfrak{so}(W)$, then
\begin{equation}\label{eq:O-moment}
\begin{split}
 \mu_{a,b}(x,y)
 &=-\tr((ax-xb)^ty)\\ 
 &=\tr(x^tay)-\tr(bx^ty)
   =-\tfrac12(\tr_U(a\mathsf{L})+\tr_W(b\mathsf{R})).
\end{split}
\end{equation}
The trace pairings on the skew-symmetric Lie algebras show that $I$ is exactly the moment ideal. In dimension one the corresponding skew matrix and Lie algebra are both zero.

\begin{proposition}\label{prop:O-quotient}
If $\chi_W=\one\boxtimes\det_W$, then
\[
 \C[P\oplus P]^{\Or(U)\times\Or(W),\chi_W}\subset I.
\]
Equivalently,
\[
 \bigl(\C[P\oplus P]/I\bigr)^{\Or(U)\times\Or(W),\chi_W}=0.
\]
\end{proposition}

\begin{proof}
View the columns of $x$ and $y$ as $2s$ vectors in $U$. The
orthogonal pairing theorem says that their left-invariant ring
is generated by the entries of the Gram blocks
\[
 A=x^tx,\qquad B=x^ty,\qquad C=y^ty.
\]
Here $A,C$ are symmetric and $B$ is initially arbitrary. More
precisely, evaluation in these matrices is a surjection from the
polynomial algebra on a formal symmetric matrix, a formal arbitrary
matrix, and a second formal symmetric matrix onto
$\C[P\oplus P]^{\Or(U)}$. The right orthogonal action takes every
Gram block to $hAh^{-1}$, $hBh^{-1}$, or $hCh^{-1}$, respectively,
so this surjection is $\Or(W)$-equivariant in all dimensions.

Let $F$ have the character $\one\boxtimes\det_W$ and choose a
polynomial Gram lift $\widetilde F$. Replace this lift by
\[
 \widetilde F_{\det}
 =\int_{\Or_s(\R)}\det(h)^{-1}U_h\widetilde F\,dh.
\]
The averaging takes place in a finite-dimensional degree space,
and its image is still $F$. It gives a determinant-equivariant
polynomial in the three formal matrices. In this argument all
polynomial actions are inverse pullbacks; a determinant character
is equal to its inverse.

Direct multiplication gives
\begin{equation}\label{eq:Gram-identities}
\begin{aligned}
 x^t\mathsf{L}x&=AB^t-BA,& x^t\mathsf{L}y&=AC-B^2,\\
 y^t\mathsf{L}x&=(B^t)^2-CA,& y^t\mathsf{L}y&=B^tC-CB,\\
 \mathsf{R}&=B-B^t.
\end{aligned}
\end{equation}
Put $B_s=(B+B^t)/2=B-\mathsf{R}/2$. All three matrices $A,B_s,C$ are
symmetric. The following identities place their commutators in the moment ideal:
\begin{equation}\label{eq:symmetric-Gram-identities}
\begin{aligned}
 [A,B_s]&=x^t\mathsf{L}x+\tfrac12(A\mathsf{R}+\mathsf{R}A),\\
 [B_s,C]&=y^t\mathsf{L}y+\tfrac12(\mathsf{R}C+C\mathsf{R}),\\
 [A,C]&=x^t\mathsf{L}y+y^t\mathsf{L}x+B\mathsf{R}+\mathsf{R}B^t.
\end{aligned}
\end{equation}
For example, $x^t\mathsf{L}x=[A,B]-A\mathsf{R}$, while
$[A,B_s]=[A,B]-(A\mathsf{R}-\mathsf{R}A)/2$, proving the first identity.
For the second use $y^t\mathsf{L}y=[B,C]-\mathsf{R}C$; for the third use
$B^2-(B^t)^2=B\mathsf{R}+\mathsf{R}B^t$ in the sum of the middle two identities in
\eqref{eq:Gram-identities}. Every entry of every right-hand side
is a polynomial multiple of entries of $\mathsf{L}$ or $\mathsf{R}$.

Restrict $\widetilde F_{\det}$ in its formal middle variable to
the symmetric matrices. This restriction is still
determinant-equivariant, because transpose and symmetrization
commute with orthogonal conjugation. After evaluation, its value
is $\widetilde F_{\det}(A,B_s,C)$. By
\eqref{eq:polynomial-telescope}, applied to the individual middle
matrix entries, the difference
\[
 \widetilde F_{\det}(A,B,C)
       -\widetilde F_{\det}(A,B_s,C)
\]
is in the ideal of the entries of $B-B_s=\mathsf{R}/2$, hence in $I$.
Lemma~\ref{lem:determinant-ideal} expresses the restricted formal
polynomial as a combination of the commutator entries of three
formal symmetric matrices. Evaluating them at $A,B_s,C$ and
using \eqref{eq:symmetric-Gram-identities} puts its value in $I$.
The first inclusion follows. Finally compact character averaging
in each degree lifts a character class in the quotient to one in
the polynomial algebra; the inclusion forces that class to be
zero, proving the second assertion.
\end{proof}

Transposition gives the same assertion with $U,W$ interchanged.
Thus, in the moment quotient,
\begin{equation}\label{eq:O-connected-sectors}
\C[\mathfrak C^2(P)]^{\Or(U)\times\Or(W),\,
       \det_U\boxtimes\one}=0,\ \C[\mathfrak C^2(P)]^{\Or(U)\times\Or(W),\,
       \one\boxtimes\det_W}=0.
\end{equation}
Equivalently, these are the two single-determinant character spaces for the residual quotient
$(\Or(U)\times\Or(W))/(\SO(U)\times\SO(W))$.
When $r=s$, the double-determinant class $\det x$ has character $\det_U\boxtimes\det_W$ and survives in
the moment quotient, since it takes value one at the commuting
pair $(x,y)=(I_r,0)$.

\begin{proof}[Proof of Theorem~\ref{thm:ambient}\textup{(iii)}]
Apply Theorem~\ref{thm:gaussian} to the connected matrix group $L_0=\SO(U)\times\SO(W)$, which maps onto the connected effective image on $P$. Reflections in the chosen bases give a splitting
\[
 \Or(U)\times\Or(W)=L_0\rtimes(\{1,-1\}\times\{1,-1\}).
\]
Take $L_c=\SO_r(\R)\times\SO_s(\R)$, which is stable under these reflections, so $K_c=\Or_r(\R)\times\Or_s(\R)$. The map $cX=\bar X$ gives $\beta(X,cX)=\sum|X_{ij}|^2$. The actual Gaussian operators are
\[
 \mathcal D_{a,b}
 =(ax-xb)\cdot\partial_x+(-ay+yb)\cdot\partial_y
       -\tfrac12(\tr_U(a\mathsf{L})+\tr_W(b\mathsf{R})).
\]
Equations \eqref{eq:O-moment} and Proposition~\ref{prop:O-quotient} meet the hypotheses of Theorem~\ref{thm:gaussian} for $\chi_W$. Its finite correction applies in every degree and proves the $\one\boxtimes\det_W$ assertion. Transposition exchanges $U$ and $W$ and proves the $\det_U\boxtimes\one$ assertion.
\end{proof}

\subsection{The Spin component}
\begin{lemma}\label{lem:odd-spin}
If $r,s$ are odd and unequal, the Spin block pair is pleasant.
\end{lemma}
\begin{proof}
This is \cite[Proposition 6.6]{LXYv1}; we include its Clifford proof.
Let $E=U\oplus W$, let $I=1_U\oplus(-1_W)$, and let $\theta$
be the induced involution of $\Spin(E)$. For an adapted orthonormal
basis the Clifford volume element $\omega$ satisfies
$\theta(\omega)=(-1)^s\omega=-\omega$. The center is
$\{\pm1,\pm\omega\}$, with $\pi(\pm\omega)=-I_E$ under
$\pi:\Spin(E)\to\SO(E)$. If $\theta(g)=\pm\omega g$,
then $I\pi(g)I^{-1}=-\pi(g)$, so the diagonal blocks of $\pi(g)$
vanish. Invertibility would imply $r=s$. Thus the multiplier is
$\pm1$. In the negative case $\theta(\omega g)=\omega g$ and
$\Ad(\omega g)=\Ad(g)$. Every multiplier therefore acts through
the fixed subgroup, which is pleasantness. This also treats
$r+s\equiv0\pmod4$, correcting the exclusion of that case in
\cite[Proposition 4.18 and Table 3]{Rubio}.
\end{proof}

\begin{lemma}\label{lem:spin-components}
Suppose $r\ne s$ and $rs$ is even. On $P=\Hom(W,U)$ the maximal
multiplier group is the effective image of
\[
 K_0=S(\Or(U)\times\Or(W))
   =\{(g,h):\det g\det h=1\}.
\]
Its identity component is the image of $\SO(U)\times\SO(W)$,
and the nontrivial quotient character is
$\chi_0(g,h)=\det g=\det h$.
\end{lemma}
\begin{proof}
This is the calculation underlying \cite[Lemma 6.8 and Proposition
6.9]{LXYv1}. If $r+s$ is odd, the Spin center is $\{\pm1\}$.
If it is even, the volume multipliers are excluded by the same
unequal-block argument as above. A multiplier therefore maps to a
block-diagonal element of $\SO(E)$, that is, an element of $K_0$.
The fixed Lie algebra has connected image $\SO(U)\times\SO(W)$.
For normalized nonisotropic $u\in U$, $w\in W$, the element
$g_0=uw$ of the Clifford algebra belongs to $\Spin(E)$ and satisfies
$\theta(g_0)=-g_0$. It induces a reflection in each block. Any two
multiplier-$-1$ elements differ by a fixed element, proving that these
are all the multiplier actions. Finally, the kernel of the orthogonal
action on $P$ is contained in $\{(I,I),(-I,-I)\}$. The latter lies
in $K_0$ only when $r+s$ is even; in the present even-product range
both dimensions are then even, so it belongs to the identity component.
The quotient and its character are therefore as asserted.
\end{proof}

Put $\widehat K=\Or(U)\times\Or(W)$. The index-two decomposition
of \cite[Lemma 6.11]{LXYv1} has the following direct proof. A reflection
in $U$ normalizes $K_0$ and acts involutively on
$\cS^*(P)^{K_0,\chi_0}$. Its negative eigenspace extends $\chi_0$
by $\det_U\boxtimes\one$, and its positive eigenspace extends it
by $\one\boxtimes\det_W$. Consequently
\begin{equation}\label{eq:index-two}
 \cS^*(P)^{K_0,\chi_0}
 =\cS^*(P)^{\widehat K,\det_U\boxtimes\one}
  \oplus\cS^*(P)^{\widehat K,\one\boxtimes\det_W}.
\end{equation}

\begin{corollary}\label{cor:Spin}
The space $\cS^*(P)^{K_0,\chi_0}$ is zero. In the residual Spin range
$|r-s|>2$, $rs$ even, the finite-component datum is componentwise regular.
\end{corollary}
\begin{proof}
Both summands in \eqref{eq:index-two} vanish by
Theorem~\ref{thm:ambient}\textup{(iii)}. Lemma~\ref{lem:spin-components}
identifies this with the only possible nontrivial multiplier character.
We have proved the consequent of \eqref{eq:componentwise}.
\end{proof}

The unequal odd--odd blocks are pleasant by Lemma~\ref{lem:odd-spin}.
The pairs with $|r-s|\leq2$, including the balanced blocks with a
possible further exchange, are treated by the nice-factor argument
below. Thus Rubio's list \cite[Lemma 4.21 and Table 3]{Rubio} and the
two lemmas above give the residual range directly. The accidental
isomorphisms $B_1=A_1$, $B_2=C_2$, $D_2=A_1+A_1$, $D_3=A_3$
reduce the small ranks to the same nice, pleasant, or exchanged blocks.

\input{evii_direct}

\section{Regularity and multiplicity one}\label{sec:consequences}

This section completes the proof of Theorem~\ref{thm:main} and derives
Corollary~\ref{cor:multiplicity}. We first establish the nice-factor input and combine
the factorwise results to prove regularity for connected complex
reductive symmetric pairs.  We then deduce the Gelfand--Kazhdan and
multiplicity-one statements, and finally prove the finite-component
transfer needed for the disconnected case.  The product and
intermediate-component arguments, as well as the Chevalley-involution
and finite central descent steps, follow the framework of
\cite[\S\S3 and~8]{LXYv1}.

\subsection{The nice factors}\label{sec:nice-factors}

Following Sekiguchi \cite{Sek85}, a semisimple
symmetric Lie algebra pair
$\mathfrak g=\mathfrak h\oplus\mathfrak p$
is called \emph{nice} if
\[
  \dim_{\mathbb C}\mathfrak g_\alpha
  +\dim_{\mathbb C}\mathfrak g_{2\alpha}
  \leq 2
  \qquad(\alpha\in\Sigma),
\]
where $\Sigma$ is the restricted root system
associated with a Cartan subspace
$\mathfrak a\subseteq\mathfrak p$, and
$\mathfrak g_\alpha$ denotes the corresponding
restricted root space. We set
$\mathfrak g_{2\alpha}=0$ if $2\alpha\notin\Sigma$.
A $\widetilde\theta$-stable simple factor
of the simply connected derived group
of the canonical cover is called nice
when its associated symmetric Lie algebra
pair is nice.

For a semisimple symmetric Lie algebra $\fg=\fh\oplus\fp$ and a nilpotent $e\in\fp$, choose a normal triple $(e,d,f)$, with $d\in\fh$, $e,f\in\fp$, and $[d,e]=2e$, $[d,f]=-2f$, $[e,f]=d$; see \cite[Lemma 7.1.11]{AG}. Write $\fp^e=\ker(\operatorname{ad}e|_\fp)$. The element $e$ is \emph{$\fp$-distinguished} if $\fp^e$ contains no nonzero semisimple element. Define, using complex dimensions and traces,
\begin{equation}\label{eq:nice-delta}
 \delta_\fp(e):=\tr(\operatorname{ad}d|_{\fp^e})
                  +2\dim\fp^e-\dim\fp.
\end{equation}
Thus if $\lambda_1,\ldots,\lambda_q$ are the highest weights on $\fp^e$, then $\delta_\fp(e)=\sum_i(\lambda_i+2)-\dim\fp$. 

The pairs considered below are nice in Sekiguchi's
sense; see \cite[Lemma~4.21]{Rubio}.
The distinguished-nilpotent inequality needed
for our argument is $\delta_{\mathfrak p}(e)>0$ for every $\mathfrak p$-distinguished nilpotent
element $e$:

\begin{proposition}[Distinguished-nilpotent inequality for the nice factors]\label{prop:nice-inequality}
For each of the pairs
\begin{gather*}
 (\mathfrak{sl}_n,\mathfrak{so}_n)\quad(n\geq2),\qquad
 (\mathfrak{sl}_{2n},\mathfrak{s}(\mathfrak{gl}_n\oplus\mathfrak{gl}_n)),\qquad
 (\mathfrak{sp}_{2n},\mathfrak{gl}_n),\\
 (\mathfrak{so}_{2n+k},\mathfrak{so}_{n+k}\oplus\mathfrak{so}_n)
       \quad(k=0,1,2),\qquad
 (\mathfrak e_7,\mathfrak{sl}_8),
\end{gather*}
with $n\geq1$ in the other classical cases, every $\fp$-distinguished nilpotent $e$ satisfies $\delta_\fp(e)>0$. The toral BDI case $n=1,k=0$ has no such element and contributes zero after passage to $Q(\fp)$.
\end{proposition}

\begin{proof}
The all-rank classical calculations and the finite root calculation for $E_7/A_7$ are given in Appendix~\ref{app:nice}.
\end{proof}

The positivity assertion is due to Sekiguchi
\cite[Theorem~6.3]{Sek85}. Appendix~\ref{app:nice} provides an
explicit verification of the numerical input used here, including the
AIII formulas \eqref{eq:AIII-contribution}--\eqref{eq:AIII-table}
and the all-rank CI and BDI calculations.

In Aizenbud's notation \cite[Definitions 5.1.33 and 5.1.38]{Aiz},
\[
 \operatorname{def}(e)=\tr(\operatorname{ad}d|_{\fh^e})-\dim\fp
                     =-\delta_\fp(e).
\]
Thus the required positivity $\delta_\fp(e)>0$ is exactly the negative distinguished defect used in the Fourier-support argument. For the sign identity, observe that the maps $\operatorname{ad}e:\fh\to\fp$ and $\operatorname{ad}e:\fp\to\fh$ are transpose up to sign for the invariant form, so
\begin{equation}\label{eq:kernel-difference}
 \dim\fp^e-\dim\fh^e=\dim\fp-\dim\fh.
\end{equation}
Also $\tr(\operatorname{ad}d|_{\fg^e})=\dim\fg-\dim\fg^e$, by the irreducible $\mathfrak{sl}_2$ decomposition. Adding the two definitions and using \eqref{eq:kernel-difference} gives zero. Restriction from complex to real scalars doubles both quantities and preserves their signs. For a reductive pair the calculation is made on the semisimple odd summand, identified with $Q_L(\fp)$; the anti-invariant center has been removed by Lemma~\ref{lem:cover-center}.

\begin{corollary}\label{cor:nice-component}
Let $L$ be the connected effective fixed group of a nontrivial semisimple pair in Proposition~\ref{prop:nice-inequality}, and put $V=Q_L(\fp)$. Let $\beta$ be the restriction to $V$ of the Killing form of $\fg$. We use the real pairing $\operatorname{Re}\beta$ and the nontrivial additive character $\psi(z)=\exp(2\pi i\operatorname{Re}z)$ of $\C$, so that $\mathcal F_\beta$ has kernel $\psi(\beta(v,w))$ and self-dual Haar measure. Then
\begin{equation}\label{eq:nice-Fourier}
 \left.
 \begin{gathered}
 T\in\cS^*(V)^L,\\
 \supp T,\ \supp\mathcal F_\beta T\subseteq\Gamma_L(\fp)
 \end{gathered}
 \right\}\quad\Longrightarrow\quad T=0.
\end{equation}
Consequently, for a finite extension \(J\) of \(L\) preserving
\(\beta\) and a character \(\lambda\) of \(J\) trivial on \(L\),
\[
\mathcal S^*(R_L(\mathfrak p))^{J,\lambda}=0
\Longrightarrow
\mathcal S^*(V)^{J,\lambda}=0.
\]
In particular, for the group
\(J=\operatorname{Ad}(\mathcal A_{\widetilde\theta})|_V\) attached in
Section~\ref{sec:finite-components} to a \(\widetilde\theta\)-stable simple factor of the
canonical cover, the datum $(L,J/L,1,V)$ is componentwise regular.
\end{corollary}

\begin{proof}
Proposition~\ref{prop:nice-inequality} gives negative distinguished defect. We apply the orbitwise method of \cite[Theorem~5.2.4]{Aiz} to the connected effective group $L$, using the real homogeneity theorem.\footnote{We use the orbitwise argument for the connected effective group $L$, with an adapted distribution supplied by the real homogeneity theorem.} The $H(\C)$-orbit filtration refines to finitely many $L$-orbits. Passing to the effective image preserves the tangent and normal actions and the stabilizer modulus along the normal-triple cocharacter. Regard $V$ as a real vector space and put $B_{\mathbb R}=\operatorname{Re}\beta$.
Let $\mathcal L$ be the space of $L$-invariant distributions $T$ such that
\[
\operatorname{supp}T,\,
\operatorname{supp}\mathcal F_\beta T
\subseteq \Gamma_L(\mathfrak p).
\]
Write
\[
Z(B_{\mathbb R})
=
\{v\in V_{\mathbb R}:B_{\mathbb R}(v,v)=0\}.
\]
Since $\Gamma_L(\mathfrak p)\subseteq Z(B_{\mathbb R})$, we have
$\mathcal L\subset
\mathcal S^*_{V_{\mathbb R}}\bigl(Z(B_{\mathbb R})\bigr)$. Here
\(\mathcal S^*_{V_{\mathbb R}}(Z(B_{\mathbb R}))\)
denotes the space of Schwartz distributions on \(V_{\mathbb R}\)
supported on \(Z(B_{\mathbb R})\).
By our choice of additive character,
$\mathcal F_\beta=\mathcal F_{B_{\mathbb R}}$.
The space $\mathcal L$ is stable under $\mathcal F_{B_{\mathbb R}}$
and under multiplication by $B_{\mathbb R}(v,v)$: after Fourier
transform, the latter becomes a constant-coefficient differential
operator and hence does not enlarge support.
Thus, if $\mathcal L\neq0$, the Archimedean homogeneity theorem
\cite[Theorem~5.1.7]{AG}, applied over $\mathbb R$, gives a nonzero
distribution in $\mathcal L$ adapted to $B_{\mathbb R}$ in the sense of
\cite[Definition~5.1.6]{AG}.
We apply the nilpotent-orbit filtration to this adapted distribution. By \cite[Corollary~4.3.5]{Aiz}, distinguished
points are dense in the support of a distribution
whose support and Fourier support lie in the
nilpotent cone. Hence a non-distinguished orbit
cannot be a maximal support stratum. On a distinguished maximal orbit
let $(e,d,f)$ be a normal triple, put $n=\dim_{\C}V$, and write
$T_e=\tr(\operatorname{ad}d|_{\fh^e})$. For $r>0$, the real
adapted character is
\[
 \rho(r)T=r^{n+\epsilon}T,\qquad \epsilon\in\{0,1\},
 \qquad \rho(r)v=r^{-1}v.
\]
The real Weil factor is one on positive scalars. The slice
$e+V^f$ has positive Kazhdan weights $w_j$ of sum $2n-T_e$.
This follows from \eqref{eq:kernel-difference}, or from the determinant
of $\kappa_t=t^2\Ad(D_t^{-1})$ on the orbit tangent and its normal,
where $D_t$ is the normal-triple cocharacter with differential $d$.
Pulling back the invariant scalar generalized function by action
coordinates, its slice distribution under $\kappa_t$ must have
character
\[
 t^{2T_e-2n-2\epsilon}.
\]
Indeed, the real determinant is $t^{4n}$ and the scalar
character contributes $t^{-2(n+\epsilon)}$, and conversion to slice
density contributes $t^{-2(2n-T_e)}$. A delta jet on the slice,
with holomorphic and antiholomorphic weighted orders $A,A'$, has
character $t^{A+A'}$. Hence a nonzero germ would require
\begin{equation}\label{eq:nice-real-degree}
 A+A'=2(T_e-n-\epsilon)
       =2\operatorname{def}(e)-2\epsilon<0,
 \qquad A,A'\geq0,
\end{equation}
which is impossible. This is the homogeneity--Frobenius argument of
\cite[proof of Proposition 7.3.7]{AG}, applied over $\R$.
No maximal orbit remains, proving \eqref{eq:nice-Fourier}.

In the applications, the central-multiplier group satisfies $J\subseteq\Ad(G)|_\fp$ and therefore preserves the Killing-form restriction $\beta$. Now let $T\in\cS^*(V)^{J,\lambda}$. Since $\lambda|_L=1$, the distribution $T$ is $L$-invariant. If the open $(J,\lambda)$-space vanishes, $T$ is supported in $\Gamma_L(\fp)$. The Fourier transform commutes with $J$ by \cite[Lemma 5.1.8]{AG}, so $\mathcal F_\beta T$ belongs to the same character space and has the same support restriction. Equation~\eqref{eq:nice-Fourier} gives $T=0$.
This is the characterwise argument of \cite[Lemma~3.7]{LXYv1}.
\end{proof}

\input{global_consequences}

\appendix
\input{nice_factors}
\input{extension_obstructions}

\input{homogeneity_obstruction}

\bibliographystyle{amsalpha}
\bibliography{references}
\end{document}

%% file: finite_components.tex
\subsection{Central covers and effective data}\label{sec:finite-components}
The effective-component framework in this subsection is due to Li--Xiao--Yu
\cite[\S3]{LXYv1}. We describe the effective group actions and their behavior under
finite central quotients. Throughout this subsection the ambient
group is connected.

\begin{lemma}\label{lem:cover-center}
Multiplication is a finite central isogeny
\[
 m:\widetilde G=Z(G)^\circ\times G_{\rm der}^{\rm sc}\longrightarrow G.
\]
There is a unique lift on the derived factor of $\theta$, giving an
involution $\widetilde\theta$ for which $m$ is equivariant. Moreover,
if $\fg=\mathfrak z\oplus\fg_{\rm der}$ and
$\fp=\mathfrak z^-\oplus\fp_{\rm der}$, then
\[
 \fp^H=\fp^{H^\circ}=\mathfrak z^-,\qquad
 Q_H(\fp)=Q_{H^\circ}(\fp)\simeq\fp_{\rm der}.
\]
\end{lemma}
\begin{proof}
These are \cite[Lemmas 3.1--3.2]{LXYv1}; the proofs are as follows. Both characteristic subgroups $Z(G)^\circ$ and $G_{\rm der}$ are
$\theta$-stable. Functoriality of the simply connected semisimple cover
lifts the second restriction uniquely. Its square lifts the identity
and is therefore the identity. The product lift is equivariant, and
$G=Z(G)^\circ G_{\rm der}$ gives surjectivity and a finite central kernel.
Equivariance makes that kernel stable.

Every inner automorphism fixes $\mathfrak z^-$. Conversely, let
$x\in\fp_{\rm der}$ commute with $\fh\cap\fg_{\rm der}$. The Killing
form is nondegenerate on both eigenspaces of the semisimple algebra. For
$y\in\fp_{\rm der}$ and $a\in\fh\cap\fg_{\rm der}$, invariance gives
$B([x,y],a)=-B([x,a],y)=0$. Hence $[x,y]=0$. Thus $x$ is central in
$\fg_{\rm der}$ and is zero. This proves both fixed-vector equalities,
including the equality for the full subgroup.
\end{proof}

Work on $V=\fp_{\rm der}$. Put
\[
 \widetilde H=\widetilde G^{\widetilde\theta},\qquad
 \mathcal A_{\widetilde\theta}
   =\{g\in\widetilde G:g^{-1}\widetilde\theta(g)\in Z(\widetilde G)\},
 \qquad J=\Ad(\mathcal A_{\widetilde\theta})|_V.
\]

By \cite[Theorem~8.1]{Steinberg}, $\left(G_{\mathrm{der}}^{\mathrm{sc}}\right)^{\widetilde\theta}$
is connected, while the central-torus factor acts trivially on \(V\).
Hence
\[
L:=\operatorname{Ad}(\widetilde H)|_V
=
\operatorname{Ad}\!\left(
\left(G_{\mathrm{der}}^{\mathrm{sc}}\right)^{\widetilde\theta}
\right)\big|_V
\]
is connected. We next identify the identity component of \(J\). If \(X\in\operatorname{Lie}(\mathcal A_{\widetilde\theta})\), then differentiating
the defining condition $g^{-1}\widetilde\theta(g)\in Z(\widetilde G)$
at the identity gives $\widetilde\theta(X)-X\in\mathfrak z(\widetilde{\mathfrak g})$.
Writing \(X=X_{\mathfrak z}+X_{\mathrm{der}}\), we obtain $\widetilde\theta(X_{\mathrm{der}})=X_{\mathrm{der}}$. Since the central summand acts trivially on \(V\), $\operatorname{Lie}J
=
\operatorname{ad}
(\mathfrak g_{\mathrm{der}}^{\widetilde\theta})|_V
=
\operatorname{Lie}L$. Thus \(L\) is connected and has the same Lie algebra as \(J\), hence $L=J^\circ$. This $L$ is also the image of $H^\circ$ on $V$, since the differential
of $m$ identifies the fixed Lie algebras. The group $J$ is the effective
image of the fixed subgroup of the involution on the adjoint group;
it is reductive with finitely many components. The multiplier
$c(g)=g^{-1}\widetilde\theta(g)$ is a homomorphism on
$\mathcal A_{\widetilde\theta}$, because its values are central.
Its kernel is $\widetilde H$, so commutators have image in $L$.
Thus $A:=J/L$ is finite abelian. We retain $B_0:=\operatorname{im}(\widetilde H\to A)=1$ in the notation to allow
a specified initial intermediate subgroup. We refer to $(L,J,A,V),\ A=J/L$ as the effective datum. Together with a specified subgroup
\(B_0\le A\), the tuple $(L,A,B_0,V)$
will be called the associated finite-component datum. This construction and the
following quotient bookkeeping follow \cite[Lemmas 3.3--3.5]{LXYv1}.

Write $G_{\mathrm{der}}^{\mathrm{sc}}=\prod_{\alpha}G_\alpha$ and group the simple factors into orbits \(\Omega\) of
\(\widetilde\theta\).  For such an orbit put $G_\Omega=\prod_{\alpha\in\Omega}G_\alpha, V_\Omega:=\mathfrak p_\Omega
=\operatorname{Lie}(G_\Omega)^{-\widetilde\theta}$,
and let $L_\Omega
 =\operatorname{Ad}(G_\Omega^{\widetilde\theta})|_{V_\Omega},
J_\Omega
 =\operatorname{Ad}(A_{\widetilde\theta,\Omega})|_{V_\Omega}$,
where $A_{\widetilde\theta,\Omega}
=
\{g\in G_\Omega:
g^{-1}\widetilde\theta(g)\in Z(G_\Omega)\}$.
Set $A_\Omega=J_\Omega/L_\Omega, B_{0,\Omega}=1$. The tuple
\[
\mathcal D_\Omega
=
(L_\Omega,A_\Omega,B_{0,\Omega},V_\Omega)
\]
is the finite-component datum associated with the
\(\widetilde\theta\)-orbit \(\Omega\).

\begin{lemma}\label{lem:central-interface}
Let $F\subset Z(\widetilde G)$ be finite and $\widetilde\theta$-stable,
and let $G=\widetilde G/F$. The effective image of its fixed
subgroup is $J_B$, the inverse image in $J$ of the subgroup
\[
 B=\operatorname{im}\bigl(\widehat H_F\to A\bigr),
 \]
where $\widehat H_F:=\{g\in\widetilde G:c(g)\in F\}$. In particular $B_0\leq B\leq A$. Grouping the simple factors of
$G_{\rm der}^{\rm sc}$ into $\widetilde\theta$-orbits gives
\[
 V=\bigoplus_iV_i,\qquad L=\prod_iL_i,
 \qquad J=\prod_iJ_i,\qquad A=\prod_iA_i.
\]
The subgroup $B$ need not be a product. Every multiplier element on
$G$ lifts to one on $\widetilde G$.
\end{lemma}
\begin{proof}
The equation $\theta(m(g))=m(g)$ is exactly $c(g)\in F$.
Thus $\widehat H_F$ is the full preimage of $G^\theta$ and contains
$\widetilde H$. Its effective image contains $L$ and has image $B$ in
$A$, so it equals $J_B$: if $jL\in B$, choose $h\in\widehat H_F$
with that class; then $j\Ad(h)^{-1}|_V\in L$ already belongs to
the same image.

The multiplier equation on the cover is factorwise on the indicated
factor orbits. Their effective actions occur on separate summands;
the central torus acts trivially. This proves the product assertions.
A diagonal subgroup $F$ can impose relations between the multiplier
classes, which are retained in the possibly diagonal subgroup $B$.

Finally, if $\bar g\in G$ is a central multiplier and $g$ is any lift,
put $y=g^{-1}\widetilde\theta(g)$. For $x\in\widetilde G$,
$[y,x]\in F$. This is a morphism from a connected group to a finite
set, with value $1$ at the identity; it is identically $1$.
Thus $y$ is central, as required.
\end{proof}

\begin{lemma}\label{lem:finite-nullcone}
For $B\leq A$, one has
\[
 V^{J_B}=V^L=0,\quad Q_{J_B}(V)=Q_L(V)=V,\quad
 \Gamma_{J_B}(V)=\Gamma_L(V),\quad R_{J_B}(V)=R_L(V).
\]
\end{lemma}
\begin{proof}
The fixed-vector assertion follows from Lemma~\ref{lem:cover-center}.
Also $\C[V]^{J_B}=(\C[V]^L)^B$. The quotient
$V\sslash L\to(V\sslash L)\sslash B$ has fibers equal to finite
$B$-orbits: invariant functions separate distinct finite orbits, by
interpolation on their union followed by finite averaging. The image
of $0$ in $V\sslash L$ is fixed by $B$, so the fiber over its image
is that single point. Pulling back proves the null-cone assertion and
then the assertion about complements. This is the argument of
\cite[Lemma 3.5]{LXYv1}.
\end{proof}

\subsection{Componentwise regularity and intermediate components}\label{sec:assembly}
\begin{definition}[\cite{LXYv1}, Definition 3.6]\label{def:componentwise}
The datum $(L,A,B_0,V)$ is \emph{componentwise regular} if, for every
nontrivial character $\lambda:A\to\C^\times$ trivial on $B_0$,
\begin{equation}\label{eq:componentwise}
 \cS^*(R_L(V))^{J,\lambda}=0
 \quad\Longrightarrow\quad \cS^*(V)^{J,\lambda}=0.
\end{equation}
Here $\lambda$ is inflated to $J$.
\end{definition}

\begin{proposition}[Products]\label{prop:component-product}
A product of componentwise-regular data is componentwise regular.
\end{proposition}
\begin{proof}
This is essentially \cite[Proposition 3.9]{LXYv1}. We give the slicing
argument. For the product, $V=\bigoplus_iV_i$ and
$\Gamma_L(V)=\prod_i\Gamma_{L_i}(V_i)$. Write
$\lambda=\boxtimes_i\lambda_i$. If
$0\ne T\in\cS^*(V)^{J,\lambda}$, density of decomposable Schwartz
tests gives $f_i\in\cS(V_i)$ with
$T(f_1\otimes\cdots\otimes f_m)\ne0$. For each $i$ the continuous
slice
\[
 T_i(\varphi)=T(f_1\otimes\cdots\otimes f_{i-1}\otimes\varphi
                    \otimes f_{i+1}\otimes\cdots\otimes f_m)
\]
is nonzero and belongs to $\cS^*(V_i)^{J_i,\lambda_i}$.
For $\lambda_i\ne1$, the contrapositive of componentwise regularity
supplies a nonzero distribution $S_i$ on $R_{L_i}(V_i)$ in that character.

We may arrange that $S_i$ extends to a tempered distribution on $V_i$
with closed support disjoint from its null cone. Indeed, choose a compact
test on which it is nonzero, embed the affine quotient equivariantly
in a finite-dimensional $A_i$-module using invariant generators
vanishing at $0$, and choose an $A_i$-invariant smooth compact cutoff
equal to one on the image of that test and zero near the quotient
origin. Multiply $S_i$ by the pullback of this cutoff along the quotient map.
Let $p=(p_1,\ldots,p_k)$ be the chosen quotient coordinates on $V_i$,
put $\Gamma_i=\Gamma_{L_i}(V_i)$, and set $d=\max_j\deg p_j$.
On the support of the cutoff pullback, $\|p(x)\|\geq\varepsilon>0$.
Since $0\in\Gamma_i$, a nearest point $z\in\Gamma_i$ satisfies
$\|z\|\leq2\|x\|$. The polynomial mean-value estimate gives
\[
 \|p(x)\|\leq C(1+\|x\|)^{d-1}
                       \operatorname{dist}(x,\Gamma_i).
\]
Thus reciprocal distance to $\Gamma_i$ has polynomial growth on the
cutoff support. All derivatives of the pullback also have polynomial
growth, and the pullback vanishes in a neighborhood of $\Gamma_i$.
The Nash Schwartz estimates therefore give a continuous multiplication
map $\cS(V_i)\to\cS(R_{L_i}(V_i))$, whose transpose defines the
required extension. This localization is
\cite[Lemma 3.8]{LXYv1}; the Schwartz-space properties used here are
recorded in \cite{AGSchwartz}.

For $\lambda_i=1$ use $S_i=\delta_0$ instead. The external tensor
product of these distributions is nonzero and has character
$\lambda$. Since at least one $\lambda_i$ is nontrivial, its support
is closed in $V$ and contained in
$R_L(V)=V\setminus\prod_i\Gamma_{L_i}(V_i)$. Its restriction there
is nonzero. This proves the contrapositive of \eqref{eq:componentwise}.
\end{proof}

\begin{proposition}[Intermediate components]\label{prop:intermediate}
Suppose $(L,A,B_0,V)$ is componentwise regular, $B_0\leq B\leq A$,
and $a\in A$ satisfies $a^2\in B$. Choose a lift $\widetilde a\in J$ of $a$; on the
$J_B$-invariant distribution spaces below, invariance under
$\widetilde a$ is independent of this choice. Then
\[
 \cS^*(R_L(V))^{J_B}\subset\cS^*(R_L(V))^a
 \quad\Longrightarrow\quad
 \cS^*(V)^{J_B}\subset\cS^*(V)^a.
\]
\end{proposition}
\begin{proof}
This is essentially \cite[Proposition 3.10]{LXYv1}. If $a\in B$ the
assertion is immediate. Otherwise let $C=\langle B,a\rangle$.
The quotient $C/B$ has order two. Its sign character $\epsilon$
is trivial on $B$ and takes value $-1$ at $a$. The premise is
equivalent to vanishing of this character space on $R_L(V)$.
Every character of a subgroup of a finite abelian group extends to
the whole group. Adjoin one generator at a time:
if its least positive power lying in the old subgroup is $u^k$,
choose a $k$th root of the already prescribed value on $u^k$.
This extends the character consistently.
Finite character projectors therefore decompose the $\epsilon$-space
as the direct sum of the $A$-character spaces extending $\epsilon$.
Each extension is nontrivial and trivial on $B_0$, so
\eqref{eq:componentwise} annihilates each corresponding space on $V$.
This is precisely the asserted $a$-invariance.
\end{proof}

An element $g$ is admissible in the regularity definition if
$\Ad(g)$ commutes with $\theta$, normalizes the $H$-action with square
in that action, and preserves every closed $H$-orbit in $\fp$;
see \cite[Definitions 6.0.1 and 7.4.1--7.4.2]{AG}. Regularity is
\begin{equation}\label{eq:regularity}
 \cS^*(R_H(\fp))^H\subset\cS^*(R_H(\fp))^{\Ad(g)}
 \quad\Longrightarrow\quad
 \cS^*(Q_H(\fp))^H\subset\cS^*(Q_H(\fp))^{\Ad(g)}.
\end{equation}
By Lemmas~\ref{lem:central-interface}--\ref{lem:finite-nullcone}, an
admissible element on a central quotient gives exactly the situation
of Proposition~\ref{prop:intermediate}, with $a^2\in B$ and the full
fixed image $J_B$. Thus that proposition proves ordinary regularity
for every such quotient, including diagonal central couplings.
A pleasant factor has $J=L$ and hence no nontrivial character to test.

%% file: evii_direct.tex
\section{Direct componentwise regularity for EVII}\label{sec:EVII}
Let $G$ be simply connected of type $E_7$, with the EVII involution.
Write $\fp=J_{27}\oplus J_{27}^*$ and let $B:\mathfrak p\times\mathfrak p\to\mathbf C$ be the \(L\)-invariant symmetric nondegenerate bilinear form $B((x,\phi),(y,\psi))=\phi(y)+\psi(x)$. 
Here $J_{27}$ is the underlying vector space of the complex Albert
algebra, equipped with its $27$-dimensional minuscule $E_6$ action.
The connected effective group is
\[
L=\operatorname{im}(E_6^{\mathrm{sc}}\times\mathbb C^\times)
   \simeq
   (E_6^{\mathrm{sc}}\times\mathbb C^\times)/\mu_3,
\]
where $(h,t)(x,\phi)
=
\bigl(t\rho(h)x,\,
t^{-1}\rho^\vee(h)\phi\bigr)$. More precisely, the kernel is $\bigl\{(z,\rho(z)^{-1}):z\in Z(E_6^{\mathrm{sc}})\bigr\}
\simeq\mu_3$. Let \(C\) be the Chevalley involution
\[
C(H_i)=-H_i,\qquad C(E_\alpha)=E_{-\alpha},
\]
and put $\tau=C|_{\mathfrak p}$. Since \(C\) is inner in type \(E_7\) and commutes with \(\theta\),
the map \(\tau\) preserves \(B\), exchanges the two summands
\(J_{27}\) and \(J_{27}^*\), and represents the nontrivial effective
component.  Thus $J=L\rtimes\langle\tau\rangle, J/L\simeq\mu_2$. Let \(\chi\) denote the nontrivial character of \(J/L\).

\begin{theorem}\label{thm:EVII}
For this effective datum,
\[
 \cS^*(\fp)^{J,\chi}=0.
\]
Consequently, the associated EVII finite-component datum is componentwise regular.
\end{theorem}

\begin{remark}
The componentwise-regularity conclusion was stated in
\cite[Corollary~7.12]{LXYv1} using speciality. The revised argument
\cite[Proposition~8.1]{LXYv2} proves the vanishing of the nontrivial character
component under the nilpotent-support conditions.
Theorem~\ref{thm:EVII} establishes the stronger
full-space vanishing without these support restrictions. Appendix~\ref{app:homogeneity} constructs an invariant
tempered distribution with nilpotent support and nilpotent Fourier
support, proving that the EVII isotropy representation is not special.
That distribution lies in the trivial character space, consistently
with the theorem above.
Recall that an $L$-representation $V$ is
\emph{special} in the sense of
\cite[Definition~6.0.6]{AG} if there is no
nonzero $L$-invariant tempered distribution
on $Q_L(V)$ supported on its null cone whose
partial Fourier transforms, for every invariant
decomposition and invariant nondegenerate
symmetric form allowed in that definition,
are also supported on the null cone.

\end{remark}

Let $\mathcal N$ denote the nilpotent cone in $\mathfrak p$.
We write $\mathcal F_B$ for the Fourier transform
\[
(\mathcal F_B f)(w)
=
\int_{\mathfrak p}
f(v)\exp\!\bigl(2\pi i\,\operatorname{Re}B(v,w)\bigr)\,dv,
\]
where $dv$ is the self-dual Haar measure for the real pairing
$\operatorname{Re}B$. We use the same notation for its extension to
tempered distributions by duality.

\subsection{Semisimple slices}
Use the Bourbaki root diagram with edges
$13,34,45,56,67,24$ and for $\alpha=\sum_{i=1}^7 a_i\alpha_i$, set $\theta(E_\alpha)=(-1)^{a_7}E_\alpha$.
The three strongly orthogonal roots
\[
\beta_1=(2,2,3,4,3,2,1),\qquad
\beta_2=(0,1,1,2,2,2,1),\qquad
\beta_3=(0,0,0,0,0,0,1)
\]
give a Cartan subspace
\[
\mathfrak a
=
\operatorname{span}\{E_{\beta_i}+E_{-\beta_i}:1\le i\le3\}.
\]
Then \(\tau\) fixes \(\mathfrak a\)
pointwise. The restricted root system is $C_3$: the roots $\pm2\varepsilon_i$ have multiplicity
$1$, and $\pm\varepsilon_i\pm\varepsilon_j$ have multiplicity $8$.
The zero-restriction roots form $D_4$. These statements follow by
restricting the $126$ roots to the three coroots. For \(s\in\mathfrak a\), write
\[
\mathfrak g^s=\ker(\operatorname{ad}s),\qquad
\mathfrak h^s=\mathfrak g^s\cap\mathfrak h,\qquad
\mathfrak p^s=\mathfrak g^s\cap\mathfrak p.
\]
Thus \(\mathfrak g^s=\mathfrak h^s\oplus\mathfrak p^s\).

Every nonzero semisimple point is conjugate under \(L\) to one of the
following patterns, up to signed permutation. Distinct letters are
nonzero with distinct squares. The table records the derived algebra
\([\mathfrak g^s,\mathfrak g^s]\), the even centralizer
\(\mathfrak h^s\), and the odd centralizer \(\mathfrak p^s\).
Pair notation in the last column denotes the corresponding isotropy
module; the displayed trivial summands are the anti-invariant centers.
\begin{center}\small
\begin{tabular}{llll}
\toprule
Pattern&$[\fg^s,\fg^s]$&$\fh^s$&$\fp^s$\\\midrule
$(c,d,e)$&$D_4$&$D_4$&$\C^3$\\
$(c,d,0)$&$D_4+A_1$&$D_4+\C$&$\C^2\oplus(A_1/\C)$\\
$(c,c,d)$&$D_5$&$B_4$&$\C^2\oplus(D_5/B_4)$\\
$(c,c,0)$&$D_5+A_1$&$B_4+\C$&$\C\oplus(D_5/B_4)\oplus(A_1/\C)$\\
$(c,0,0)$&$D_6$&$D_5+\C$&$\C\oplus(D_6/(D_5+\C))$\\
$(c,c,c)$&$E_6$&$F_4$&$\C\oplus(E_6/F_4)$\\\bottomrule
\end{tabular}
\end{center}
The dimensions $(\dim\fg^s,\dim\fh^s,\dim\fp^s)$ are respectively
\[
 (31,28,3),\ (33,29,4),\ (47,36,11),\
 (49,37,12),\ (67,46,21),\ (79,52,27).
\]
To obtain the second column, retain the restricted roots
vanishing on the pattern and the zero-restriction $D_4$ roots.
Kernels of \(\operatorname{ad}s\) on the even and odd root spaces give
the two remaining dimensions.  Together with the semisimple root
systems in the second column, these fixed dimensions distinguish the
indicated involution types. A central element of $\fg^s$
centralizes the Cartan subspace and the zero-restriction $D_4$;
it therefore lies in that Cartan subspace. The center is odd and
has exactly the displayed dimension. For example, let \(s=(c,0,0)\) with \(c\neq0\).
The vanishing restricted roots are
\[
\pm2\varepsilon_2,\qquad
\pm2\varepsilon_3,\qquad
\pm\varepsilon_2\pm\varepsilon_3.
\]
Together with the zero-restriction \(D_4\)-roots, their ordinary-root
spaces form a root subsystem of type \(D_6\).  Hence
\[
[\mathfrak g^s,\mathfrak g^s]\simeq D_6.
\]
The restricted involution has fixed algebra \(D_5+\mathbb C\).
Since \(\mathfrak g^s\) has rank \(7\), its center is one-dimensional;
by the preceding observation this center lies in \(\mathfrak p^s\).
Therefore
\[
\mathfrak p^s
\simeq
\mathbb C\oplus\bigl(D_6/(D_5+\mathbb C)\bigr),
\]
as in the fifth row of the table. By Chevalley restriction for symmetric spaces,
\[
\mathbf C[\mathfrak p]^L
\simeq
\mathbf C[\mathfrak a]^{W(C_3)}
=
\mathbf C[P_2,P_4,P_6],
\]
where, on \(\mathfrak a\) with coordinates \((c,d,e)\), the basic
invariants may be chosen to restrict to
\[
c^2+d^2+e^2,\qquad
c^2d^2+c^2e^2+d^2e^2,\qquad
c^2d^2e^2,
\]
of degrees \(2,4,6\), respectively; see \cite{KR}.
Their common zero set is the nilpotent cone \(\mathcal N\).

We next determine the component action. The centralizer $G_s$
of a semisimple Lie algebra element in the connected group $G$ is
connected reductive. A representative of a multiplier fixing $s$
lies in $G_s$. Lemma~\ref{lem:central-interface} lifts it to a
central multiplier on the canonical cover of $G_s$. Its inner action
is factorwise on the simple factors above and trivial on the center.
For $E_6/F_4$, the outer involution inverts the center $\mu_3$.
If $\theta(g)=zg$, a central correction $c$ makes $cg$ fixed,
because $c\mapsto\theta(c)c^{-1}=c^{-2}$ is onto $\mu_3$.
Its action is therefore in the connected fixed group $F_4$.
For $D_5/B_4$, the blocks are $(9,1)$ and the same conclusion is
Lemma~\ref{lem:odd-spin}. The two remaining modules are
\[
A_1/\mathbb C:\operatorname{Mat}_{1,2}(\mathbb C),
\qquad
D_6/(D_5+\mathbb C):\operatorname{Mat}_{10,2}(\mathbb C).
\]
Their connected effective groups are the images of
\(SO_r\times SO_2\), for \(r=1,10\), and by
Lemma~\ref{lem:spin-components} a possible extra component has
determinants \((-1,-1)\).  The index-two decomposition \eqref{eq:index-two} reduces
its odd space to the two single-determinant spaces, which vanish by
Theorem~\ref{thm:ambient}\textup{(iii)}.

For each of the six semisimple patterns, the image of \(J_s\) on
\(\mathfrak p^s\) is the product of the connected effective images on
the nonorthogonal factors with, when present, the single orthogonal
block just described.  There is at most one such orthogonal block, and
\(\chi\) is trivial on all connected factors.  Hence, after testing
against decomposable functions in the connected factors, the remaining
distribution lies in the corresponding odd space of the orthogonal
block and is zero.  Thus, for every nonzero semisimple \(s\),
\[
\mathcal S^*(\mathfrak p^s)^{J_s,\chi}=0.
\]
Since $\tau$ fixes our Cartan subspace, it belongs to $J_s$ and
has character $-1$.

\begin{lemma}\label{lem:EVII-open}
One has $\mathcal D'(\fp\setminus\mathcal N)^{J,\chi}=0$ and
$\cS^*(\fp\setminus\mathcal N)^{J,\chi}=0$.
Here $\mathcal D'(X)$ is the continuous dual of $C_c^\infty(X)$.
\end{lemma}
\begin{proof}
For a closed orbit through a nonzero semisimple \(s\), the normal
representation is \(\mathfrak p^s\), with reductive stabilizer \(J_s\).
The character \(\chi|_{J_s}\) is algebraic.  Hence Theorem~4.0.2 of
\cite{AG}, applied to \(X=\mathfrak p^s\), upgrades the vanishing
$\mathcal S^*(\mathfrak p^s)^{J_s,\chi}=0$
proved above to $\mathcal D'(\mathfrak p^s)^{J_s,\chi}=0$.
The basic invariants $P_2,P_4,P_6$ are $\tau$-fixed,
because $\tau$ fixes a Cartan subspace pointwise. Their nonvanishing
principal affine opens cover $\fp\setminus\mathcal N$.
Every closed orbit in such an open is a nonzero semisimple orbit:
the closed orbit in an ambient orbit closure retains the same
nonzero invariant value. Apply \cite[Theorem 3.1.1]{AG} separately
to these smooth affine opens. This proves ordinary-distribution
vanishing on their union, and hence Schwartz vanishing there.
\end{proof}

\subsection{The orbit-20 node}

For a diffeomorphism \(g\), we use the push-forward convention $(g_*T)(\varphi)=T(\varphi\circ g)$.
\begin{lemma}\label{lem:EVII-node}
Let \(\C^\times\) act on \(\C^2\) by $t(u,v)=(tu,t^{-1}v)$, and let $\sigma(u,v)=(v,u)$. Let \(T\) be a distribution defined in a neighborhood of \(0\), with $\operatorname{supp}T\subset\{uv=0\}$ near \(0\).  If, as a germ at \(0\),
\[
t_*T=T\qquad (t\in\C^\times),
\qquad
\sigma_*T=-T,
\]
then \(T=0\) in a neighborhood of \(0\). More generally, the same conclusion holds on
\(\C^2\times\C^m\) when both actions are trivial on the
\(\C^m\)-factor and $\operatorname{supp}T\subset\{uv=0,\ w=0\}$ near the origin.
\end{lemma}
\begin{proof}
The underlying two-axis mechanism is the complex analogue of
Prasad's two-axis lemma \cite[Lemma~4.6]{P90}; the present statement also allows arbitrary transverse orders and fixed auxiliary coordinates.

On a relatively compact neighborhood, a distribution has finite order.
Hence a distribution supported on the smooth branch \(v=0\), away from
the origin, has a finite transverse expansion
\[
T=\sum_{k,\ell}\mu_{k\ell}(u)\otimes
\partial_v^k\partial_{\bar v}^{\ell}\delta(v).
\]
Indeed, a finite-order distribution supported on \(v=0\) depends only
on a finite normal Taylor jet of the test function along that branch.
Thus, on the punctured \(u\)-axis,
\[
T=\sum_{k,\ell}
\mu_{k\ell}(u)\otimes
\partial_v^k\partial_{\bar v}^{\ell}\delta(v),
\]
with only finitely many nonzero terms.  The coefficients
\(\mu_{k\ell}\) are uniquely determined.

For the term $\mu_{k\ell}(u)\otimes
\partial_v^k\partial_{\bar v}^{\ell}\delta(v)$,
infinitesimal invariance under the complex torus gives, in the
distributional sense,
\[
u\partial_u\mu_{k\ell}=-(k+1)\mu_{k\ell},
\qquad
\bar u\partial_{\bar u}\mu_{k\ell}=-(\ell+1)\mu_{k\ell}.
\]
On the connected punctured disc these equations have, up to a scalar,
the solution
\[
\mu_{k\ell}(u)
=
u^{-k}\bar u^{-\ell}|u|^{-2}\,dA(u).
\]
Hence the restriction to the punctured \(u\)-axis is a finite linear
combination of
\[
A^+_{k\ell}
=
u^{-k}\bar u^{-\ell}|u|^{-2}dA(u)
\otimes
\partial_v^k\partial_{\bar v}^{\ell}\delta(v),
\qquad k,\ell\ge0.
\]
Oddness under \(\sigma\) determines the corresponding terms on the
punctured \(v\)-axis with opposite coefficients.
Regularize $A^+_{k\ell}$ by $|u|^{2z}$. Taylor subtraction gives a
simple pole at $z=0$, with residue
\begin{equation}\label{eq:EVII-residue}
 R_{k\ell}=\frac{\pi(-1)^{k+\ell}}{k!\ell!}
 \partial_u^k\partial_{\bar u}^{\ell}
 \partial_v^k\partial_{\bar v}^{\ell}\delta_{(0,0)}.
\end{equation}
Indeed, for \(\Re z\gg0\), let $A^+_{k\ell}(z)=|u|^{2z}A^+_{k\ell}$. If \(\varphi\) is a test function, put $F_{k\ell}(u) =
\partial_v^k\partial_{\bar v}^{\ell}\varphi(u,0)$.
Then
\[
\langle A^+_{k\ell}(z),\varphi\rangle
=
(-1)^{k+\ell}
\int_{\C}
u^{-k}\bar u^{-\ell}|u|^{2z-2}
F_{k\ell}(u)\,dA(u).
\]
Taylor-expand \(F_{k\ell}\) at \(0\).  For a term
\(u^p\bar u^q\), polar coordinates give the factor
\[
r^{2z+p+q-k-\ell-1}
e^{i((p-k)-(q-\ell))\theta}.
\]
The angular integral is nonzero only if $p-q=k-\ell$, while a pole at \(z=0\) requires $p+q=k+\ell$. Hence only \(p=k,\ q=\ell\) contributes to the residue.
Since
\[
\int_0^\varepsilon r^{2z-1}\,dr
=
\frac{\varepsilon^{2z}}{2z},
\]
the radial residue is \(1/2\), and the angular integral contributes
\(2\pi\).

Let \(X\) generate push-forward by $r(u,v)=(ru,r^{-1}v),\qquad r>0$. The two regularized branches transform by \(r^{-2z}\) and \(r^{2z}\),
respectively.  Write
\[
A^+_{k\ell}(z)
=
\frac{R_{k\ell}}{z}
+
P_{k\ell}
+
O(z),
\qquad
P_{k\ell}:=\operatorname{FP}A^+_{k\ell}.
\]
Since $r^{-2z}=1-2z\log r+O(z^2)$, comparison of constant Laurent coefficients gives $r_*P_{k\ell}-P_{k\ell}
=
-2(\log r)R_{k\ell}$. Differentiating at \(r=1\) yields
\[
XP_{k\ell}=-2R_{k\ell}.
\]
Likewise, for the exchanged branch,
\[
X\,\operatorname{FP}(\sigma_*A^+_{k\ell})
=
2R_{k\ell}.
\]
An odd branch combination with residue $R$ has anomaly $-4R$.
Every other extension differs by a finite delta jet \(Q\) at the origin.
A monomial point-supported jet $\partial_u^a\partial_v^b
\partial_{\bar u}^c\partial_{\bar v}^d\delta_0$ has \(X\)-weight $a+c-b-d$. Hence \(X\) acts diagonally on the space of point-supported jets, and
its image has zero projection to the zero-weight eigenspace.
On the other hand, by \eqref{eq:EVII-residue}, each \(R_{k\ell}\) has \(X\)-weight 0. Therefore the equation \(XQ=4R\) can hold only if \(R=0\). Since the residues are linearly independent,
\(R=0\) forces every punctured-branch coefficient to vanish.

It remains to treat a distribution supported at the origin.  A monomial
jet
\[
J_{a,b,c,d}
=
\partial_u^a\partial_v^b
\partial_{\bar u}^c\partial_{\bar v}^d\delta_0
\]
transforms under \(t\in\C^\times\) by the character $t^{\,a-b}\bar t^{\,c-d}$. Hence invariance under the full complex torus forces $a=b$ and $c=d$. For these indices one has
\[
\sigma_*J_{a,a,c,c}=J_{a,a,c,c},
\]
so every torus-invariant point-supported jet is \(\sigma\)-even.
Thus no nonzero \(\sigma\)-odd point-supported distribution remains.

If additional coordinates fixed by both actions are present and are
constrained to vanish on the support, expand in their holomorphic and
antiholomorphic delta derivatives.  Each coefficient is a distribution
in the \((u,v)\)-variables satisfying the same hypotheses, so the
preceding argument applies coefficientwise.  This proves the lemma.
\end{proof}

For the nilpotent orbits we use the connected complex isotropy-orbit
classification of \cite[Theorem 3 and Table XIII]{Dok88}.
We label an orbit by the Kostant--Rallis characteristic \(d\) of a
normal \(\mathfrak{sl}_2\)-triple, chosen \(W(E_6)\)-dominant.
Orbit numbers below refer to Table XIII; these are \(L\)-orbits on
\(\mathfrak p\).  Explicit triples can also be compared with
\cite[Table 6]{Dok99}.

We first single out orbit \(20\), a \(50\)-dimensional orbit whose
transverse slice contains the two \(51\)-dimensional branches \(21\)
and \(22\).  The complete orbit and weight table is given in Section~\ref{subsec:EVII-quadratic}.

\begin{lemma}\label{lem:EVII-slice20}
At orbit 20 there is a normal $\mathfrak{sl}_2$-triple $(e,d,f)$,
so that
\[
[d,e]=2e,\qquad [d,f]=-2f,\qquad [e,f]=d,
\]
with $e,f\in\mathfrak p$ and $d\in\mathfrak h$, and a representative
$\sigma_e\in J\setminus L$ fixing the triple for which the slice
$e+\mathfrak p^f$ has coordinates $(b,u,v,a)$. Let
\[
A:=Z\bigl(Z_L(e,d,f)\bigr)^\circ .
\]
This is a one-dimensional torus.  After identifying its effective
image on the slice with \(\mathbb C^\times\), write its parameter as
\(t\). Then
\[
 t(b,u,v,a)=(b,tu,t^{-1}v,a),\quad
 \sigma_e(b,u,v,a)=(b,v,u,a),\quad
 (e+\fp^f)\cap\mathcal N=\{a=b=uv=0\}.
\]
For the normal triple, let $\lambda_d:\mathbb C^\times\longrightarrow L$
be the associated cocharacter, normalized by $\operatorname{Ad}(\lambda_d(s))x=s^j x$ if $[d,x]=jx$. Put \(D_s=\lambda_d(s)\), and let
\[
\kappa_s(x)=s^2\operatorname{Ad}(D_s^{-1})x,
\qquad s\in\mathbb C^\times,
\]
be the Kazhdan action.  Its weights on the displayed slice coordinates
are \((8,6,6,4)\). The two punctured branches
\[
\{a=b=v=0,\ u\neq0\},
\qquad
\{a=b=u=0,\ v\neq0\},
\]
are the two nilpotent orbits \(21\) and \(22\).
\end{lemma}

\begin{proof}
Here are explicit data for the calculation, with the root-basis
convention specified at the end of this section. Number Cartan
vectors $0,\ldots,6$ and root vectors $7,\ldots,132$; write $v_i$
for the resulting basis. A triple has
\begin{align*}
 e={}&-v_{56}+v_{111}-4v_{62}+4v_{101}-4v_{63}+4v_{103}
             -2v_{106}+2v_{59},\\
 f={}&\tfrac38v_{33}-\tfrac34v_{38}-v_{40}+\tfrac32v_{43}
       -\tfrac32v_{122}+\tfrac34v_{125}+v_{126}-\tfrac38v_{127},\\
 d={}&6v_0+8v_1+10v_2+14v_3+10v_4+6v_5.
\end{align*}
Its characteristic is $(2,2,0,0,0,2,-6)$. This characteristic determines the displayed \(d\) in the fixed
Cartan.  The displayed \(e\) and \(f\) lie in the \(+2\) and
\(-2\) \(\operatorname{ad}d\)-weight spaces of \(\mathfrak p\), respectively, and direct calculation gives $[e,f]=d$. Thus they form a normal \(\mathfrak{sl}_2\)-triple. An ordered basis of
$\fp^f$ is
\begin{align*}
 n_b={}&v_7+v_{132},\qquad
 n_u=\tfrac14v_{14}-\tfrac12v_{20}+v_{26},\\
 n_v={}&-v_{129}-\tfrac12v_{130}+\tfrac14v_{131},\\
 n_a={}&-v_{33}+2v_{38}+\tfrac83v_{40}-4v_{43}
          +4v_{122}-2v_{125}-\tfrac83v_{126}+v_{127}.
\end{align*}
Solving $[f,n]=0,\ n\in\mathfrak p$ shows that these four vectors form a basis of \(\mathfrak p^f\).
Their \(\operatorname{ad}d\)-weights are
\[
-6,\,-4,\,-4,\,-2,
\]
so their Kazhdan weights are
\[
8,\,6,\,6,\,4.
\] 
The triple centralizer has dimension $9$. Its one-dimensional connected center
acts on this basis with weights $(0,1,-1,0)$. To obtain the component,
apply the Chevalley map $C(v_i)=-v_i$ on the Cartan,
$C(E_\alpha)=E_{-\alpha}$, then a representative of $w_0(E_6)$,
then $\theta$, and finally the toral sign
$E_\alpha\mapsto(-1)^{\alpha_4}E_\alpha$.
The representative is the product of the simple root representatives
in the word
\[
 (123456)^4(12345)(1234)(13)(1),\qquad
 n_i=\exp(\operatorname{ad}E_{\alpha_i})
      \exp(\operatorname{ad}E_{-\alpha_i})
      \exp(\operatorname{ad}E_{\alpha_i}),
\]
applied from left to right. Substitution in the brackets shows that
the composite fixes $e,d,f$, squares to one, exchanges the two
$27$-spaces, and has on the displayed normal basis the matrix
$\diag(1,\left(\begin{smallmatrix}0&1\\1&0\end{smallmatrix}\right),1)$.
Its conjugation inverts the central torus.

For $x=e+b n_b+u n_u+v n_v+a n_a$, write \(x=x_++x_-\) according to $\mathfrak p=J_{27}\oplus J_{27}^*$ and put
\[
D_x=\operatorname{ad}(x_+)\operatorname{ad}(x_-)|_{J_{27}}.
\]
Let \(N_+\) and \(N_-\) denote the invariant cubic norms on
\(J_{27}\) and \(J_{27}^*\), normalized as below.  Direct calculation
in the root basis gives
\begin{equation}\label{eq:EVII-slice-invariants}
 \tr D_x=-960a,\quad
 \tr D_x^2=3072b+\tfrac{413696}{9}a^2,\quad
 N_+(x_+)=-16u,\quad N_-(x_-)=16v.
\end{equation}
The cubic norms are the unique, up to scalar, $E_6$-invariant
cubics on the two minuscule modules. The restrictions show that
\[
\operatorname{tr}D_x,\qquad
\operatorname{tr}D_x^2,\qquad
N_+N_-
\]
are independent.  They have degrees \(2,4,6\), respectively, so
Chevalley restriction shows that they form basic invariants, up to an
invertible weighted triangular change of generators.  Hence the
nilpotent cone cuts the slice out by
\[
a=0,\qquad b=0,\qquad uv=0.
\]

On the two punctured branches \(u\ne0,\ v=0\) and
\(u=0,\ v\ne0\), exactly one of the two cubic norms is nonzero.
The stabilizer has dimension \(28\); since \(\dim L=79\), both
branches are \(51\)-dimensional orbits.  The classification identifies
them with orbits \(21\) and \(22\).
\end{proof}

\begin{lemma}\label{lem:EVII-exclude20}
A distribution in $\cS^*(\fp)^{J,\chi}$ supported on $\mathcal N$
has support disjoint from orbits $20,21,22$.
\end{lemma}
\begin{proof}
Use invariant ambient density and pull the scalar generalized
function back under the local action submersion
$L\times(e+\fp^f)\to\fp$. Connected invariance makes it independent
of the group variable. The slice germ is invariant under the
triple-center torus and odd under $\sigma_e$. The effective torus action has determinant \(1\) on the slice, so the
constant slice density is invariant and no additional density character
appears when passing from scalar generalized functions to distributions. Its support is $a=b=uv=0$.
Hence lemma~\ref{lem:EVII-node} shows that the slice germ vanishes. Injectivity of submersion
pullback excludes orbit $20$. Both other orbits have orbit $20$
in their closure by Lemma~\ref{lem:EVII-slice20}; a closed invariant
support meeting either would contain that orbit. They are excluded too.
\end{proof}

\subsection{Quadratic operators and normal degrees}\label{subsec:EVII-quadratic}
Let
\[
 \mathcal M_\chi=\{T\in\cS^*(\fp)^{J,\chi}:
    \supp T,\supp\mathcal F_BT\subseteq\mathcal N\}.
\]
Lemma~\ref{lem:EVII-open} and equivariance of Fourier transform give
\begin{equation}\label{eq:EVII-double-support}
 \cS^*(\fp)^{J,\chi}=\mathcal M_\chi.
\end{equation}
In particular both support conditions follow from global sign
equivariance. Put $q(z)=\frac{1}{2}B(z,z)$ and let \(\Delta\) be the holomorphic constant-coefficient Laplacian
defined by the inverse form.  Write $E_z=\sum_{i=1}^{54}z_i\partial_{z_i}$
for the holomorphic Euler field.  On scalar generalized functions set
\[
e_q=q,\qquad
f_q=-\frac12\Delta,\qquad
h_q=E_z+27.
\]
The identities
\[
[E_z,q]=2q,\qquad
[E_z,\Delta]=-2\Delta,\qquad
[\Delta,q]=2E_z+54
\]
show that \(e_q,f_q,h_q\) satisfy the \(\mathfrak{sl}_2\)-relations.
Their antiholomorphic counterparts
\[
\bar e_q=\bar q,\qquad
\bar f_q=-\frac12\bar\Delta,\qquad
\bar h_q=\bar E_z+27
\]
commute with the holomorphic triple, and hence give an action of
$\mathfrak{sl}_2\oplus\mathfrak{sl}_2$
on \(\mathcal M_\chi\).

Both triples act locally finitely.  Indeed, since $\operatorname{supp}T\subset\mathcal N$
and \(q,\bar q\) vanish on \(\mathcal N\), finite order of \(T\)
implies
\[
q^N T=\bar q^{\,N}T=0
\]
for \(N\) sufficiently large.  Applying the same argument to
\(\mathcal F_B T\), whose support is also contained in \(\mathcal N\),
and conjugating by Fourier transform gives
\[
\Delta^N T=\bar\Delta^{\,N}T=0
\]
for \(N\) sufficiently large.  Thus Lemma~C.0.3 of \cite{AG}
applied to the two commuting triples gives a locally finite
\(\mathfrak{sl}_2\oplus\mathfrak{sl}_2\)-module.

Let \(V_m\) denote the irreducible \(\mathfrak{sl}_2\)-module of
highest weight \(m\).  Every finite-dimensional simple constituent is
therefore of the form
\[
V_m\otimes V_n,\qquad m,n\ge0.
\]
Every two nonzero vectors in such a constituent have the same support:
each generates the constituent under polynomial differential
operators, and such operators do not enlarge support.

Now let \(L\cdot e\) be a maximal orbit in the support and choose a
normal triple \((e,d,f)\).  Put
\[
T(e)=\operatorname{tr}(\operatorname{ad}d\mid\mathfrak h^e),
\qquad
T_{\mathfrak p}(e)
=
\operatorname{tr}(\operatorname{ad}d\mid\mathfrak p^e).
\]

If $[d,n_j]=-\ell_j n_j,\ n_j\in\mathfrak p^f$, then the corresponding Kazhdan weight is \(w_j=\ell_j+2\).
The invariant form pairs \(\mathfrak p^e(r)\) perfectly with
\(\mathfrak p^f(-r)\). Hence
\[
\sum_j\ell_j=T_{\mathfrak p}(e),
\qquad
\dim\mathfrak p^f=\dim\mathfrak p^e,
\]
and therefore
\[
\sum_jw_j
=
T_{\mathfrak p}(e)+2\dim\mathfrak p^e.
\]

On the other hand, each irreducible summand \(V_r\) contributes
\(r\) both to
\(\operatorname{tr}(\operatorname{ad}d\mid\mathfrak g^e)\)
and to
\(\dim\mathfrak g-\dim\mathfrak g^e\).  Thus
\[
T(e)+T_{\mathfrak p}(e)
=
\dim\mathfrak g-\dim\mathfrak g^e.
\]
Using $\dim\mathfrak p^e-\dim\mathfrak h^e
=
\dim\mathfrak p-\dim\mathfrak h$ (see \ref{eq:kernel-difference}), we obtain
\[
\begin{aligned}
\sum_j w_j
&=
\dim\mathfrak g-\dim\mathfrak g^e
-T(e)+2\dim\mathfrak p^e \\
&=
2\dim_{\mathbb C}\mathfrak p-T(e)
=
108-T(e).
\end{aligned}
\] Let $E_{\mathrm{wt}}
=
\sum_j w_j x_j\partial_{x_j}$
be the weighted Euler field on the normal slice.  The infinitesimal
Kazhdan action is the sum of \(2E_z\) and an orbit-tangent vector
field.  Hence, on an \(L\)-invariant scalar slice germ, $E_z\equiv \frac12E_{\mathrm{wt}}$
modulo orbit-tangent vector fields.

For a holomorphic normal delta jet $\partial_{x_1}^{k_1}\cdots\partial_{x_r}^{k_r}\delta_0$, define its weighted derivative degree by $D=\sum_j k_jw_j$.
Since $E_{\mathrm{wt}}\delta_0
=
-\Bigl(\sum_jw_j\Bigr)\delta_0$, such a jet has \(E_{\mathrm{wt}}\)-eigenvalue
\[
-\Bigl(\sum_jw_j+D\Bigr)
=
-\bigl(108-T(e)+D\bigr).
\]
Thus the action induced by the operator
\(h_q=E_z+27\) on this transverse jet has eigenvalue $\frac{T(e)-54-D}{2}$.
For a simultaneous highest-weight vector in a simple constituent
\(V_m\otimes V_n\), comparison with $h_qT=mT,\ \bar h_qT=nT$ gives
\begin{equation}\label{eq:EVII-degree}
 D=T(e)-54-2m,\qquad
D'=T(e)-54-2n,
\qquad D,D'\ge0.
\end{equation}

For reference, the entire trace and normal-degree table is below.
In the characteristic column the first six digits are
$\alpha_1(d),\ldots,\alpha_6(d)$; the last entry is $\alpha_7(d)$.
Exponents in the last column denote multiplicities.
\begin{center}\small
\begin{tabular}{rcrl}
\toprule
Orbit&Characteristic&$T(e)$&Kazhdan weights\\\midrule
1&$(100000;0)$&16&$2^{20},3^{16},4$\\
2&$(000001;-2)$&16&$2^{20},3^{16},4$\\
3&$(000001;0)$&16&$2^2,3^{16},4^{10}$\\
4&$(100000;-2)$&16&$2^2,3^{16},4^{10}$\\
5&$(100001;-2)$&32&$2^{10},3^{16},4^2$\\
6&$(000000;2)$&0&$4^{27}$\\
7&$(000000;-2)$&0&$4^{27}$\\
8&$(000002;-2)$&32&$2^{16},4^{11}$\\
9&$(200000;-2)$&32&$2^{16},4^{11}$\\
10&$(020000;-2)$&44&$2^{10},4^{11}$\\
11&$(010010;-2)$&48&$3^8,4^4,5^4$\\
12&$(011000;-3)$&48&$3^8,4^4,5^4$\\
13&$(300001;-2)$&52&$2^2,4,5^8,8$\\
14&$(100003;-6)$&52&$2^2,4,5^8,8$\\
15&$(200002;-4)$&60&$2^2,4^8,6^2$\\
16&$(200002;-2)$&44&$4^2,6^8,8$\\
17&$(400000;-2)$&60&$4^{10},8$\\
18&$(000004;-6)$&60&$4^{10},8$\\
19&$(200002;-6)$&44&$4^2,6^8,8$\\
20&$(220002;-6)$&84&$4,6^2,8$\\
21&$(400004;-6)$&84&$4,8,12$\\
22&$(400004;-10)$&84&$4,8,12$\\\bottomrule
\end{tabular}
\end{center}
The orbit list, numbering, and characteristics are taken from
\cite[Theorem 3 and Table XIII]{Dok88}; the values of \(T(e)\) and the
Kazhdan weights in the last two columns are recomputed here from the
\(\operatorname{ad}d\)-weight multiplicities. For \(j\ge 0\), let
$h_j=\dim\mathfrak h(j),\
p_j=\dim\mathfrak p(j)$, where \(\mathfrak h(j)\) and \(\mathfrak p(j)\) are the
\(\operatorname{ad}d\)-weight-\(j\) spaces. Let
\(m_j^{\mathfrak h}\) and \(m_j^{\mathfrak p}\) denote the
multiplicities of the irreducible \(\mathfrak{sl}_2\)-summands of
highest weight \(j\) whose highest-weight lines lie in
\(\mathfrak h\) and \(\mathfrak p\), respectively.  Since
\(\operatorname{ad}e\) interchanges \(\mathfrak h\) and
\(\mathfrak p\),
\[
m_j^{\mathfrak h}=h_j-p_{j+2},
\qquad
m_j^{\mathfrak p}=p_j-h_{j+2}.
\]
Consequently,
\[
T(e)=\sum_{j\ge0}j\,m_j^{\mathfrak h},
\]
and the Kazhdan weight \(j+2\) occurs in \(\mathfrak p^f\) with
multiplicity \(m_j^{\mathfrak p}\). The zero orbit is excluded separately: on a holomorphic delta jet
\(\partial^\alpha\delta_0\), the operator
\(h_q=E_z+27\) has eigenvalue $-27-|\alpha|<0$, whereas a finite-dimensional \(\mathfrak{sl}_2\)-highest weight is
nonnegative. For a nonzero maximal support orbit, (\ref{eq:EVII-degree}) implies \(T(e)\ge54\). Hence the table leaves only $15,\ 17,\ 18,\ 20,\ 21,\ 22$ as possible maximal support orbits.  Lemma~\ref{lem:EVII-exclude20} has already excluded
orbits \(20,21,22\) from the support, so it remains to treat
\(15,17,18\).

\subsection{The remaining transverse equations}
After the node argument and the normal-degree reduction, only
orbits~$15$,~$17$, and~$18$ remain as possible maximal support orbits.
At orbit~$15$ the transverse quadratic symbol contradicts local
nilpotence of the Laplacian. At orbits~$17$ and~$18$ we compute the
third iterate of the radial Laplacian on the complete candidate
jet spaces.

Orbit 15 has a two-dimensional
\(\operatorname{ad}d\)-weight-zero conormal subspace. Using the invariant form \(B\), the conormal space to
\(L\cdot e\) at \(e\) is identified with $N_e^*(L\cdot e)\simeq \mathfrak p^e$ since $B(z,[X,e])=-B([z,e],X)$.
Moreover, \(B\) pairs \(\operatorname{ad}d\)-weight \(j\) with
weight \(-j\).  Since the weights occurring in \(\mathfrak p^e\)
are nonnegative, the restriction of the inverse quadratic symbol
to the conormal is supported on the weight-zero part.
In the explicit basis \eqref{eq:EVII-15-conormal} below, the Gram
matrix of the restricted inverse
quadratic symbol is
\[
 \begin{pmatrix}0&-3\\-3&0\end{pmatrix},\qquad\det=-9.
\]
The corresponding principal transverse symbol of \(\Delta\) is therefore
nonzero.  By \(L\)-equivariance, the quadratic symbol satisfies
\[
b_{\ell\cdot e}
=
\operatorname{Ad}(\ell)_* b_e,
\qquad \ell\in L,
\]
and is therefore nonzero along the whole orbit.

Near a maximal support orbit, the germ is supported on the smooth
orbit and has finite transverse order. If \(s\neq0\) is its highest transverse symbol, then the highest
transverse symbol of \(\Delta^kT\) is $b^k s$. Fiberwise multiplication by \(b^k\) is injective on the
symmetric algebra.  Since this bundle map has constant rank along the
orbit, it admits locally a smooth left inverse; hence it is also
injective on distribution-valued symbols, and remains so after tensoring
with antiholomorphic derivatives. Thus
\[
\Delta^kT\neq 0
\qquad\text{for every }k\geq 1.
\]
On the other hand, the Fourier-support condition, as established
above, implies that \(\Delta\) acts locally nilpotently on \(\mathcal M_\chi\).
Hence
\[
\Delta^N T=0
\]
for some \(N\), a contradiction. Therefore orbit \(15\) cannot be a
maximal support orbit of a nonzero simple constituent of \(\mathcal M_\chi\).

At orbits \(17,18\), we have $T(e)=60$ and the Kazhdan weights on the normal slice are $(8,4^{10})$.
Hence, for a holomorphic highest-weight vector of highest weight \(m\),
formula~\eqref{eq:EVII-degree} gives
\[
D=60-54-2m=6-2m.
\]
Since \(D\ge0\) and every holomorphic weighted degree is divisible by
\(4\), necessarily
\[
m\in\{1,3\}.
\]
The same argument applies antiholomorphically.

Both \(V_1\) and \(V_3\) contain an \(h_q\)-weight-\(+1\) vector, and $f_q^3=\left(-\frac12\Delta\right)^3$ annihilates that weight space.  For an \(h_q\)-weight-\(+1\) vector,
formula~\eqref{eq:EVII-degree} gives
\[
D=60-54-2=4.
\]
Since the normal Kazhdan weights are \(8,4,\ldots,4\), such a
holomorphic jet is necessarily linear in the ten weight-\(4\)
coordinates. Its ordinary order is therefore one.
The connected-stabilizer equations leave a two-dimensional candidate
space, on which we compute the complete radial action.

We now compute the radial operator.
Let $n_0,\ldots,n_{10}$ be the ordered normal basis, of weights
$(-6,-2,\ldots,-2)$, and choose an \(\operatorname{ad}d\)-graded pivot-column complement
\[
\mathfrak h=\mathfrak h^e\oplus\mathfrak m,
\qquad
\mathfrak m=\langle m_1,\ldots,m_{43}\rangle .
\] Write $s=\sum s_jn_j$ and form
\[
 W(s)=([m_1,e+s],\ldots,[m_{43},e+s],n_0,\ldots,n_{10}),
 \quad U=W^{-1},\quad P=U_{\rm last\ 11\ rows},\quad K=UQU^t,
\]
where $Q$ is the inverse ambient form matrix. Twice differentiating
the inverse of the local action coordinates
$\exp(\operatorname{ad}\sum t_im_i)(e+s)$ gives
\begin{equation}\label{eq:EVII-radial}
 \Delta_{\rm rad}=\sum_{i,j}g_{ij}(s)\partial_i\partial_j
                         +\sum_ih_i(s)\partial_i,\qquad
 g=PQP^t,
\end{equation}
with
\[
 h=-P\left(\sum_{a,b=1}^{43}K_{ab}[m_a,[m_b,e+s]]
       +2\sum_{a=1}^{43}\sum_{j=0}^{10}K_{a,44+j}[m_a,n_j]\right).
\]
The column index $44+j$ here is one-based. For an \(L\)-invariant generalized function, pullback under the local
action map is constant in the orbit variables.  Hence
\(\Phi^*(\Delta T)\) is likewise independent of those variables, and
evaluation at the identity orbit coordinate gives the displayed
radial operator.  Thus \eqref{eq:EVII-radial} holds for invariant distributions as well
as for functions. Since
\[
W(s)=W_0+\sum_k s_kW_k,
\qquad
U(s)=W(s)^{-1},
\]
differentiating \(U(s)W(s)=I\) at \(s=0\) gives
\[
\left.\partial_{s_k}U\right|_{0}
=
-U_0W_kU_0.
\]
The linear terms of $g(s)=P(s)QP(s)^t$
follow by differentiating this identity at \(s=0\). The Kazhdan degrees of $g_{ij},h_i$ are respectively
$w_i+w_j-4,w_i-4$. Thus $g$ has ordinary degree at most three and
$h$ at most one. The resulting finite Taylor expansion determines
the complete radial operator. In this normalization, writing $q_0(s)=B(e+s,e+s)$,
\[
 \Delta_{\rm rad}q_0=108,\qquad g\,dq_0=(w_i s_i)_i.
\]

Encode a finite holomorphic normal jet by $\partial_s^\alpha\delta_0 \longleftrightarrow x^\alpha$, thus \(x_j\) records one application of the normal derivative
\(\partial_{s_j}\). Under this identification,
\[
\partial_{s_j}\longleftrightarrow x_j,
\qquad
s_j\longleftrightarrow-\partial_{x_j},
\]
since $s_j\partial_s^\alpha\delta_0
=
-\alpha_j\partial_s^{\alpha-e_j}\delta_0$. Therefore the full radial operator \eqref{eq:EVII-radial} acts on the polynomial
encoding \(F(x)\) of a normal jet by
\begin{equation}\label{eq:EVII-jet-action}
 \mathcal R F=\sum_{i,j}g_{ij}(-\partial_x)(x_ix_jF)
                    +\sum_i h_i(-\partial_x)(x_iF).
\end{equation}
The derivatives act on the displayed products, giving the full
radial action, including all terms from derivatives of the coefficients.

For \(z\in\mathfrak h^e\), put $v_z(s)=P(s)[z,e+s]$. This is the vector field induced on the slice by the infinitesimal
action of the connected stabilizer \(L_e^\circ\).  Hence an invariant
transverse germ satisfies
\[
v_z(s)\cdot\partial_s T=0,
\qquad z\in\mathfrak h^e.
\]
On the holomorphic weighted-degree-\(4\) jet space $\mathcal J_4=\operatorname{span}\{x_1,\ldots,x_{10}\}$, the common kernel of these equations is two-dimensional at each of
orbits \(17\) and \(18\).

At orbit \(17\), a basis of this common kernel is
\[
 A=x_{10}+x_2+3x_3+3x_4+3x_5-x_6+x_7+x_8+x_9,
 \qquad B_0=x_1.
\]
The coefficient functionals $[x_2^4]$ and $[x_1^3x_2]$ on
$(\mathcal R^3A,\mathcal R^3B_0)$ give the matrix
\begin{equation}\label{eq:EVII-M17}
 M_{17}=\begin{pmatrix}3/500&0\\286/9&10\end{pmatrix},
 \qquad\det M_{17}=3/50.
\end{equation}
The entries admit the following calculation.
Write $g_{ij}^{(1)}(s)=\sum_k c_{ij}^k s_k$ for the linear part of \(g_{ij}\), and let \(h_i^{(0)}\) denote
the constant part of \(h_i\).  Only these terms contribute to the
highest ordinary degree.  Hence that part of \(\mathcal R\) has the form
\[
\mathcal T=M_{\mathcal T(1)}+\mathscr D,
\qquad
\mathscr D
=
-\sum_{i,j,k}c_{ij}^k x_ix_j\partial_{x_k},
\]
where \(\mathscr D\) is a derivation. Here $M_f$ denotes multiplication by $f$.
For orbit \(17\), the exact coefficients give $\mathcal T(1)=2B_0$, so $\mathcal T=M_{2B_0}+\mathscr D$.
Put
\[
C_0
=
B_0\left(
x_3+x_4+\frac34x_5-\frac12x_6+\frac12x_7
\right)
+\frac34\left(
x_{10}x_2-x_3x_5+x_3x_9-x_4x_5-x_4x_8+x_6x_7
\right).
\]
Then \(\mathbb C[A,B_0,C_0]\) is stable under \(\mathscr D\), and
\[
\mathscr D(A)
=
\frac1{10}A^2+\frac52AB_0-10C_0,
\]
\[
\mathscr D(B_0)
=
\frac34AB_0-\frac23B_0^2-3C_0,
\]
\[
\mathscr D(C_0)
=
\frac7{80}A^2B_0
+\frac7{24}AB_0^2
-\frac14AC_0
-\frac{11}{6}B_0C_0.
\]
Since \(A\) and \(B_0\) have ordinary degree \(1\), three
applications of \(\mathcal T\) give the degree-\(4\) parts of
\(\mathcal R^3A\) and \(\mathcal R^3B_0\).  Extracting the coefficients
\([x_2^4]\) and \([x_1^3x_2]\) gives precisely the two rows of
\(M_{17}\) in \eqref{eq:EVII-M17}.

At orbit \(18\), a basis of the corresponding common kernel is
\begin{align*}
 A&=\tfrac32x_{10}+x_2-\tfrac12x_3+\tfrac12x_4-x_5
                  -\tfrac32x_6-x_7+\tfrac12x_9,\\
 B_0&=x_1-\tfrac12x_{10}-\tfrac12x_3+\tfrac12x_4
                  +\tfrac12x_6-x_8+\tfrac12x_9.
\end{align*}
The coefficient functionals $[x_1^4]$ and $[x_3^4]$ give
\begin{equation}\label{eq:EVII-M18}
 M_{18}=\begin{pmatrix}
 -354111/125000&-763911/125000\\
 -192/15625&-192/15625
 \end{pmatrix},\qquad
 \det M_{18}=-393408/9765625.
\end{equation}
The explicit triples, normal bases, and pivot-column complements
specified in Section~\ref{sec:EVII-coordinates}, together with
\eqref{eq:EVII-radial}--\eqref{eq:EVII-jet-action}, determine the
two coefficient maps represented by \(M_{17}\) and \(M_{18}\) over
\(\mathbb Q\). All nonlinear stabilizer equations vanish on the
displayed candidate jets and on their first two
radial images.  At both orbits, the zero-weight part of
\(\mathfrak h^e\) has dimension \(21\) and the positive-weight part
has dimension \(15\).  Thus the two-dimensional spaces displayed
above are the complete connected-stabilizer candidate spaces, and the
matrices \eqref{eq:EVII-M17}--\eqref{eq:EVII-M18} test all possible holomorphic
\(h_q\)-weight-\(+1\) jets. After complexifying the real stabilizer action, the holomorphic and
antiholomorphic infinitesimal actions split and impose the stabilizer
equations separately on the two factors. The mixed
weight-\(+1\) jet space satisfying the connected-stabilizer equations
is contained in $\operatorname{span}\{A,B_0\}
\otimes
\mathbb C[\bar x_0,\ldots,\bar x_{10}]$. 
Here the variables $\bar x_j$ encode
antiholomorphic normal derivatives of the
delta distribution on the slice. Consequently, injectivity of \(\Delta^3\) on
\(\operatorname{span}\{A,B_0\}\) is preserved after tensoring with
arbitrary antiholomorphic jets.

Since the determinants in \eqref{eq:EVII-M17} and \eqref{eq:EVII-M18} are nonzero,
\(\Delta^3\) is injective on the entire holomorphic
\(h_q\)-weight-\(+1\) candidate space at either orbit. The operator $f_q^3=\left(-\frac12\Delta\right)^3$
annihilates the \(h_q\)-weight-\(+1\) space of both \(V_1\) and
\(V_3\). Thus no nonzero such germ exists, and neither orbit \(17\)
nor orbit \(18\) can be a maximal support orbit.

\begin{proof}[Proof of Theorem~\ref{thm:EVII}]
By \eqref{eq:EVII-double-support}, a nonzero sign distribution would
give a nonzero simple quadratic constituent in $\mathcal M_\chi$.
The orbit list is finite, so its support has a maximal orbit.
Formula~\eqref{eq:EVII-degree} excludes every orbit with $T(e)<54$.
Lemma~\ref{lem:EVII-exclude20} excludes $20,21,22$ from the support.
The transverse quadratic symbol excludes $15$, and
\eqref{eq:EVII-M17}--\eqref{eq:EVII-M18} exclude $17,18$.
There is no remaining maximal orbit, a contradiction. Since
$\fp^L=0$ and $J/L\simeq\mu_2$, this proves componentwise regularity.
\end{proof}

\subsection{Exact coordinates for the finite calculations}
\label{sec:EVII-coordinates}
We fix the root and sign conventions for the rational maps.
If $C_{E_7}$ is
the Cartan matrix of the diagram above, let
$\alpha_1,\ldots,\alpha_7$ be the simple roots and generate the
positive roots by adjoining a simple root when the inner product is
$-1$. Order them first by height and then, within each height, in ascending
lexicographic order. Number their
vectors $7,\ldots,69$, and their negatives in the same order
$70,\ldots,132$. Write
\[
 \alpha=\sum_i a_i\alpha_i,\qquad
 \beta=\sum_i b_i\alpha_i,
\]
and let $v_0,\ldots,v_6=H_1,\ldots,H_7$. We use
\[
 [H_i,E_\alpha]=(\alpha_i,\alpha)E_\alpha,\qquad
 [E_\alpha,E_{-\alpha}]=-\sum_i a_iH_i.
\]
For $\beta\ne-\alpha$,
\[
 [E_\alpha,E_\beta]=
 (-1)^{\sum_i a_i b_i+
              \sum_{i<j,\,(C_{E_7})_{ij}=-1}a_i b_j}
 E_{\alpha+\beta}
\]
when $\alpha+\beta$ is a root, and the bracket is zero otherwise.
Normalize the invariant form by $B(E_\alpha,E_{-\alpha})=-1$.
For the orbit-$15$ symbol calculation, use
\begin{align*}
 e_{15}={}&-v_{56}-v_{59}+v_{62}-v_{64}+v_{65}+v_{66}+v_{67}+v_{68}\\
         &-v_{77}-v_{83}+v_{89}-v_{95}+v_{96}+v_{101}-v_{106}+v_{111},\\
 d_{15}={}&4v_0+4v_1+6v_2+8v_3+6v_4+4v_5.
\end{align*}
The weight-zero kernel of $\operatorname{ad}e_{15}$ in $\fp$ has basis
\begin{equation}\label{eq:EVII-15-conormal}
 \begin{split}
 z_1&=2v_{40}+v_{53}+v_{57}-v_{60}+v_{63}+v_{115},\\
 z_2&=v_{52}+\tfrac12v_{103}-\tfrac12v_{108}
                    -\tfrac12v_{112}+\tfrac12v_{117}+v_{126}.
 \end{split}
\end{equation}
Under the invariant-form identification the conormal is $\fp^{e_{15}}$.
The root-pairing rule gives $B(z_1,z_1)=B(z_2,z_2)=0$ and
$B(z_1,z_2)=-3$, proving the asserted symbol calculation.

For the characteristics of orbits $17,18$, choose the triples specified
in the following coordinate record. For each $d$-weight in increasing
order, take the nullspace of $\operatorname{ad}f$ on the odd
root-space basis in increasing index order, using reduced row echelon
form with free coordinates successively equal to $1$. These are the
normal vectors $n_0,\ldots,n_{10}$. Take pivot columns of
$a\mapsto[a,e]$ on the ordered even basis for the $m_i$.
This specifies every entry in \eqref{eq:EVII-radial}.

\input{evii_coordinate_record}

Exact coordinate data and verification scripts for the finite
computations in this section are provided in the ancillary material. The exhaustiveness of
the orbit list follows from the classification theorem
\cite[Theorem~3]{Dok88}.

%% file: evii_coordinate_record.tex
With this numbering the two triples are as follows. The subscripts
$17,18$ denote orbit numbers, not basis indices.
\begin{align*}
e_{17}={}&-v_{40}-v_{45}+v_{49}-v_{53}+v_{54}+v_{56}+v_{57}\\
&+v_{59}-v_{60}-v_{62}+v_{63}-v_{64}+v_{65}+v_{66}\\
&-v_{67}+v_{68}+v_{70}-v_{77}-v_{83}+v_{89}-v_{95}\\
&-v_{96}-v_{101}-v_{106}+v_{111}-v_{115},\\
d_{17}={}&6v_{0}+5v_{1}+8v_{2}+10v_{3}+7v_{4}+4v_{5}+v_{6},\\
f_{17}={}&-\tfrac{4}{3}v_{7}+\tfrac{1}{3}v_{14}+\tfrac{2}{3}v_{20}+\tfrac{2}{3}v_{26}+\tfrac{1}{3}v_{32}+\tfrac{4}{3}v_{33}+\tfrac{1}{3}v_{38}\\
&-\tfrac{1}{3}v_{43}-\tfrac{4}{3}v_{48}+\tfrac{1}{3}v_{52}+\tfrac{1}{2}v_{103}-\tfrac{3}{2}v_{108}-\tfrac{3}{2}v_{112}+\tfrac{3}{2}v_{116}\\
&-\tfrac{1}{2}v_{117}-\tfrac{1}{2}v_{119}-\tfrac{1}{2}v_{120}-\tfrac{3}{2}v_{122}-\tfrac{1}{2}v_{123}-\tfrac{3}{2}v_{125}-\tfrac{3}{2}v_{126}\\
&+\tfrac{1}{2}v_{127}-\tfrac{1}{2}v_{128}-\tfrac{1}{2}v_{129}-\tfrac{1}{2}v_{130}-\tfrac{1}{2}v_{131}.
\end{align*}
\begin{align*}
e_{18}={}&-v_{52}-v_{56}+v_{59}-v_{62}+v_{64}+v_{65}+v_{66}\\
&+v_{67}-v_{68}-v_{69}+v_{77}-v_{83}+v_{89}+v_{95}\\
&-v_{96}+v_{101}+v_{103}-v_{106}-v_{108}+v_{111}-v_{112}\\
&-v_{116}-v_{117}-v_{120}+v_{123}-v_{126},\\
d_{18}={}&2v_{0}+3v_{1}+4v_{2}+6v_{3}+5v_{4}+4v_{5}-v_{6},\\
f_{18}={}&-\tfrac{9}{10}v_{14}+\tfrac{3}{10}v_{20}-\tfrac{3}{10}v_{26}-\tfrac{9}{10}v_{32}+\tfrac{3}{10}v_{33}+\tfrac{3}{10}v_{38}-\tfrac{9}{10}v_{40}\\
&+\tfrac{9}{10}v_{43}-\tfrac{3}{10}v_{45}-\tfrac{3}{2}v_{48}-\tfrac{3}{10}v_{49}+\tfrac{3}{10}v_{53}+\tfrac{3}{10}v_{54}-\tfrac{3}{10}v_{57}\\
&+\tfrac{3}{10}v_{60}+\tfrac{9}{10}v_{63}+\tfrac{7}{5}v_{115}+\tfrac{1}{5}v_{119}-\tfrac{4}{5}v_{122}+\tfrac{4}{5}v_{125}-\tfrac{1}{5}v_{127}\\
&+\tfrac{2}{5}v_{128}-\tfrac{1}{5}v_{129}-\tfrac{7}{5}v_{130}+\tfrac{4}{5}v_{131}-\tfrac{2}{5}v_{132}.
\end{align*}
Direct brackets give $[e_i,f_i]=d_i$, $[d_i,e_i]=2e_i$,
and $[d_i,f_i]=-2f_i$. Together with the stated row-reduction rule,
these data determine the normal bases used in the two matrices.

%% file: global_consequences.tex
\subsection{Global regularity}\label{sec:global-regularity}
\begin{proof}[Proof of Theorem~\ref{thm:main}\textup{(i)}]
Take the canonical cover of Lemma~\ref{lem:cover-center} and group
its simple factors into orbits of the lifted involution. We use
pleasantness for the simply connected factors of this cover, while
niceness depends only on the symmetric Lie algebra. A two-factor
orbit gives a diagonal pair, which is pleasant
\cite[Example 4.3]{Rubio}; a factor with trivial involution has zero
odd space. The anti-invariant center disappears in $Q_H(\fp)$.

For the remaining simple factors, Rubio's Tables~2--4
\cite[\S4]{Rubio} give the following alternatives according to
pleasantness and niceness.
AII and unequal AIII are pleasant by \cite[Proposition 4.14]{Rubio},
unequal CII by \cite[Proposition 4.16]{Rubio}, and all exceptional
factors other than $E_7/A_7$ and EVII by \cite[Theorem 4.20]{Rubio}.
Unequal odd--odd Spin blocks are pleasant by Lemma~\ref{lem:odd-spin}.
Among the nonpleasant factors, the nice ones, including $E_7/A_7$ and Spin blocks with
$|r-s|\leq2$, satisfy \eqref{eq:componentwise} by
Corollary~\ref{cor:nice-component}. The remaining classical factors
are covered by Corollaries~\ref{cor:DC-regular} and \ref{cor:Spin},
and EVII by Theorem~\ref{thm:EVII}. 

Proposition~\ref{prop:component-product} gives componentwise regularity
on the product. Lemma~\ref{lem:central-interface} identifies the
full fixed image of the original group with the intermediate
subgroup $J_B$, and Proposition~\ref{prop:intermediate} proves
\eqref{eq:regularity}. This product and central-quotient argument
follows \cite[\S3]{LXYv1}; its ingredients were established in
Section~\ref{sec:assembly}.
\end{proof}

\subsection{Gelfand--Kazhdan consequence}\label{sec:gk-consequence}
We follow the canonical-cover argument of \cite[\S8.3]{LXYv1},
applying the regularity theorem together with results of Steinberg
and Aizenbud--Gourevitch.

\begin{proof}[Proof of Theorem~\ref{thm:main}\textup{(ii)}]
Let $\widetilde g$ have a closed $\widetilde H\times\widetilde H$
orbit, and set $x=\widetilde g\widetilde\theta(\widetilde g^{-1})$.
By \cite[Proposition 7.2.1 and Definition 7.2.2]{AG}, $x$ is
semisimple and the descendant is
$(\widetilde G_x,\widetilde H_x,\widetilde\theta|_{\widetilde G_x})$.
Write $x=(z,s)$. The torus is central, so
\[
 \widetilde G_x=Z(G)^\circ\times(G_{\rm der}^{\rm sc})_s.
\]
Steinberg's connected-centralizer theorem \cite{Steinberg} makes
this group connected reductive. Since $\widetilde\theta(x)=x^{-1}$,
the involution preserves it, and its full fixed subgroup is exactly
$\widetilde H_x$. Thus every descendant satisfies
Theorem~\ref{thm:main}\textup{(i)}. The cover is a connected complex
symmetric pair and hence is good by \cite[Corollary 7.1.7]{AG}. 
Here a symmetric pair is good if every
closed $H\times H$-orbit in $G$ is preserved by
the anti-involution
$\sigma(g)=\theta(g^{-1})$. Thus \cite[Theorem 7.4.5]{AG} gives the Gelfand--Kazhdan property.
\end{proof}

\subsection{Multiplicity one for connected groups}
\label{sec:multiplicity-consequence}
For a Casselman--Wallach representation $E$, write $E^\vee$ for
its smooth contragredient. GP2 is the inequality
\[
 \dim\Hom_H(E,\C)\dim\Hom_H(E^\vee,\C)\leq1;
\]
GP1 bounds the first dimension alone.

\begin{lemma}\label{lem:compatible-C}
On $\widetilde G=Z\times G_{\rm ss}$, with $G_{\rm ss}$ simply
connected semisimple, there is a Chevalley involution $C$ commuting
with $\widetilde\theta$. It preserves $\widetilde H$ and satisfies
$E^\vee\simeq E^C$ for every irreducible Casselman--Wallach $E$.
Here $E^C$ means the representation precomposed with $C$.
\end{lemma}
\begin{proof}
The construction is essentially \cite[Proposition 8.1]{LXYv1}.
On $Z$ use inversion, which commutes with every torus automorphism.
On an exchanged pair write
$\widetilde\theta(x,y)=(\alpha^{-1}y,\alpha x)$.
Choose a pinning on the first factor and transport it by $\alpha$
to the second. If $C_1$ is the first Chevalley involution, use
$C_2=\alpha C_1\alpha^{-1}$. Their product commutes with the exchange.

On a stable simple factor choose a $\widetilde\theta$-stable Borel
subgroup and maximal torus, and then a pinning. The structure theorem
for semisimple automorphisms \cite{Steinberg} gives
\[
 \widetilde\theta=\operatorname{Int}(t)\delta,
 \qquad t\in T,\qquad t\delta(t)\in Z(G_{\rm ss}),
\]
where $\delta$ is the pinned diagram involution. The pinned Chevalley
involution $C_0$ inverts $T$ and commutes with $\delta$.
Set $C=\operatorname{Int}(t)C_0$. Then $C^2=1$, and
\[
 C\widetilde\theta C^{-1}
 =\operatorname{Int}(\delta(t)^{-1})\delta
 =\operatorname{Int}(t)\delta=\widetilde\theta.
\]
Their product gives the involution on the cover.

The anti-automorphism $\sigma=C\circ\mathrm{inv}$ is involutive,
normalizes conjugation, and fixes the chosen maximal torus pointwise.
Every closed conjugacy class is semisimple and meets that torus;
hence $\sigma$ is $\Ad(\widetilde G)$-admissible. It preserves
$\widetilde H$, as does $C$. The duality theorem
\cite[Theorem 8.2.1]{AG} now gives $E^\vee\simeq E^C$.
A character proof using Harish--Chandra regularity is described in
\cite[Remark 8.2.2]{AG}.
\end{proof}

\begin{proposition}\label{prop:connected-GP1}
If $G$ is connected, $(G,G^\theta)$ satisfies GP1.
\end{proposition}
\begin{proof}
The Gelfand--Kazhdan criterion
\cite[Theorem 8.1.5 and Corollary 8.1.6]{AG} gives GP2 on the cover.
Lemma~\ref{lem:compatible-C} makes its two period dimensions equal:
the underlying continuous functional is unchanged by twisting because
$C(\widetilde H)=\widetilde H$. GP2 therefore gives GP1 there.

For the finite descent, following \cite[\S8.3]{LXYv1}, pull an
irreducible Casselman--Wallach representation $E$ back along $m$.
Surjectivity preserves irreducibility. The covering is a local
diffeomorphism with finite kernel, so smoothness is unchanged.
Algebraic group norms on the two groups are polynomially comparable,
which preserves moderate growth. The inverse image of a maximal
compact subgroup is maximal compact and its representations with
trivial covering kernel are precisely those of the original compact
group. Thus admissibility and finite generation of the compact-finite
module are preserved, with the same Fr\'echet topology.
Since $m(\widetilde H)\subseteq G^\theta$,
\[
 \Hom_{G^\theta}(E,\C)
 \subseteq\Hom_{m(\widetilde H)}(E,\C)
 \simeq\Hom_{\widetilde H}(E\circ m,\C).
\]
The last space has dimension at most one.
\end{proof}

\subsection{Finite components and the full GP1 conclusion}
\label{sec:disconnected-GP1}
For disconnected ambient groups, we use restriction to the identity
component and continuity of periods.

\begin{lemma}\label{lem:CW-restriction}
Let $N=G^\circ$ for a possibly disconnected complex reductive group
$G$. Every irreducible Casselman--Wallach representation $E$ of $G$
contains a nonzero closed irreducible Casselman--Wallach
$N$-subrepresentation.
\end{lemma}
\begin{proof}
Choose a maximal compact subgroup $C\subset G$ and put $C_N=C\cap N$.
The Cartan decomposition for disconnected complex reductive groups
\cite[Theorem 0.1 and \S1.3.2, equation (5)]{JRT} gives
$G=NC$, with $C_N$ maximal compact in $N$ and
$[C:C_N]=[G:N]<\infty$. Finite coset representatives show that
$E_{C\text{-fin}}=E_{C_N\text{-fin}}=:V_0$.
This is the same finitely generated module for
$\mathfrak n_{\C}:=\Lie_{\R}(N(\C))\otimes_{\R}\C$.
For an irreducible compact type $\tau$ of $C_N$,
\[
 \Hom_{C_N}(\tau,V_0)
 \simeq\Hom_C(\operatorname{Ind}_{C_N}^C\tau,V_0)
\]
is finite-dimensional, since the induced compact representation is
finite-dimensional. Smoothness and moderate growth restrict to $N$.
Hence $E|_N$ is Casselman--Wallach for the connected group $N$.

By the Casselman--Wallach equivalence for the connected group \(N\),
Harish--Chandra submodules correspond to closed subrepresentations
\cite[\S3, Theorem~1 and the following discussion]{Wallach14}.
Since \(U(\mathfrak n_{\mathbb C})\) is noetherian and \(V_0\) is
finitely generated, its submodules satisfy the ascending chain
condition.  Hence the closed \(N\)-subrepresentations of \(E\) do
as well, and we may choose a maximal proper one \(U\subset E\).

Let \(g_1=1,\ldots,g_d\) represent \(G/N\), and put
\(U_j=g_jU\).  Each \(U_j\) is a maximal proper closed
\(N\)-subrepresentation.  Since \(G\) permutes the \(U_j\), $\bigcap_{j=1}^d U_j$
is a proper closed \(G\)-subrepresentation of \(E\), and hence is
zero.  Therefore
\[
E|_N\longrightarrow\bigoplus_{j=1}^d E/U_j,
\qquad
v\longmapsto(v+U_j)_j
\]
is injective.

Passing to \(C_N\)-finite vectors embeds \(V_0\) into a finite
direct sum of simple Harish--Chandra modules.  Hence \(V_0\)
contains a simple Harish--Chandra submodule \(W\).
Globalizing the inclusion \(W\hookrightarrow V_0\) gives a nonzero
closed irreducible Casselman--Wallach \(N\)-subrepresentation
\(S\subset E\).
\end{proof}

\begin{proposition}[Finite-component transfer of GP1]\label{prop:GP1-transfer}
Let $G$ be a possibly disconnected complex reductive algebraic group,
$N=G^\circ$, and $H$ a closed algebraic subgroup such that $G=NH$.
Put $K=H\cap N$. If
\[
 \dim\Hom_K(S,\C)\leq1
\]
for every irreducible Casselman--Wallach representation $S$ of $N$,
then $\dim\Hom_H(E,\C)\leq1$ for every irreducible
Casselman--Wallach representation $E$ of $G$.
\end{proposition}
\begin{proof}
Choose $S\subset E$ as in Lemma~\ref{lem:CW-restriction}.
Restriction gives a map
$\Hom_H(E,\C)\to\Hom_K(S,\C)$ of spaces of continuous functionals.
If $\lambda$ vanishes on $S$, its $H$-invariance makes it vanish on
every $hS$. Since $N$ is normal, $G=HN$, and $S$ is $N$-stable,
\[
 \operatorname{span}_{h\in H}hS
 =\operatorname{span}_{g\in G}gS.
\]
The closure is a nonzero closed $G$-subrepresentation, hence all of
$E$. Continuity gives $\lambda=0$. Restriction is injective, proving
the bound.
\end{proof}

\begin{proof}[Proof of Corollary~\ref{cor:multiplicity}]
Set $H=G^\theta$, $N=G^\circ$, and let $p:G\to F=G/N$ be the
finite component map. The surjection
\[
 G/H\longrightarrow F/p(H)
\]
has the finite discrete coset space $F/p(H)$ as its target.
Connectedness of $G/H$ forces
that space to be a point, so $G=NH$. Over $\C$ the equality also
holds on complex points. The involution preserves $N$ and
$H\cap N=N^{\theta|_N}$. Proposition~\ref{prop:connected-GP1}
therefore supplies the hypothesis of Proposition~\ref{prop:GP1-transfer},
which proves part~\textup{(i)}. Finally
\cite[Proposition 8.1.3]{AG} gives
$\mathrm{GP1}\Rightarrow\mathrm{GP2}\Rightarrow\mathrm{GP3}$. Since \(\mathrm{GP}3\) is van Dijk's conjecture in this setting,
part~(ii) follows.
\end{proof}

%% file: nice_factors.tex
\section{Distinguished-nilpotent calculations}\label{app:nice}

We give explicit calculations for Sekiguchi's inequality
\cite[Theorem 6.3]{Sek85} in the cases used in
Proposition~\ref{prop:nice-inequality}. For a normal triple $(e,d,f)$, write
\[
\mathfrak p^{e,d,f}
=
\{x\in\mathfrak p:[e,x]=[d,x]=[f,x]=0\}.
\]
First recall the criterion:
\begin{equation}\label{eq:distinguished-zero}
 e\text{ is }\fp\text{-distinguished}
 \quad\Longleftrightarrow\quad \fp^{e,d,f}=0.
\end{equation}
Indeed, the invariant form is nondegenerate on the triple centralizer and on its odd part. If the latter is nonzero, it contains $z$ with $\beta(z,z)\ne0$; the nonzero semisimple part of $z$ is odd and centralizes the triple. Conversely, when the odd triple centralizer is zero, every weight on $\fp^e$ is positive. Each element of $\fp^e$ then strictly raises the finite $d$-grading of $\fg$ and is nilpotent.

\subsection{AI}
On $V=\C^n$ with a nondegenerate symmetric form $B$, the space $\fp$ consists of the self-adjoint trace-zero endomorphisms. A self-adjoint nilpotent $e$ is an orthogonal sum of nondegenerate cyclic Jordan blocks. To see this directly, let $m$ be its largest block size. The symmetric form $B(v,e^{m-1}w)$ is nonzero, so choose $v$ with $B(v,e^{m-1}v)\ne0$. The Gram matrix of $v,ev,\ldots,e^{m-1}v$ has zero entries beyond the anti-diagonal and nonzero entries on it. This cyclic subspace is nondegenerate; its orthogonal complement is $e$-stable, and induction gives the decomposition.

If there is more than one block, a nonzero trace-zero linear combination of two orthogonal block projections is semisimple, self-adjoint, and commutes with $e$. Thus a distinguished $e$ has one block. Its endomorphism centralizer is $\C[e]$, so
\[
 \fp^e=\langle e,e^2,\ldots,e^{n-1}\rangle,
 \qquad \lambda_j=2j\quad(1\leq j<n).
\]
Consequently
\begin{equation}\label{eq:AI-delta}
 \delta_\fp(e)=\sum_{j=1}^{n-1}(2j+2)
       -\left(\frac{n(n+1)}2-1\right)
       =\frac{n(n+1)}2-1>0.
\end{equation}

\subsection{Equal AIII}
Let $V=V_+\oplus V_-$, with both dimensions $n$, and let $J$ be the grading involution. Use graded Jordan chains $W_{p_i}^{\varepsilon_i}$ with basis $v_0,\ldots,v_{p_i-1}$ and
\begin{equation}\label{eq:graded-chain}
 ev_t=v_{t-1},\qquad dv_t=(p_i-1-2t)v_t,
 \qquad Jv_t=\varepsilon_i(-1)^t v_t.
\end{equation}
Here $v_{-1}=0$ and $fv_t=(t+1)(p_i-1-t)v_{t+1}$. Passing from $\End V$ to $\mathfrak{sl}(V)$ removes only an even scalar and does not change $\fp$.

For the ordered pair of blocks $i,j$, let $U_{ij}$ be the odd part of $\Hom(W_{p_j}^{\varepsilon_j},W_{p_i}^{\varepsilon_i})$ and $V_{ij}=U_{ij}^e$. Put $t_i=0$ for even $p_i$ and $t_i=\varepsilon_i$ for odd $p_i$. Then $\dim U_{ij}=(p_ip_j-t_it_j)/2$. Write $p=p_i$, $q=p_j$, and denote the target and source chain bases by $u_0,\ldots,u_{p-1}$ and $v_0,\ldots,v_{q-1}$, respectively, with the normalization of \eqref{eq:graded-chain}. The $e$-intertwiners have a basis $T_r$, $0\leq r<\min(p,q)$, determined by
\[
 T_r(v_{q-1})=u_r,\qquad T_r(v_t)=e^{q-1-t}u_r.
\]
The image of the cyclic generator can be any vector killed by $e^q$, so these are all intertwiners. Subtracting the weights and signs in \eqref{eq:graded-chain} gives
\[
 \lambda_r=p+q-2-2r,\qquad
 \operatorname{parity}(T_r)=\varepsilon_i\varepsilon_j(-1)^{q-1+r}.
\]
The block contribution is therefore exactly
\begin{equation}\label{eq:AIII-contribution}
 c_{ij}=\sum_{\substack{0\leq r<\min(p,q)\\
                  \varepsilon_i\varepsilon_j(-1)^{q-1+r}=-1}}
               (p+q-2r)-\frac{pq-t_it_j}{2},
 \qquad \delta_\fp(e)=\sum_{i,j}c_{ij}.
\end{equation}
For equal parities this gives the following values, with $p\geq q$; in the mixed case the table gives the sum of the two ordered contributions:
\begin{equation}\label{eq:AIII-table}
\begin{array}{c|c|c}
 \text{block sizes}&\text{highest signs}&\text{contribution}\\ \hline
 p,q\text{ even}&\text{equal}&c_{ij}=c_{ji}=q\\
 p,q\text{ even}&\text{opposite}&c_{ij}=c_{ji}=0\\
 p,q\text{ odd}&\text{equal}&c_{ij}=c_{ji}=(q-p)/2\\
 p,q\text{ odd}&\text{opposite}&c_{ij}=c_{ji}=(p+q)/2\\
 \text{mixed}&\text{either}&c_{ij}+c_{ji}=\min(p,q)
\end{array}
\end{equation}
For $p=2a$, $q=2b$, $a\geq b$, equal signs select even $r$, whose sum is $2b(a+1)$; opposite signs select odd $r$, whose sum is $2ab$. Subtracting $pq/2=2ab$ gives the first two rows. In particular, an even diagonal block contributes $c_{ii}=p_i$. For $p=2a+1$, $q=2b+1$, equal signs select odd $r$, with sum $2b(a+1)$ and odd-space dimension $2ab+a+b$; opposite signs select even $r$, with sum $2(a+1)(b+1)$ and dimension $2ab+a+b+1$. These give the third and fourth rows, including $c_{ii}=0$ on an odd diagonal block. In the mixed case, reversing the ordered pair reverses the parity condition. Together the two sums include each $r$ once, so their sum before subtracting dimensions is
\[
 \sum_{r=0}^{m-1}(p+q-2r)=m(p+q-m+1),\qquad m=\min(p,q).
\]
Since $pq=m(p+q-m)$, subtracting the sum $pq$ of the two odd-space dimensions gives $c_{ij}+c_{ji}=m$.

For comparison with the block notation in
\cite[Lemma~6.9]{Sek85}, our contribution $c_{ij}$ subtracts
$\dim U_{ij}$, the dimension of the full odd block space.
Formula~\eqref{eq:AIII-contribution} and the table above give the
values directly: an even diagonal block contributes $p_i$, and an
opposite-sign odd pair contributes $(p_i+p_j)/2$ in each ordered
direction.

Let the odd plus sizes be $P_1\geq\cdots\geq P_b$ and the odd minus sizes be $Q_1\geq\cdots\geq Q_b$. The counts agree because $\dim V_+=\dim V_-$. All odd/odd contributions sum to
\begin{align*}
 &\sum_{a<c}(P_c-P_a)+\sum_{a<c}(Q_c-Q_a)
                          +\sum_{a,c}(P_a+Q_c)\\
 &\hspace{35mm}=\sum_{a=1}^b(2a-1)(P_a+Q_a).
\end{align*}
This is positive if $b>0$. The even diagonals are positive, and all unordered even/even and mixed contributions are nonnegative. If $b=0$, some even block is present. Thus $\delta_\fp(e)>0$ for every nilpotent in equal AIII, in particular for every distinguished one.

\subsection{CI}
Let $V$ be symplectic of dimension $2n$, with form $B$, and let $J^2=1$ satisfy $B(Ju,Jv)=-B(u,v)$. Its eigenspaces are Lagrangian and the fixed Lie algebra is $\mathfrak{gl}_n$. Decompose the standard module for the normal triple as
\[
 V=\bigoplus_{m\geq1} W_m\otimes M_m.
\]
Here $W_m$ is the $m$-dimensional irreducible
$\mathfrak{sl}_2$-module, and $M_m$ is its multiplicity space. On $W_m$ in the normalization \eqref{eq:graded-chain}, the invariant form $b_m$ has anti-diagonal entries $b_m(v_i,v_{m-1-i})=(-1)^i$. Its symmetry is $(-1)^{m-1}$, and the grading $P_mv_i=(-1)^iv_i$ changes the form by the same factor. Writing $B=b_m\otimes b_{M_m}$ and $J=P_m\otimes K_m$ gives
\begin{equation}\label{eq:CI-multiplicity}
 b_{M_m}(v,u)=(-1)^{m}b_{M_m}(u,v),
 \qquad b_{M_m}(K_mu,K_mv)=(-1)^m b_{M_m}(u,v).
\end{equation}
The triple centralizer acts through the isometry algebras of these multiplicity spaces.

If $m$ is odd, $M_m$ is symplectic with anti-symplectic grading. When it is nonzero, its odd symplectic algebra contains the semisimple matrix $\left(\begin{smallmatrix}0&I\\I&0\end{smallmatrix}\right)$ in a Lagrangian splitting. Thus a distinguished element has no odd-sized blocks. If $m$ is even, $M_m$ is orthogonal and $K_m$ is an isometry. If both signs occur, an off-diagonal skew matrix on one nondegenerate line of each sign is a nonzero semisimple odd element. If only one sign occurs, the odd centralizer is zero. By \eqref{eq:distinguished-zero}, the distinguished configurations are therefore sums of even symplectic chains $W_{2a_i}^{\varepsilon_i}$, with one sign for each repeated size.

Under $\mathfrak{sp}(V)\simeq\Sym^2V$, the involution is $-J\otimes J$: this follows from $uv\mapsto[w\mapsto B(v,w)u+B(u,w)v]$ and the anti-symplectic property of $J$. Thus $\fp$ is the $+1$-eigenspace of $J\otimes J$. The symmetric square of a chain of size $2a$ has highest weights $4a-2,4a-6,\ldots,2$, all of tensor sign $+1$. Its contribution is
\[
 \sum_{j=1}^a4j-a(a+1)=a(a+1).
\]
On $W_p^\varepsilon\otimes W_q^\eta$ the highest weights are $p+q-2-2r$, $0\leq r<\min(p,q)$, with highest sign $\varepsilon\eta(-1)^r$. Indeed, a highest vector is a linear combination of $v_t\otimes w_{r-t}$; the equation $e\xi=0$ gives alternating coefficients. For equal factors, transposition acts by $(-1)^r$, so symmetric squares select even $r$ and exterior squares select odd $r$.

For two even chains of sizes $2a\geq2b$, the $\fp$-dimension of the cross term is $2ab$. Equal signs select even $r$ and give a sum of $2b(a+1)$; opposite signs select odd $r$ and give $2ab$. Their contributions are respectively $2b$ and zero. Hence
\begin{equation}\label{eq:CI-delta}
 \delta_\fp(e)=\sum_i a_i(a_i+1)
       +\sum_{\substack{i<j\\\varepsilon_i=\varepsilon_j}}
                      \min(2a_i,2a_j)>0.
\end{equation}
The first sum is nonempty because $n\geq1$.

\subsection{BDI with block difference at most two}
Now let $V$ be orthogonal of dimension $2n+k$, and let $J$ be an isometric involution with eigenspace dimensions $n+k,n$. In the same multiplicity decomposition,
\begin{equation}\label{eq:BDI-multiplicity}
 b_{M_m}(v,u)=(-1)^{m-1}b_{M_m}(u,v),
 \qquad b_{M_m}(K_mu,K_mv)=(-1)^{m-1}b_{M_m}(u,v).
\end{equation}
For even $m$, the multiplicity space is symplectic with anti-symplectic grading; the first semisimple matrix used above excludes it. For odd $m$, it is orthogonal with isometric grading, and the second matrix excludes mixed signs at a fixed size. Thus the distinguished configurations are precisely sums of odd chains $W_{2a_i+1}^{\varepsilon_i}$, $a_i\geq0$, with one sign at each repeated size.

Here $\mathfrak{so}(V)\simeq\Lambda^2V$ and the involution is $J\otimes J$, so $\fp$ is the $-1$-eigenspace of $J\otimes J$. In $\Lambda^2W_{2a+1}$ the highest weights are $4a-2,4a-6,\ldots,2$, with odd highest vectors, and the odd-space dimension is $a(a+1)$. Its contribution is $a(a+1)$, including zero when $a=0$. For two chains of sizes $2a+1\geq2b+1$, the highest weights are $2a+2b-2r$, $0\leq r\leq2b$. Equal signs select odd $r$ and give
\[
 2b(a+1)-(2ab+a+b)=b-a;
\]
opposite signs select even $r$ and give
\[
 2(a+1)(b+1)-(2ab+a+b+1)=a+b+1.
\]

List the plus parameters as $A_1\geq\cdots\geq A_r\geq0$ and the minus parameters as $B_1\geq\cdots\geq B_s\geq0$. The trace of $J$ on an odd chain is its highest sign, so $r-s=k$. Adding the contributions just computed gives the closed formula
\begin{equation}\label{eq:BDI-delta}
\begin{aligned}
 \delta_\fp(e)
 &=\sum_i A_i(A_i+1)+\sum_j B_j(B_j+1)\\
 &\quad+\sum_{i<j}(A_j-A_i)+\sum_{i<j}(B_j-B_i)
   +\sum_{i,j}(A_i+B_j+1)\\
 &=\sum_{i=1}^r A_i(A_i+2i-k)
    +\sum_{j=1}^s B_j(B_j+2j+k)+rs.
\end{aligned}
\end{equation}
For $k=0,1,2$, every term in the last line is nonnegative, and every nonzero $A_i$ or $B_j$ contributes positively; when $k=2$ the first term is $A_1^2$. If $s>0$, then $rs>0$. If $s=0$, then $r=k$ and $\sum A_i=n\geq1$, so a positive term remains. This covers all nontrivial ranks. For $n=1,k=0$, the odd line in the toral algebra $\mathfrak{so}_2$ contains semisimple elements, so its zero nilpotent is not distinguished; its quotient $Q(\fp)$ is zero.

\subsection{The pair \texorpdfstring{$E_7/A_7$}{E7/A7}}\label{app:E7}
We use exact root arithmetic to obtain a positive lower bound for every possible distinguished characteristic. The root numbering in this subsection is independent of that used in
Section~\ref{sec:EVII}. In the hyperplane $\sum x_i=0$ in $\mathbb Q^8$, take the roots
\[
 \{u_i-u_j:i\ne j\}\ \cup\
 \left\{\tfrac12(\pm u_1\pm\cdots\pm u_8):
                              \text{four plus signs}\right\},
\]
where $u_i$ are the coordinate vectors. The difference-root spaces together with the Cartan subalgebra form
$\mathfrak{sl}_8$, while the remaining $70$ root spaces form
$\mathfrak p\simeq \Lambda^4\C^8$. Choose
\[
 \alpha_i=u_i-u_{i+1}\ (1\leq i\leq6),\qquad
 \alpha_7=\tfrac12(-u_1-u_2-u_3+u_4+u_5+u_6+u_7-u_8).
\]
The edges of the Dynkin diagram are $12,23,34,45,56,37$.

Move the characteristic $d$ to the dominant chamber of $E_7$. Its labels $\alpha_i(d)$ are nonnegative integers, being weights of the adjoint $\mathfrak{sl}_2$-module, and they belong to $\{0,1,2\}$. To check the upper bound, delete any nodes whose labels exceed two. All roots of weight two lie in the remaining Levi, so $e,f$, and hence $d=[e,f]$, lie in its derived algebra. The inverse Cartan matrix on each remaining component has nonnegative entries, so the coefficients of the dominant $d$ in that component's simple coroots are nonnegative. A deleted node then takes a nonpositive value on $d$, a contradiction.

For each of the $3^7$ label vectors $a$, let $c_j$ be the multiplicity of weight $j$ in the adjoint root character, including seven Cartan zero weights, and put
\[
 m_j=c_j-c_{j+2}\quad(j\geq0),\qquad
 p=\sum_{j\geq0}m_j,\qquad q=(p+7)/2.
\]
For a triple, the $m_j$ are its highest-weight multiplicities. They are nonnegative, and $m_j$ is even for odd $j$, since the adjoint form is symmetric and an odd-highest-weight irreducible has alternating invariant form. These tests leave $99$ candidates. Equation~\eqref{eq:kernel-difference} gives $q=\dim\fp^e$, because $\dim\fp-\dim\fh=70-63=7$. Distinguishedness requires at least $q$ positive highest weights. This leaves $79$ candidates. If $\nu_1\leq\cdots\leq\nu_q$ are the $q$ smallest positive highest weights, counted with multiplicity, then
\begin{equation}\label{eq:E7-bound}
 \delta_\fp(e)\geq b(a):=\sum_{i=1}^q(\nu_i+2)-70.
\end{equation}

The finite calculation is recorded below, grouped by $q$. The middle column counts the necessary candidates; the last is the minimum of $b(a)$ in that group.
\begin{center}
\small
\begin{tabular}{r|r|r@{\qquad}r|r|r@{\qquad}r|r|r}
 $q$&number&$\min b$&$q$&number&$\min b$&$q$&number&$\min b$\\ \hline
 7&1&70 &15&6&20 &23&3&24\\
 8&2&44 &16&6&16 &24&2&28\\
 9&3&30 &17&3&22 &25&1&24\\
 10&3&22&18&5&15 &28&3&42\\
 11&6&26&19&4&14 &29&2&30\\
 12&2&14&20&3&20 &32&1&42\\
 13&7&34&21&2&20 &35&1&50\\
 14&9&10&22&3&26 &38&1&54
\end{tabular}
\end{center}
The calculation uses the Cartan matrix $C$ with diagonal entries $2$ and entries $-1$ on the six listed edges. Start with the seven unit vectors and close under
\[
 r\longmapsto r-\left(\sum_j C_{ij}r_j\right)e_i
 \qquad(1\leq i\leq7).
\]
The resulting set $R$ has $126$ roots in simple-root coordinates. For each $a\in\{0,1,2\}^7$, compute
\[
 c_j=\#\{r\in R:r\cdot a=j\}+7\,\mathbf1_{j=0}.
\]
Discard $a$ if some $m_j<0$, if an odd $j$ has odd $m_j$, or if $\sum_{j>0}m_j<q$. Expanding the multiset with $m_j$ copies of each positive $j$ and summing its first $q$ entries gives exactly the displayed table. Every distinguished characteristic satisfies these necessary conditions, so the table gives $\delta_\fp(e)\geq10>0$.

%% file: extension_obstructions.tex
\section{Local extension obstructions and a distinguished CII orbit}\label{app:obstructions}

\subsection{The two extension steps}\label{app:extension-steps}
The comparisons in this appendix concern the arguments of
\cite{LXYv1}. The following calculation gives the local extension
obstruction in the DIII and even Spin constructions of that version.
The revised classical proofs use Gaussian test functions and moment
quotients instead \cite[\S\S3--4]{LXYv2}.

\begin{lemma}\label{lem:axes}
Let $\C^\times$ act on $\C^2$ by $t(b,c)=(tb,t^{-1}c)$, and let $\tau(b,c)=(c,b)$. On $\Omega=\C^2\setminus\{0\}$ the distribution
\[
 T^\circ(f)=
 \int_{\C^\times}f(z,0)\,d^\times z
 -
 \int_{\C^\times}f(0,z)\,d^\times z,\qquad
 d^\times z=\frac{dr}{r}\frac{d\vartheta}{2\pi},
\]
where $z=re^{i\vartheta}$, is torus-invariant and $\tau$-odd. It has no torus-invariant distributional extension to $\C^2$.
\end{lemma}

\begin{proof}
The integrals define a Schwartz distribution on the Nash open $\Omega$; its test functions are flat at the deleted point and rapidly decreasing at infinity. Invariance and oddness follow by change of variables. Choose a radial $\varphi\in C_c^\infty(\C)$ equal to one near zero and put $F(b,c)=\varphi(b)\varphi(c)$. Let $E_b$ and $E_c$ be the real Euler vector fields in the two
complex coordinates; for example,
\[
(E_bF)(b,c)
=
\left.\frac{d}{dt}\right|_{t=0}F(e^t b,c).
\]
For the radial infinitesimal generator $D=E_b-E_c$, the function $DF$ vanishes near the origin, whereas
\[
 T^\circ(DF)=2\int_0^\infty\varphi'(r)\,dr=-2.
\]
An invariant extension must annihilate $DF$. Since this test vanishes near the origin, a distribution supported there cannot alter the displayed value.
\end{proof}

In the DIII construction of \cite[\S4.3]{LXYv1}, write $V=V_0\oplus P$, with $\dim P=2$ and $\dim V_0=r-2$ even. Choose standard skew matrices $J_0,J_2$ with square $-I$, and take the boundary point
\[
 x_0=(J_0\oplus0,J_0\oplus0),\qquad J_0^2=-I.
\]
In the fibre $BC=\diag(-I_{V_0},0_P)$, the $P$-blocks are
\begin{equation}\label{eq:D-active-axes}
 B_{22}=bJ_2,\qquad C_{22}=cJ_2,\qquad bc=0.
\end{equation}
The group $\GL(P)$ acts by $(b,c)\mapsto(\det(g)b,\det(g)^{-1}c)$. In the semisimple slice $x_0+\fp^{x_0}$, the $V_0$-block with skew coordinate $K$ is
\[
 (J_0+K,J_0-J_0KJ_0).
\]
Since $(J_0+K)(J_0-J_0KJ_0)=-I+2KJ_0-KJ_0KJ_0$,
the fibre equation is $KJ_0(2I-KJ_0)=0$, which forces $K=0$ near the origin. Thus the transverse support is precisely the two axes. The boundary and punctured-axis stabilizers are $\Sp(V_0)\times\GL(P)$ and $\Sp(V_0)\times\Sp(P)$; they are unimodular, and $\GL(P)/\Sp(P)\simeq\C^\times$. Frobenius descent identifies the orbital difference with the distribution in Lemma~\ref{lem:axes}, so the restriction isomorphism asserted in \cite[Proposition 4.13]{LXYv1} fails for that difference.

In the Spin construction of \cite[Proposition 6.26]{LXYv1}, a self-adjoint two-dimensional Gram block in a hyperbolic basis has the form
\[
 S=\lambda I_2+\begin{pmatrix}a&b\\c&a\end{pmatrix},\qquad \lambda\ne0.
\]
Fixing both eigenvalues at $\lambda$ gives $a=0$, $bc=0$. Conjugation by $\diag(t,t^{-1})\in\SO_2$ has weights $2,-2$ on $(b,c)$, and a reflection exchanges the axes. The invertible Gram map is a submersion, so descent gives the same obstruction, with Euler value $-4$.

These extension steps in \cite{LXYv1} invoke the finite-dimensional criterion in the published \cite[Theorem 1.15]{CS}, whose proof uses Proposition 7.11 there. The corrected version of 11 September 2026 introduces non-occurrence hypotheses involving the Levi and unipotent-radical Schwartz homology in the specified total degree \cite[Proposition 7.11]{CScorr}, and also in the next degree for an isomorphism \cite[Theorem 1.15]{CScorr}. For a degree-zero isomorphism, the degree-one conditions are new in the corrected Chen--Sun version. The radial calculation above shows that the hypotheses of the corrected isomorphism criterion cannot all hold for this local model.

\subsection{A distinguished orbit in balanced CII}\label{app:CII-orbit}
Proposition 5.12 of \cite{LXYv1} cites \cite[Proposition 3.8]{Bul09}, whose printed statement says that all odd-length rows begin with the same letter. The preceding fixed-length centralizer calculation gives
\[
 \forall\,d\text{ odd},\qquad a_db_d=0,
\]
where $a_d,b_d$ count the rows of length $d$ beginning with the two letters. This fixed-length condition does not imply the global condition:
\[
\begin{aligned}
 &\forall\,d\text{ odd},\quad (a_d=0)\ \text{or}\ (b_d=0)\\
 &\qquad\not\Longrightarrow\quad
 (\forall\,d\text{ odd},\,a_d=0)
 \ \text{or}\ (\forall\,d\text{ odd},\,b_d=0).
\end{aligned}
\]
The allowed letter can depend on the odd length.

For the pair $(\Sp_{12},\Sp_6\times\Sp_6)$, the signed diagram
\[
 5_+^2\,1_-^2
\]
gives a $\fp$-distinguished orbit which is not block-swap-stable. One realization is
$E=(L_5\otimes M)\oplus M'$, where $L_5$ is the five-dimensional irreducible $\mathfrak{sl}_2$-module and $M,M'$ are symplectic planes. In a basis $v_0,\ldots,v_4$ of $L_5$, put
\[
 ev_k=v_{k+1},\qquad dv_k=(2k-4)v_k,\qquad
 f v_k=k(5-k)v_{k-1},\qquad
 Q(v_k,v_l)=(-1)^k\delta_{k+l,4}.
\]
Out-of-range vectors are zero. Give $E$ the symplectic form
$Q\otimes\omega_M\oplus\omega_{M'}$ and the grading
$s(v_k\otimes m)=(-1)^kv_k\otimes m$, $s|_{M'}=-1$.
Its eigenspaces $U,W$ both have dimension six, and $e\in\fp$.

The highest-weight vectors of $\mathfrak{sp}(E)$ under this triple that lie in $\fp$ have $d$-weights
$6,4,4,4,4,2$. This follows from
\[
 \mathfrak{sp}(E)\simeq\Sym^2E,\qquad
 \Sym^2(L_5\otimes M)
 \simeq(\Sym^2L_5\otimes\Sym^2M)
       \oplus(\Lambda^2L_5\otimes\Lambda^2M):
\]
the $\Lambda^2L_5$ summand contributes odd highest vectors of weights $6,2$, and $L_5\otimes M\otimes M'$ contributes four of weight $4$. The highest vectors in $\Sym^2L_5$, of weights $8,4,0$, and those in $\Sym^2M'$ are even. Hence every element of $\fp^e$ strictly raises the finite $d$-grading of $E$ and is nilpotent, proving distinguishedness directly. On the other hand,
\[
 \bigl(\rk(e^2|_U),\rk(e^2|_W)\bigr)=(4,2).
\]
This ordered pair is constant on an $H$-orbit and is reversed by the block swap. Consequently, the block-swap-stability assertion in
\cite[Corollary~5.13]{LXYv1} fails for this orbit.

%% file: homogeneity_obstruction.tex
\section{Complex homogeneity and failure of EVII speciality}
\label{app:homogeneity}

\subsection{The complex homogeneity assertion}
In arXiv:0812.5063v5 of \cite{AG}, Theorem 5.1.7 states the
Archimedean homogeneity theorem, and Definition 5.1.6 specifies
adapted distributions. Over \(\mathbb C\), if \(V\) has complex dimension \(n\), the two
asserted scalar characters are
\[
 |t|_{\C}^{n/2},\qquad t|t|_{\C}^{n/2},
\]
where \(\rho(t)v=t^{-1}v\).
The first version \cite[Lemma~2.2 and Sections~7.2--7.6]{LXYv1}
uses this two-character assertion in its EVII speciality argument.
The example below concerns the formulation in arXiv:0812.5063v5
of \cite{AG}. Version~6 enlarges the list of complex adapted
characters \cite[Definition~5.1.6]{AGcorr}; the example is compatible
with that revised definition. The revised LXY argument also retains
holomorphic and antiholomorphic characters
\cite[Lemma~5.2]{LXYv2}. The real homogeneity theorem used in
Corollary~\ref{cor:nice-component} is unaffected by the example.

\subsection{A hyperbolic plane counterexample}
Let $V=\C^2$, with
\[
 B((x,y),(u,v))=xv+yu,\qquad q(x,y)=B((x,y),(x,y))=2xy.
\]
We write $\mathcal F_B$ for the Fourier transform with kernel
$\exp(2\pi i\operatorname{Re}(xv+yu))$ and self-dual measure
$dA(x)dA(y)$, where $dA$ denotes ordinary area measure on $\C$. Put
\[
 T(f)=\int_{\C}\bar y f(0,y)\,dA(y),\qquad
 \nu=\partial_{\bar x}\delta_0(x)\otimes dA(y),
 \qquad \partial_{\bar x}=\tfrac12(\partial_{\Re x}
                                      +i\partial_{\Im x}).
\]
These are nonzero tempered distributions supported on $x=0$.
With distributional derivatives acting by minus differentiation
of the test function, one has exactly
\begin{equation}\label{eq:C-hyperbolic-Fourier}
 \mathcal F_BT=\frac1{\pi i}\nu=-\frac i\pi\nu,
 \qquad \mathcal F_B\nu=-\pi iT.
\end{equation}
Indeed,
$\partial_{\bar u}\exp(2\pi i\operatorname{Re}(yu))
=\pi i\bar y\exp(2\pi i\operatorname{Re}(yu))$ and
the Fourier transform of $dA(y)$ is $\delta_0(u)$.
For the second identity, differentiating in $\bar x$ contributes
$-\pi i\bar v$ and integration in $y$ again contributes
$\delta_0(u)$. These calculations also give
$\mathcal F_B^2=-1$ on the span of $T,\nu$, consistently with
push-forward by $-1$.

Thus $\mathcal L=\C T\oplus\C\nu$ is Fourier-stable and both
it and its Fourier transform are supported on $q=0$.
Multiplication by $q$ kills both vectors. Even if one also requires
stability under $\bar q$, it holds: $\bar q T=0$ and
$\bar q\nu=-2T$. Direct change of variables gives
\begin{equation}\label{eq:C-hyperbolic-characters}
 \rho(t)T=\bar t|t|_{\C}T,
 \qquad \rho(t)\nu=t\nu.
\end{equation}
The two characters are distinct, and neither is $|t|_{\C}$ or
$t|t|_{\C}$. Hence no nonzero vector of $\mathcal L$ has either
of the complex adapted characters asserted in version~5. This
contradicts Theorem~5.1.7 and its restatement as Theorem~C.0.1
in arXiv:0812.5063v5 of \cite{AG}, with the auxiliary manifold
a point. The character $\bar t|t|_{\C}$ is permitted by the
revised definition in \cite{AGcorr}.

The source of the distinction is radial versus angular homogeneity.
The real $\mathfrak{sl}_2$ mechanism controls positive real scaling;
it does not fix the $U(1)$ character. For $r>0$, $T$ has character
$r^3$, which is an allowed adapted character for the real
four-dimensional space. The complexified real quadratic action has
two commuting $\mathfrak{sl}_2$ factors, and this example is
$V_0\otimes {V_1}$. In Appendix~C of arXiv:0812.5063v5 of
\cite{AG}, Fact~(1)(ii) asserts vanishing of the imaginary root
operators. This is incompatible with the complex Weil action,
since the bracket of the two imaginary root elements is minus
the real diagonal element. That assertion is removed in
\cite{AGcorr}.
The two commuting factors used in Section~\ref{sec:EVII} retain
the angular information.

\subsection{An invariant EVII distribution with nilpotent Fourier support}
We construct a nonzero invariant tempered distribution on an EVII
nilpotent orbit closure and prove that its Fourier transform is
supported on the same closure. This proves failure of speciality.
Use the notation of Section~\ref{sec:EVII}, and take
$P_2=\tr D_x$, $P_4=\tr D_x^2$, with the cubic norms $N_+,N_-$.
Their product is a degree-six basic invariant, after a weighted
triangular change of quotient coordinates. Put
\[
 Z=V(P_2,P_4,N_+,N_-)\subset\fp.
\]

\begin{proposition}\label{prop:C-global-measure}
The scheme $Z$ is a reduced irreducible complete intersection of
dimension $50$, with underlying variety $\overline{\mathcal O_{20}}$.
Fix a nonzero constant holomorphic volume form $\Omega$ on $\fp$.
The residue form on $Z_{\rm sm}$, specified by
\[
 dP_2\wedge dP_4\wedge dN_+\wedge dN_-\wedge\omega
       =\Omega\quad\hbox{along }Z,
\]
defines a nonzero $L$-invariant tempered distribution
\begin{equation}\label{eq:C-global-measure}
 u(\phi)=\int_{Z_{\rm sm}}\phi\,|\omega|^2.
\end{equation}
It has support $Z$ and character
$\rho(z)u=z^{42}\bar z^{42}u$.
\end{proposition}
\begin{proof}
The four equations imply the three quotient equations, so $Z$ lies
in the nilpotent cone. By \eqref{eq:EVII-slice-invariants} their
slice ideal at orbit $20$ is $(a,b,u,v)$. They exclude the two
regular orbits $21,22$, on each of which one cubic is nonzero.
The classification \cite[Theorem 3 and Table XIII]{Dok88} has no
other orbit of dimension at least $50$. For clarity, its orbit
dimensions are obtained from the normal dimensions in the table in
Section~\ref{sec:EVII}: they are
\[
 17,17,26,26,26,27,27,27,27,33,38,38,42,42,42,43,43,43,43,50,51,51.
\]
Every component of a four-equation locus in dimension $54$ has
dimension at least $50$, and a component of a finite-orbit variety
has a dense orbit. Therefore $Z$ has just the component
$\overline{\mathcal O_{20}}$. The ideal has height four and is
generated by four elements, so it is a complete intersection and
has no embedded associated primes. At a point \(e\in\mathcal O_{20}\), use the local action chart
\[
\Phi:\mathfrak m\times(e+\mathfrak p^f)\longrightarrow\mathfrak p,
\qquad
(X,e+s)\longmapsto \exp(\operatorname{ad}X)(e+s).
\]
Since \(Z\) is \(L\)-stable, the pullback of its defining ideal is
the extension of the slice ideal \((a,b,u,v)\) to the local action
coordinates. Thus the pullback of \(Z\) is reduced along
\(\{s=0\}\).  As \(\Phi\) is a local submersion (indeed a local
coordinate chart after choosing \(\mathfrak m\)), reducedness descends
to \(Z\) at the generic point of \(\mathcal O_{20}\).  Hence \(Z\) is
generically reduced.  Since \(Z\) is an unmixed complete intersection
with a single irreducible component, it is reduced.

The determinant character of $L$ on $\fp$ is trivial. The two
cubic relative characters are inverse, so the residue form and its
positive density are $L$-invariant. Ordinary dilation
$s_r:x\mapsto rx$ satisfies
\begin{equation}\label{eq:C-residue-scaling}
 s_r^*\omega=r^{54-(2+4+3+3)}\omega=r^{42}\omega.
\end{equation}
We now prove local finiteness of the smooth-locus integral.

At $e\in Z$ choose a normal triple and the slice $S_e=e+\fp^f$.
Its positive Kazhdan weights have sum $108-T(e)$, as computed in
\ref{subsec:EVII-quadratic}. Under $\kappa_t=t^2\Ad(D_t^{-1})$ the
four restricted equations have total weight $24$. The individual
cubic weights may differ from $6$, but their shifts under $D_t$
cancel. Hence the slice residue form $\sigma_e$ satisfies
\[
 \kappa_t^*\sigma_e=t^{84-T(e)}\sigma_e.
\]
In local action coordinates a complement to $\fh^e$ gives a
smooth orbit-coordinate factor. The ambient residue form is a
nonvanishing analytic factor times the product of a smooth volume
form there and $\sigma_e$. On compact coordinate boxes the factor
is bounded above and below. Tonelli's theorem, first for extended
nonnegative integrals, reduces local finiteness to the slice.

Induct down the dimensions of the finitely many orbits in $Z$.
The measure is smooth along the open orbit $20$. Every strict
boundary orbit has $T(e)\leq60$, including $T(0)=0$, by the trace
table. A sufficiently small slice meets no other orbit of dimension
at most $\dim Le$. Indeed, there are finitely many such orbits;
their closures do not contain $e$, and contraction of a slice point
to $e$ preserves its nilpotent orbit. A compact weighted annulus
in that slice therefore has finite mass by induction and a finite
cover. Choose a radius satisfying $r(\kappa_t s)=t r(s)$, for
example $r(s)=(\sum|s_i|^{2/w_i})^{1/2}$. Dyadic contractions of
the annulus have total mass bounded by a constant times
\[
 \sum_{k\geq0}2^{-2k(84-T(e))}<\infty.
\]
This proves local finiteness at every boundary orbit, including zero.

The positive measure is nonzero on the dense orbit, so its support
is $Z$. Equation~\eqref{eq:C-residue-scaling} gives
$u(B_R)=R^{84}u(B_1)$ on Euclidean balls. Summation over ordinary
dyadic shells proves continuity on Schwartz functions, using a
weighted supremum seminorm of order greater than $84$. Thus it
is tempered, and a change of variables in convergent integrals
gives the stated character and invariance.
\end{proof}

\begin{lemma}\label{lem:C-quadratic-annihilation}
For the holomorphic Laplacian of the inverse form, one has
$\Delta^{16}u=\bar\Delta^{16}u=0$ globally on $\fp$.
\end{lemma}
\begin{proof}
We first calculate on orbit $20$. Identify distributions with scalar
generalized functions using the invariant ambient density. Pullback
under the local action submersion is then constant in the group
variable. The restriction of the orbital measure is of order zero
and is supported at the origin of the normal slice, so its transverse
scalar germ is $c\delta_0$ for some $c\ne0$. The inverse-coordinate
radial operator acts on this germ. Since $\Delta$ preserves invariance,
successive applications are computed by iterating this radial operator.

\input{evii_twenty_record}

Apply the inverse-coordinate formula \eqref{eq:EVII-radial} with
$50$ orbit-coordinate columns and four normal columns. It gives
$\mathcal R_{20}=\sum g_{ij}\partial_i\partial_j+\sum h_i\partial_i$,
where $g$ is symmetric and its upper-triangular entries are
\begin{align*}
 g_{00}&=\tfrac{14}{15}s_0s_3-\tfrac32s_1s_2-\tfrac{64}{15}s_3^3,
 &g_{01}&=-\tfrac{26}{15}s_1s_3,\\
 g_{02}&=-\tfrac{26}{15}s_2s_3,
 &g_{03}&=-\tfrac3{10}s_0,\\
 g_{11}&=g_{22}=0,
 &g_{12}&=-\tfrac12s_0-\tfrac{16}{9}s_3^2,\\
 g_{13}&=-\tfrac9{40}s_1,
 &g_{23}&=-\tfrac9{40}s_2,\\
 g_{33}&=-\tfrac3{20}s_3,
 &h&=(-\tfrac{52}{15}s_3,0,0,-\tfrac{81}{20}).
\end{align*}
Here $q_0(s)=B(e'+s,e'+s)=-80s_3/3$;
$g\,dq_0=(8s_0,6s_1,6s_2,4s_3)^t$ and
$\mathcal R_{20}q_0=108$. Using
\[
 s^\gamma\partial^\alpha\delta_0
 =(-1)^{|\gamma|}\frac{\alpha!}{(\alpha-\gamma)!}
       \partial^{\alpha-\gamma}\delta_0\quad(\gamma\leq\alpha),
\]
and zero otherwise, successive applications give
$\mathcal R_{20}^{16}\delta_0=0$. To record the finite recurrence
completely, start with $F_0=1$ and use
\[
 F_{j+1}=\sum g_{ik}(-\partial_x)(x_ix_kF_j)
              +\sum h_i(-\partial_x)(x_iF_j).
\]
The numbers of nonzero monomials in $F_1,\ldots,F_{16}$ are
\[
 1,2,3,4,5,7,8,10,11,13,14,16,17,19,20,0.
\]
The recurrence has rational coefficients. Complex conjugation gives
the second local identity.

Consequently $w=\Delta^{16}u$ is supported on the strict boundary
of $Z$. It is invariant and tempered. A holomorphic differential
operator changes only the holomorphic scalar exponent, so its
normalized antiholomorphic exponent remains $15$, since
$z^{42}\bar z^{42}=z^{15}\bar z^{15}|z|_{\C}^{27}$.
At a maximal boundary orbit, the same slice-character calculation
as \eqref{eq:EVII-degree} forces an antiholomorphic normal degree
\[
 A'=T(e)-54-2\cdot15=T(e)-84\leq-24.
\]
This follows directly from the scalar character and the slice
transformation law. A finite-order germ on
the orbit has nonnegative normal degrees, so no such maximal orbit
can occur. The finite orbit filtration gives $w=0$. The conjugate
argument uses the unchanged holomorphic exponent and gives
$\bar\Delta^{16}u=0$.
\end{proof}

\begin{proposition}\label{prop:C-not-special}
The EVII isotropy representation is not special in the sense of
\cite[Definition 6.0.6]{AG}.
\end{proposition}
\begin{proof}
Put $b=\operatorname{Re}B$, $Q(x)=b(x,x)$, and let $\mathcal F$
be the positive real Fourier transform for $b$. Its inverse-form
differential operator is $\Delta_b=2(\Delta+\bar\Delta)$, and
\[
 D=\mathcal F^{-1}Q\mathcal F=-\frac1{4\pi^2}\Delta_b.
\]
Lemma~\ref{lem:C-quadratic-annihilation} and the binomial formula
give $D^{31}u=0$. Also $Qu=0$, because $u$ is an order-zero
measure on $Z\subset\mathcal N$.

Let \(C\) denote multiplication by \(e^{\pi iQ}\), and put $S=\mathcal F^{-1}C\mathcal F$.
The real signature of \(b\) is \((54,54)\), so the Gaussian Fourier
phase is \(1\). The kernel identity for \(S\) is understood
distributionally: after inserting a Gaussian damping factor
\(e^{-\pi\varepsilon\|z\|^2}\), \(\varepsilon>0\), the relevant
integrals are absolutely convergent and the usual Gaussian
completion-of-squares computation gives the kernel $e^{-\pi iQ(x-y)}$. Letting \(\varepsilon\downarrow0\) in \(\mathcal S'\) yields the
undamped identity. Since multiplication by \(e^{\pi iQ}\) is a
continuous Schwartz multiplier, multiplication by the two chirps
therefore gives $CSC=\mathcal F$
on \(\mathcal S(\mathfrak p)\), and hence by continuity on
\(\mathcal S'(\mathfrak p)\). Since \(Q u=0\), we have \(Cu=u\).  Moreover,
\[
S
=
\mathcal F^{-1}e^{\pi iQ}\mathcal F
=
e^{\pi iD},
\qquad
D=\mathcal F^{-1}Q\mathcal F.
\]
As \(D^{31}u=0\),
\[
Su
=
\sum_{j=0}^{30}\frac{(\pi i)^j}{j!}D^ju.
\]
Applying \(CSC=\mathcal F\) to \(u\) therefore gives
\begin{equation}\label{eq:C-actual-Fourier}
 \mathcal F u=e^{\pi iQ}\sum_{j=0}^{30}\frac{(\pi i)^j}{j!}D^ju.
\end{equation}
To justify the truncation distributionally, write
\[
e^{\pi i w}
-
\sum_{j=0}^{30}\frac{(\pi i w)^j}{j!}
=
w^{31}g(w),
\]
where \(g\) extends to an entire function whose derivatives are
bounded on \(\mathbb R\).  Since \(Q\) is quadratic, \(g\circ Q\)
and all of its derivatives have polynomial growth; hence $g\circ Q\in\mathcal O_M(\mathfrak p)$
and multiplication by \(g(Q)\) is a continuous Schwartz multiplier.
Now
\[
D^{31}u=0
\quad\Longleftrightarrow\quad
Q^{31}\mathcal Fu=0,
\]
so
\[
Q^{31}g(Q)\mathcal Fu=0.
\]
Thus the exponential may indeed be replaced, on \(\mathcal Fu\), by
its Taylor polynomial of degree \(30\), giving \eqref{eq:C-actual-Fourier} as an identity
of tempered distributions.
Every differential term on the right is supported on $Z$, and
multiplication by the chirp does not enlarge support. Thus
$\supp\mathcal F u\subseteq Z\subseteq\mathcal N$.
This proves actual Fourier compatibility for the nonzero
distribution \eqref{eq:C-global-measure}.

Finally, the two $L$-irreducible summands $J_{27},J_{27}^*$ have
opposite nontrivial central characters. Each is totally isotropic,
and an invariant nondegenerate symmetric form can exist only on
their sum. Such a form is a nonzero scalar multiple of $B$.
There is therefore no proper nonzero invariant orthogonal summand
to test in the definition of speciality. Replacing $B$ by a scalar
multiple rescales the Fourier variable, which preserves the cone
$Z$. 
Thus, since $\mathfrak p^L=0$, speciality
of the EVII isotropy representation is
equivalent to
\[
\left\{
T\in\mathcal S^*(\mathfrak p)^L:
\operatorname{supp}T\subseteq\mathcal N,\quad
\operatorname{supp}\mathcal F_BT
\subseteq\mathcal N
\right\}=\{0\}.
\]
The nonzero distribution $u$ constructed above
violates this vanishing, proving that the EVII
isotropy representation is not special.
\end{proof}
\begin{remark}
The counterexample above is in fact invariant
under the full effective multiplier group $J$.
Indeed, $\tau$ preserves $Z$. Since \(\tau\) fixes the
Cartan subspace pointwise, Chevalley restriction gives $\tau^*P_2=P_2$ and $\tau^*P_4=P_4$. Moreover, \(\tau\) exchanges \(J_{27}\) and \(J_{27}^*\); by uniqueness
of the invariant cubic on each summand, $\tau^*N_+\in\mathbb C^\times N_-$ and $\tau^*N_-\in\mathbb C^\times N_+$. Hence the ideal $(P_2,P_4,N_+,N_-)$
is \(\tau\)-stable, and so \(\tau(Z)=Z\). On the dense orbit \(\mathcal O_{20}\), the form
\((\tau^{-1})^*\omega\) is again \(L\)-invariant, since \(\tau\)
normalizes \(L\).  As \(L\) acts transitively on
\(\mathcal O_{20}\), there is a scalar \(\lambda\in\mathbb C^\times\)
such that $(\tau^{-1})^*\omega=\lambda\omega$.
Consequently, $\tau_*u=|\lambda|^2u$.
Since \(\tau^2\in L\) and \(u\) is \(L\)-invariant, $u=\tau_*^2u=|\lambda|^4u$.
Thus \(|\lambda|=1\), and therefore $\tau_*u=u$. Hence \(u\) is \(J\)-invariant, and the same obstruction applies to
every full fixed image between \(L\) and \(J\). The distribution $u$
belongs to the trivial character space; Theorem~\ref{thm:EVII}
concerns the nontrivial character $\chi$.
\end{remark}

%% file: evii_twenty_record.tex
For this radial calculation use the following representative and
ordered normal basis in the root convention of Section~\ref{sec:EVII-coordinates}:
\begin{align*}
e'={}&-v_{56}-v_{59}+v_{62}-v_{63}+v_{64}+v_{96}+v_{101}\\
&+v_{103}-v_{106}-v_{111},\\
d'={}&6v_{0}+8v_{1}+10v_{2}+14v_{3}+10v_{4}+6v_{5},\\
f'={}&-\tfrac{3}{2}v_{33}-\tfrac{3}{2}v_{38}-4v_{40}+\tfrac{3}{2}v_{43}+\tfrac{3}{2}v_{48}+\tfrac{3}{2}v_{119}+\tfrac{3}{2}v_{122}\\
&-\tfrac{3}{2}v_{125}+4v_{126}-\tfrac{3}{2}v_{127},\\
n_0\,={}&v_{7}+v_{132},\\
n_1\,={}&v_{14}+v_{20}-v_{26}+v_{32},\\
n_2\,={}&-v_{128}+v_{129}+v_{130}+v_{131},\\
n_3\,={}&v_{33}+v_{38}+\tfrac{8}{3}v_{40}-v_{43}-v_{48}-v_{119}-v_{122}\\
&+v_{125}-\tfrac{8}{3}v_{126}+v_{127}.
\end{align*}
Let
\[
L^d:=Z_L(d),\qquad
\mathfrak p_2(d):=\{x\in\mathfrak p:[d,x]=2x\}.
\]
Since \(d'=d\), and both \(e'\) and the representative \(e\) of
Lemma~\ref{lem:EVII-slice20} lie in the unique open \(L^{d,\circ}\)-orbit in
\(\mathfrak p_2(d)\), they are \(L\)-conjugate. Hence \(e'\) also
represents orbit \(20\). For the following radial calculation, write $s=\sum_{i=0}^3 s_i n_i$
with coordinate order \((b,u,v,a)\).